\documentclass[12pt]{amsart}

\usepackage{verbatim}
\usepackage{framed}
\usepackage{amsfonts}
\usepackage{amsthm,amsmath}
\usepackage[integrals]{wasysym}
\usepackage{enumerate}
\usepackage{fancyhdr}
\usepackage{amssymb}
\usepackage{chemarrow}

\usepackage{mathtools}
\usepackage{hyperref}
\usepackage{appendix}

\usepackage{enumitem}

\usepackage{epsfig,epic,eepic,graphicx}

\usepackage[usenames,dvipsnames,svgnames]{xcolor}
\usepackage{tikz-cd}
\usepackage{mathtools}

\usetikzlibrary{arrows,arrows.meta}

\numberwithin{equation}{section}
\let\eps=\varepsilon
\let\wt=\widetilde

\def\N{{\mathbb N}}
\def\R{{\mathbb R}}
\def\T{{\mathbb T}}

\def\ep{\varepsilon}
\def\and{\quad{\rm and}\quad}

\def\virgp{\raise 2pt\hbox{,}}
\def\cdotpv{\raise 2pt\hbox{;}}

\def\div{ \hbox{\rm div}\,  }

\def\dr{\delta\!\rho}
\def\du{\delta\!u}

\newtheorem{thma}{Theorem}

\newtheorem*{thmstar}{Theorem}
\numberwithin{equation}{section}
\newtheorem{thm}{Theorem}[section]
\newtheorem{lem}[thm]{Lemma}
\newtheorem{defi}[thm]{Definition}
\newtheorem{prop}[thm]{Proposition}

\newtheorem{rmk}[thm]{Remark}

\newcommand{\ben}{\begin{eqnarray}}
\newcommand{\een}{\end{eqnarray}}
\newcommand{\beno}{\begin{eqnarray*}}
\newcommand{\eeno}{\end{eqnarray*}}
\definecolor{darkgreen}{rgb}{0,0.5,0}
\definecolor{darkblue}{rgb}{0,0,0.7}
\definecolor{darkred}{rgb}{0.9,0.1,0.1}
\definecolor{lightblue}{rgb}{0,0.51,1}

\begin{document}
\title[]{Free boundary regularity of vacuum states for  incompressible viscous flows in unbounded domains}

\author[C. Prange]{Christophe Prange}
\address[C. Prange]{\newline  Laboratoire de Math\'{e}matiques AGM, UMR CNRS 8088, Cergy Paris Universit\'{e}, 2 Avenue Adolphe Chauvin, 95302     Cergy-Pontoise Cedex, France}
\email{christophe.prange@cyu.fr}

 \author[J. Tan]{Jin Tan}
\address[J. Tan]{\newline Laboratoire de Math\'{e}matiques AGM, UMR CNRS 8088, Cergy Paris Universit\'{e}, 2 Avenue Adolphe Chauvin, 95302     Cergy-Pontoise Cedex, France}
\email{jin.tan@cyu.fr}

\date{\today}

\begin{abstract}
In the well-known book of Lions [{\em Mathematical topics in fluid
mechanics. Incompressible models}, 1996],  global existence results of finite energy weak solutions of the inhomogeneous incompressible  Navier-Stokes equations (INS) were proved without assuming positive lower bounds on the initial density, hence allowing for vacuum. Uniqueness, regularity and persistence of boundary re\-gularity of density patches were listed as open problems. A breakthrough on Lions' problems was   recently made by  Danchin and   Mucha  [The incompressible Navier-Stokes equations in vacuum, {\em Comm. Pure Appl. Math.},
72   (2019), 1351--1385]  in the case where the fluid domain is either bounded or the torus. However, the  case of unbounded domains was left open because  of the lack of Poincar\'{e}-type inequalities.  
 In this paper,
we obtain regularity and uniqueness  
of   Lions' weak solutions  for (INS) with \emph{only bounded and nonnegative initial density} and additional
regularity only assumed for the initial velocity,  in the whole-space case $\R^d$, $d=2$ or $3$. 
In particular, our  result allows us to study the evolution of a vacuum bubble embedded in an incompressible fluid, as well as  a patch of a homogeneous fluid embedded in the vacuum,
 which  provides  an answer to Lions' question   in the whole-space case.  
\end{abstract}
\maketitle

\section{Introduction}

In the present paper, we are concerned with the Inhomogeneous incompressible Navier-Stokes equations (INS) in the whole-space $\R^d$ (with $d=2,3$)
\begin{equation}\label{INS}
\left\{\begin{aligned}
  &\partial_t \rho+\nabla\cdot(\rho u) =0, \\
& \partial_t (\rho u) +\nabla\cdot(\rho u\otimes u)+\nabla P =\nu\Delta u,\\
&\nabla\cdot u=0.
 \end{aligned}\right.\tag{INS}
\end{equation}
in the presence of vacuum. 
The unknowns are  the velocity field $u=u(t,x),$  the  density $\rho=\rho(t,x)$ and  the pressure $P=P(t,x)$ and $\nu>0$ is the viscosity constant.

The initial values are prescribed as follows:
\begin{equation}
\rho|_{t=0}=\rho_0, \quad \rho u|_{t=0}=m_0.
\end{equation}
It will be assumed that the initial density $\rho_0$ is nonnegative and bounded in $L^\infty(\R^d)$, but not necessarily bounded from below by a positive constant. This allows us to study, for instance, the evolution of a vacuum bubble embedded in an incompressible fluid, as well as of a patch of a homogeneous fluid embedded in the vacuum. Let us mention that this problem may be reformulated as a  free-boundary problem. Lions in \cite[page 34]{PLL}  raised the following question about the persistence of the interface regularity through the  evolution:
\begin{quote}
[In the case of a density patch $\rho=\mathbf{1}_{D(t)}$, the system] \eqref{INS} can be reformulated as a somewhat complicated free boundary problem. It is also very natural to ask whether the regularity of $D$ is preserved by the time evolution.
\end{quote}
The question asked by Lions was solved by Danchin and Mucha in \cite{DM19} in the case of bounded domains or the Torus, relying in particular on Poincar\'e-type inequalities. Hence, these authors mention the whole-space case as an open problem:
\begin{quote}
[T]he generalization to unbounded domains (even the whole space) within our
approach [is] unclear as regards global-in-time results.
\end{quote}
One of the main motivations in this paper is to address that question. We give a positive answer, see Theorem \ref{C1} below.

\subsection{A brief state of the art}

\paragraph{\underline{Global-in-time finite-energy weak solutions}} 
Since  the pioneering works by  Leray \cite{Leray} and   Ladyzhenskaya  \cite{Lady} on the incompressible homogeneous Navier-Stokes equations,  the existence and uniqueness  issues of   solutions   for the inhomogeneous system \eqref{INS}  has been intensively investigated.   Finite-energy weak solutions in the spirit of `Leray solutions' were built  first by  Kazhikov et al.  \cite{Ka74, AKM90},  in the case when $
\rho_0$ is bounded away from $0$. This result  was extended by Simon \cite{Si} allowing $\rho_0$ to vanish, in the case of bounded domains.
Then,   Lions \cite{PLL} (see also   Desjardins \cite{BD97, BD97-2}) considered  the  so-called density-dependent Navier-Stokes equations, i.e.
the case where viscosity $\nu$ depends on density $\rho,$
 by using general results on transport equations obtained in  DiPerna and Lions' work \cite{DL}.
As a consequence, in the whole-space $\R^d$, $d=2$ or $3$, it is proved in Lions's book \cite{PLL}, see Theorem \ref{theom-Th-PLL} below, that  global weak solutions such that $\rho$ is bounded and  
$u\to u^\infty \quad{\rm as}~|x|\to\infty,\quad{\rm for~all}~t\geq0$, 
exist provided that
\begin{equation}\label{initialcond}
\left\{\begin{aligned}
&\rho_0\geq 0\quad {\rm a.e.~in}~ \R^d,\quad \rho_0\in L^\infty(\R^d),\\
&m_0\in L^2(\R^d),\quad m_0=0 \quad {\rm a.e.~on}~ \{\rho_0=0\},\\
&|m_0|^2/\rho_0\in L^1(\R^d).
\end{aligned}\right.
\end{equation}
Because of the fact that the momentum equation is degenerate when there is vacuum,  Lions needs to assume in addition to \eqref{initialcond}, one of the following three conditions:\\
\emph{\underline{almost No Vaccum}}
\begin{equation}\label{cond1}
(1/\rho_0)\,\mathbf{1}_{\rho_0<\delta_0}\in L^1(\R^d),\quad{\rm for ~some}~\delta_0>0,\tag{aNV}
\end{equation}
or \emph{\underline{Vacuum Bubble}}
\begin{equation}\label{cond2}
(\underline{\rho}-\rho_0)_+\in L^p(\R^d), \quad{\rm for ~some}~\underline{\rho}\in (0, \infty), \,\, p\in (d/2, \infty),\tag{VB}
\end{equation}
or \emph{\underline{Far-Field Vacuum}}
\begin{align}\label{cond3}
\left\{\begin{aligned}
 &{\rm if}~d=2, \quad \int_{\R^2} \rho_0^p\,\langle x\rangle^{2(p-1)}(\log\langle x\rangle)^r\,dx<\infty\\
& \hspace*{2cm}\quad {\rm for~some}~p\in (1, +\infty]~ {\rm and}~r>2p-1\\
 &{\rm if}~d=3,\quad \rho_0\in L^{\frac{3}{2}, \infty}(\R^3),
\end{aligned}\right.\tag{FFV-1}
\end{align}
  where $\langle x\rangle:= (e+|x|^2)^{1/2}.$

It is important to note that above three conditions allow several  physically interesting cases. For example,  condition \eqref{cond1}  allows the density to vanish on zero-measure sets,  condition \eqref{cond2}  allows vacuum bubbles,   while condition \eqref{cond3}  allows  density patches.

Let us give a precise definition of a weak solution to the Cauchy problem associated to system \eqref{INS}.  The following definition and  result are stated in the book \cite{PLL} of Lions.

\begin{defi}[{\cite[Chapter 2]{PLL}}]\label{defweaksolu}
Let $T>0$. We say that $(\rho,  {u})$ is a finite-energy weak solution  of  system \eqref{INS} with the initial conditions \eqref{initialcond}, if $(\rho,  {u})$ satisfies the following properties:
\begin{align*}
&\rho\geq0,\quad\rho\in L^\infty((0, T)\times \R^d), \quad\rho\in \mathcal{C}([0, T]; L^p_{\rm loc}(\R^d)) ~~\,\,{\rm for~all~}1\leq p<\infty,\\
&\rho|{u}|^2\in L^\infty(0, T; L^1({\R^d})), \and\\
&u\in L^2(0, T;\wt{\mathcal D}^{1, 2}{(\R^2)})~{\rm if}~d=2\quad{\rm{or}}\quad u\in L^2(0, T; \mathcal{D}^{1, 2}(\R^3)) ~{\rm if}~d=3,
\end{align*}
where the spaces
\begin{align*}
\wt{\mathcal{D}}^{1, 2}(\R^2):=\{z\in H^1_{\rm loc}(\R^2); \nabla z\in L^2(\R^2)\},  \quad  {\mathcal{D}}^{1, 2}(\R^3):=\{z\in L^6(\R^3);  \nabla z\in L^2(\R^3)\},
\end{align*}
are endowed with the norms
$$
\|z\|_{\wt{\mathcal{D}}^{1, 2}(\R^2)}:=\|\nabla z\|_{L^2(\R^2)},\quad \|z\|_{\mathcal{D}^{1, 2}(\R^3)}:=\|z\|_{L^6(\R^3)}+\|\nabla z\|_{L^2(\R^3)}.
$$
If \eqref{cond1} or \eqref{cond2} hold, we require in addition that $u\in L^2((0, T)\times \R^d).$
Moreover, $(\rho, u)$ satisfies system \eqref{INS} in the sense of distributions in $(0, \infty)\times \R^d$ and  the following energy inequality  for a.e. $t\in(0, T)$
\begin{equation}\label{enineq}
\int_{\R^d}\rho(t, x)|u(t, x)|^2\,dx+2\nu\int_0^t\int_{\R^d}|\nabla u(s, x)|^2\,dxds\leq \int_{\R^d}\frac{|m_0(x)|^2}{\rho_0(x)}\,dx.
\end{equation}
The solution $(\rho,  {u})$ is a \emph{global-in-time} finite-energy weak solution  of  system \eqref{INS} if the properties stated above hold for all $T\in(0,\infty)$.
\end{defi}

\begin{thmstar}[{\cite[Theorem 2.1]{PLL}}]\label{theom-Th-PLL}
Assume that the initial data $(\rho_0, u_0)$ satisfies condition \eqref{initialcond} and one of the conditions \eqref{cond1}-\eqref{cond3}, then there exists a global weak solution $(\rho, u)$ of system \eqref{INS} in the sense of Definition \ref{defweaksolu}. Furthermore, one has for all $0\leq \alpha_0\leq \beta_0<\infty$
\begin{align}\label{meas-density}
{\rm meas}\{x\in \R^d\,|\,\,\alpha_0\leq\rho(t, x)\leq \beta_0\}~\,~{\rm is~independent~of}~t\geq0.
\end{align}
And if $\rho_0-\rho^\infty\in L^p(\R^d)$ for some $1\leq p<\infty,\, \rho^\infty\in [0, \infty),$ then $\rho-\rho^\infty\in \mathcal{C}([0, \infty); L^p(\R^d))$ and $\|\rho-\rho^\infty\|_{L^p(\R^d)}$ is independent of $t\geq0.$ Also, if $\rho\equiv \bar\rho$ for some $\bar\rho\in[0, \infty)$, then $\rho_0\equiv\bar\rho$ on $[0, \infty)\times \R^d.$
\end{thmstar}

The uniqueness of global weak solutions for system \eqref{INS} remains an open question   even in the 2D case, although several weak-strong uniqueness results were established e.g.  in Lions \cite{PLL} and Germain \cite{Germain}.

\medskip

\paragraph{\underline{Uniqueness in the absence of vacuum}} 
Let us recall some recent developments on the unique solvability of the inhomogeneous incompressible Navier-Stokes equations \eqref{INS} in the absence of vacuum.   Since the density is bounded below away from zero, this case is rather close to the homogeneous flows (i.e.  density is constant).  Strong solutions were first considered by Ladyzhenskaya and Solonnikov \cite{LS} in the bounded domain case, whenever initial velocity and density are smooth enough and away from vacuum. 

After these early works, a number of papers were devoted to the study of strong solutions to system \eqref{INS}, with particular interest in classes of initial data generating regular unique solutions. Here, an important  feature for system \eqref{INS} is the scaling invariance: if  $(\rho, u,  P)$ is a solution associated to the initial data $(\rho_0, u_0)$, then $(\rho, \lambda u,   \lambda^2 P)(\lambda^2 t, \lambda x)$ is a solution  associated to  $(\rho_0, \lambda u_0)(\lambda x),$ for all $\lambda>0.$
In the critical\footnote{A functional space for the data $(\rho_0, u_0)$  or for the solution $(\rho, u)$  is said to be critical if its norm is invariant under the natural scaling of \eqref{INS}.}  regularity framework, the local and global existence results were first obtained by  Danchin  \cite{Dan1}  in the case when initial  density  has small variation, but still not including patches of density.
Abidi   \cite{A} and   Abidi and  Paicu   \cite{AP} extended these results to the case with variable viscosity in critical Besov spaces. Then, many efforts focused on removing these smallness assumptions on the density, see for example  \cite{Dan2, Cosmin, AGZ, XLZ, AG}.

For discontinuous densities,   Danchin and  Mucha \cite{DM12} first  proved   well-posedness results for data including initial density patches (that have small variation)  by a Lagrangian approach.    Later,    Paicu, Zhang and  Zhang \cite{PZZ} (see also \cite{DM13}) established  global unique solvability with only bounded initial density, in addition bounded from below;  see also further developments \cite{HPZ,  DPZ, Z, DW} on initial velocity in  critical functional spaces. These  results enable initial density of the type  $\rho_0=\rho_1\, \mathbf{1}_{\Omega}+ \rho_2\, \mathbf{1}_{\Omega^{c}},$ where  $\rho_1, \rho_2$ are positive constants, $\Omega$ is a bounded domain in $\R^d$ and  $\mathbf{1}_{\Omega}$ is the charac\-teristic function of $\Omega.$
In connection with Lions' question on the persistence of boundary regularity of $\Omega,$ we refer to the works \cite{DZ, GGJ18,  LZ1, LZ2, LZ3, GGJ23, PZ}.

\medskip

\paragraph{\underline{Uniqueness in the presence of vacuum}} 
Concerning strong solutions allowing vacuum,   local  well-posedness was  proved by Choe and   Kim \cite{CK2}  and Cho and Kim \cite{CK}  under compatibility conditions. Later that condition was removed by Li \cite{Li}.  Recently, the work of Craig, Huang and Wang \cite{CHW13}  and L\"{u}, Shi and Zhong \cite{LSZ} established global  strong solutions. All these results allow compactly supported initial densities, but still need to be  smooth enough so that discontinuous initial densities are not allowed. 

A breakthrough on the global unique solvability for  system \eqref{INS} with only bounded and nonnegative initial density was made very recently by  Danchin and   Mucha \cite{DM19} in the case where the fluid domain is either bounded or the torus.  We also mention the stability result  \cite{DMP} of the density patches problem after the paper \cite{DM19}.

\medskip

Finally, let us briefly recall the main ideas from \cite{DM19} for handling vacuum in the two-dimensional case. In order to obtain $H^1(\T^2)$ regularity for the  velocity the authors of \cite{DM19} test the momentum equation by $\partial_t u.$ It appears that the only difficult term is $\|\sqrt{\rho}u\cdot\nabla u\|_{L^2(\T^2)}^2.$  More precisely, if the density contains regions of vacuum one does not have obvious  control of $\|\sqrt{\rho} u\|_{L^p(\T^2)}$ for some $p>2,$  and this also reveals  the  lack of lower-order bound for the velocity.   However, by taking advantage of the following Desjardins interpolation inequality (see \cite{BD97-3, DM19}):
 {\begin{multline}\label{ineq-BD}
\left(\int_{\T^2} \rho |u|^4\,dx\right)^{1/2}  \\
\leq C\|\sqrt{\rho} u\|_{L^2(\T^2)}\,\|\nabla u\|_{L^2(\T^2)}\cdot  \ln^{1/2}\left(e +\frac{\|\rho-M\|_{L^2(\T^2)}^2}{M^2}+\frac{\rho_*\|\nabla u\|_{L^2(\T^2)}^2}{\|\sqrt{\rho}u\|_{L^2(\T^2)}^2}\right),
\end{multline}
which is an improvement of the well-known Ladyzhenskaya inequality, it turns out that  the  $H^1(\T^2)$ norm of the  velocity is bounded globally-in-time by its initial values. Above, $M:=\|\rho_0\|_{L^1(\T^2)}$ represents the total mass. Moreover,  thanks to the  Poincar\'{e} inequality
 \begin{align}\label{in-P}
 \|u-\bar u(t)\|_{L^2(\T^2)}\leq C\|\nabla u\|_{L^2(\T^2)} \quad{\rm with}\quad \bar u(t):=\frac{1}{|\T^2|}\int_{\T^2} u(t, x)\,dx,
 \end{align}
and the conservation laws of mass and momentum, one has $u\in L^\infty(\R_+; L^2(\T^2)).$

Unfortunately, neither   \eqref{ineq-BD} nor  \eqref{in-P} is valid in the whole-space case. This creates significant difficulties to handle   vacuum
in  our setting, especially for the far-field vacuum   in the two-dimensional case.  Moreover,   the roughness of the density    causes  additional difficulties for the uniqueness issue. Indeed,  we can only perform estimates in low regularity spaces for the following degenerate equation 
\begin{multline*}
 \underbrace{\rho(\partial_t\du +u\cdot\nabla \du)}_{\rm degenerate ~when~ there~ is ~vacuum}+\nabla\delta P- \Delta \du=\underbrace{-\dr\,\dot{\bar u}}_{\rm loss~of ~one ~derivative~already}-\rho\du\cdot\nabla \bar u;
\end{multline*}
see more remarks in the Subsection \ref{S1.4}.
%%%%%%%%%%%%%%%%%%%%%%%%%%%%%%%%%%%%%%%%%%%

\subsection{Main results}
The main goal of the present paper is to prove  regularity and uniqueness results of  weak solutions  for system \eqref{INS} with only bounded and nonnegative initial density. 
Hence our results are extensions to the whole-space of the results of Danchin and   Mucha \cite{DM19} (bounded domains or the Torus). 
We assume the additional regularity $\wt{\mathcal{D}}^{1, 2}(\R^2)$ or $\mathcal{D}^{1, 2}(\R^3)$ for the initial velocity.

\medskip

Let us now state our main results.  We first address the  two-dimensional case.

\begin{thma}[existence and uniqueness in 2D]\label{thm2d}
Consider  any initial data $(\rho_0, u_0)$ satisfying  \eqref{initialcond} such that  $\nabla u_0\in L^2(\R^2)$ and $\nabla\cdot u_0=0.$   Assume that for some constant $\rho_*>0,$
 \begin{equation}\label{cond:inidensity}
  0\leq \rho_0\leq\rho_*.
 \end{equation}
 Then, there are two cases.
 
 \smallskip
 
$\triangleright$ \underline{\emph{Case 1: $\rho_0$ satisfies either
 \eqref{cond1} or  \eqref{cond2}.}}\\
There exists a \emph{unique global-in-time solution} $(\rho, u)$ for the Cauchy problem of system \eqref{INS} supplemented with the initial data $(\rho_0, u_0)$, in the sense of Definition \ref{defweaksolu}, satisfying in addition the following  regularity properties:
\begin{align*}
0\leq \rho\leq \rho_*,\quad u\in L^\infty(\R_+; H^1(\R^2)),\quad
\sqrt{\rho}u_t, ~\Delta u,~ \nabla P\in L^2(\R_+\times \R^2),\\
 {\rm  and~for~any}~T>0,~ \sqrt{\rho t}u_t\in L^\infty(0, T; L^2(\R^2)),~\nabla(\sqrt{t} u_t)\in L^2((0, T)\times \R^2),  \\
 \nabla^2 (\sqrt{t} u),~ \nabla (\sqrt{t}P)\in L^q(0, T; L^r(\R^2)),~{\rm for~all}~r\in[2, \infty),~q\in [2,  {2r}{/(r-2)}).
\end{align*}
Moreover, we have\footnote{This $L^1_tLip_x$  regularity is essential not only for the uniqueness part of the statement, but also for Theorem \ref{C1} below.}
$\nabla u\in L^1_{\rm loc}(\R_+; L^\infty(\R^2))$.

\smallskip

$\triangleright$ \underline{\emph{Case 2: neither \eqref{cond1} nor \eqref{cond2} are satisfied.}}\\
There still exists a  \emph{unique global-in-time solution} $(\rho,u)$ for the Cauchy problem of system \eqref{INS} supplemented with the initial data $(\rho_0, u_0)$, in the sense of Definition \ref{defweaksolu}, which satisfies\footnote{Here and elsewhere in the paper $\dot{u}$ denotes the convective derivative, i.e. $\dot{u}=(\partial_t+u\cdot\nabla)u$.}
\begin{align*}
0\leq \rho\leq \rho_*,~u\in L^\infty_{\rm loc}(\R_+\times \R^2),~ \nabla u\in L^\infty(\R_+;L^2(\R^2)),\\
\sqrt{\rho} \dot{u}, ~\Delta u,~ \nabla P\in L^2(\R_+\times \R^2),\\
 \sqrt{\rho t}\dot{u}\in L^\infty(\R_+; L^2(\R^2)),~\nabla(\sqrt{t} \dot{u})\in L^2(\R_+\times \R^2),
 \end{align*}
 and for any $T>0$,
\begin{align*}
  \nabla^2 (\sqrt{t} u),~ \nabla (\sqrt{t}P)\in L^q(0, T; L^r(\R^2)),~{\rm for~all}~r\in[2, \infty),~q\in [2,  {2r}{/(r-2)}),\\
 {\rm and}~~ \sqrt{\rho }\dot{u}, ~\Delta u,~ \nabla P\in L^\infty(0, T; L^2(\R^2)),~\nabla  \dot{u}\in L^2(0, T; \R^2),
\end{align*}
provided that $(\rho_0, u_0)$  additionally satisfies the following \emph{Far-Field Vacuum} conditions\footnote{Notice that the parameter $\alpha>1$ in \eqref{condition3'} is fixed throughout the paper.}
\begin{align}\label{condition3'}
\bar x^\alpha\rho_0\in L^1(\R^2)\cap L^\infty(\R^2) \quad{\rm for~some}~\alpha>1\quad{\rm with}\quad\bar x:= \langle x\rangle\, (\ln\langle x\rangle)^2\tag{FFV-2}
\end{align}
and the \emph{Compatibility} condition
\begin{align}\label{cond-velocity}
-\Delta u_0 +\nabla P_0=\sqrt{\rho_0}g, \quad{\rm  for} ~~u_0\in L^2(\R^2),\ \nabla u_0\in L^1(\R^2),
~ g\in L^2(\R^2).\tag{Compa}
\end{align}
Moreover, we have 
$\nabla u\in L^1_{\rm loc}(\R_+; L^\infty(\R^2))$.
\end{thma}

Next, we state our result in the  three-dimensional case.
\begin{thma}[existence and uniqueness in 3D]\label{thm3d}
Consider  any initial data $(\rho_0, u_0)$ satisfying  \eqref{initialcond} such that  $  u_0\in \mathcal{D}^{1, 2}(\R^3)$, $\nabla\cdot u_0=0$  and
$\rho_0$ satisfies     \eqref{cond:inidensity}.
 Then system \eqref{INS} supplemented with initial data $(\rho_0, u_0)$ admits    a \emph{unique}   solution $(\rho, u)$  on the time interval $(0, T_0)$, in the sense of Definition \ref{defweaksolu}, satisfying in addition the following  regularity properties:\footnote{The $L^1_tLip_x$  regularity below is essential for Theorem \ref{C1}.}
\begin{align*}
0\leq \rho\leq \rho_*,~ u\in L^\infty(0, T_0;  \mathcal{D}^{1, 2}(\R^3)),~
\sqrt{\rho}u_t, ~\Delta u\in L^2((0, T_0)\times \R^3),\\
P\in L^2(0, T_0; \mathcal{D}^{1, 2}(\R^3)),\ \nabla u\in L^1(0, T_0; L^\infty(\R^3)),\\
\sqrt{\rho t}u_t\in L^\infty(0, T_0; L^2(\R^3)),~ \sqrt{t} u_t\in L^2(0, T_0; \mathcal{D}^{1, 2}(\R^3)),  \\
 \nabla^2 (\sqrt{t} u),~ \nabla (\sqrt{t}P)\in L^q(0, T_0; L^r(\R^3)),~\mbox{for all}~r\in[2, 6],~q\in [2, {4r}/{(3r-6)}],
\end{align*}
  where $T_0:= \frac{c\nu^3}{\rho_*^3  \|\nabla u_0\|_{L^2(\R^3)}^4}$ for some universal constant $c$. 
Finally,  there exists a universal constant $c_0>0$ such that   if
\begin{align}\label{small3d}
\rho_*^{\frac{3}{2}}\|\sqrt{\rho_0}u_0\|_{L^2(\R^3)}\|\nabla u_0\|_{L^2(\R^3)}\leq c_0 \nu^2,
\end{align}
then the local-in-time solution  can be extended globally-in-time.
\end{thma}

\begin{rmk}\label{rem.bddnessL2}
The initial condition $\sqrt{\rho_0}u_0, \nabla u_0\in L^2(\R^d)$ together with assumption \eqref{cond1} or \eqref{cond2} imply that $u_0\in L^2(\R^d)$, $d=2$ or $3$. Under \eqref{cond1} or \eqref{cond2}, there exists a positive constant $C_*$ depending only on the factors in condition \eqref{cond1} or \eqref{cond2} such that on the time interval $(0,\infty)$ ($d=2$) and on the life-span $[0, T_*)$ of the solution in Theorem \ref{thm3d} ($d=3$), inequality \eqref{Prop-intp} implies
\begin{align}\label{es-lower1}
\|u(t, \cdot)\|_{L^2(\R^d)}\leq C_*\big(\|\sqrt{\rho}u(t,\cdot)\|_{L^2(\R^d)}+\|\nabla u(t,\cdot)\|_{L^2(\R^d)}\big).
\end{align}
In the two-dimensional far-field vacuum case, i.e. under \eqref{condition3'} and \eqref{cond-velocity}, we also have $u_0\in L^2(\R^2)$ by assumption. However, we cannot show that $u(t,\cdot)\in L^2(\R^2)$. Indeed, instead of the interpolation inequality, we have the weighted estimate \eqref{e.estprop14}.
\end{rmk}

As a by-product, we obtain the following result, which give a positive answer to  Lions' question (\cite[page 34]{PLL}, see above in the preamble) in the whole-space case. As mentioned above the case of bounded domains or the Torus was solved in \cite{DM19}.

\begin{thma}[solution to Lions' problem]\label{C1}
Assume $\gamma\in(0, 1)$ if $d=2$, and $\gamma\in(0, \frac{1}{2})$ if $d=3.$  Let $\Omega_0\subset\R^d$   be a bounded simply connected domain with boundary  $\partial \Omega_0\in \mathcal{C}^{1, \gamma}$.  Suppose that the initial velocity satisfies all the conditions in Theorem \ref{thm2d} or Theorem \ref{thm3d}, and  initial density 
$$\rho_0(x)= 1-\mathbf{1}_{\Omega_0}(x)\quad{\rm or}\quad \rho_0(x)= \mathbf{1}_{\Omega_0}(x), \quad x\in\R^d.$$
Then for each case the unique global solution $(\rho, u)$ of system \eqref{INS} provided by previous theorems satisfies, respectively,
\begin{align*}
\rho(t, x)= 1-\mathbf{1}_{\Omega_t}(x)   \quad{\rm or}\quad  \rho(t, x)= \mathbf{1}_{\Omega_t}(x) \quad {\rm with}\quad\partial \Omega_t \in \mathcal{C}^{1, \gamma},
\end{align*}
where  $\Omega_t:= X(t, \Omega_0)$  and $X$  is the flow associated to the velocity $u,$ that is, the unique solution to
\begin{align*}
X(t, y)= y+ \int_0^t u(s, X(s, y))\,ds, \quad y\in \R^d.
\end{align*}
\end{thma}

\begin{rmk}
Here, contrary to Theorem \ref{thm2d} and Theorem \ref{thm3d}, only one bubble/patch is allowed. As a consequence, the case of multiple bubbles/patches is open.
\end{rmk}

\subsection{A few remarks on the main results} 

\paragraph{\underline{Main novelties}} 
First, our results settle the uniqueness question for the difficult two-dimensional system in the presence of far-field vacuum. Notice that this is a novelty of the whole-space case because in the bounded domain and Torus case treated previously by Danchin and Mucha \cite{DM19} there is no far-field vacuum. Notice that the far-field vacuum case is not accessible by the methods of \cite{DM19}. Indeed, our work seems to be the first unique solvability  result concerning \eqref{INS}  in unbounded domains supplemented with merely bounded densities that can contain vacuum. Moreover, the strategy followed in the current paper maybe  adapted to the half-space case \cite{DPZ}, which remains an interesting problem in the field.

Second, let us stress that in Theorem \ref{thm2d} and Theorem \ref{thm3d}, the density is bounded but does not satisfy any smoothness assumption. Notice that this is not the case of other works concerned with the whole-space. For example, in \cite{LSZ} the initial density satisfies the condition 
\begin{equation*}
\bar x^\alpha\rho_0\in L^1(\R^2)\cap H^1(\R^2)\cap W^{1,p}(\R^2),\quad p>2,
\end{equation*}
while in \cite{CK} the authors assume that $\rho_0\in W^{1,p}(\R^d)$, with $p>d$, $d=2$ or $3$. 
These conditions do not allow for non smooth densities with jump discontinuities like patches or vacuum bubbles.

Third, we emphasize that in both theorems, Theorem \ref{thm2d} and Theorem \ref{thm3d},  we construct finite-energy weak solutions in the sense of Definition \ref{defweaksolu}. These solutions have additional regularity properties that follow from the assumption that the initial velocity belongs to the spaces $\wt{\mathcal{D}}^{1, 2}(\R^2)$ or $\mathcal{D}^{1, 2}(\R^3)$. In this framework, we can prove the uniqueness of the solutions that we construct. Notice that we do not rely on Lions' theorem of existence of finite-energy weak solutions (Theorem \ref{theom-Th-PLL} above), but prove the existence of the solutions with the properties stated above, see Section \ref{s:existence3d} for the existence part of Theorem \ref{thm3d} in the three-dimensional case, and see Section \ref{s:existence2d} for the existence part of Theorem \ref{thm2d} in the two-dimensional case. Finally, let us also remark that condition \eqref{cond3} is not needed neither in 2D nor 3D. In 2D, the far-field vacuum condition \eqref{cond3} is replaced by the condition \eqref{condition3'} and \eqref{cond-velocity}.

\medskip

\paragraph{\underline{Some further estimates}}
We state here certain further boundedness properties for the velocity that can be obtained from the estimates of the paper.

In 3D under \eqref{cond1} or \eqref{cond2}, we already know, see Theorem \ref{thm3d} and Remark \ref{rem.bddnessL2} that $u\in L^\infty_tL^2_x$ on the life-span of the solution. This in combination with the techniques of the paper and the following compatibility assumption 
\begin{align*}
-\Delta u_0 +\nabla P_0=\sqrt{\rho_0}g, \quad{\rm  for} ~ g\in L^2(\R^2).\tag{Compa-3D}
\end{align*}
gives that $u\in L^\infty_{t,x}$ on the life-span of the solution.

In 2D under \eqref{cond1} or \eqref{cond2}, we already know, see Theorem \ref{thm2d} and Remark \ref{rem.bddnessL2} that $u\in L^\infty(0,\infty;L^2(\R^2))$. Combining this with the assumption \eqref{cond-velocity} gives that $u\in L^\infty(\R_+\times \R^2)$. In the case when neither \eqref{cond1} nor \eqref{cond2}, we get at best the weighted-boundedness of the velocity
\begin{equation}\label{e.weighteduLinfty}
\sup_{t\in[0, T]}\| \bar{x}^{-b} {u}(t, \cdot) \|_{L^\infty(\R^2)}<\infty,
\end{equation}
see Proposition \ref{Prop-2dW} under \eqref{condition3'} and \eqref{cond-velocity}. 

\subsection{Main difficulties and strategy for the proofs}\label{S1.4}

First of all, let us point to the main difficulties. Those emerge from the facts that: (i) the density is rough (our analysis includes density patches and vacuum bubbles) and (ii) the density may vanish on some part of $\R^d$. 

\medskip

\paragraph{\underline{Dealing with the vacuum}} 
If there is vacuum at initial time, the velocity equation is degenerate and it becomes particularly difficult to gain control of the velocity itself. It is important to note here that there is an essential difference between the two- and three-dimensional cases. In the three-dimensional case, see Theorem \ref{thm3d}, we can propagate the assumption that $u_0\in \mathcal{D}^{1, 2}(\R^3)$ and hence show that $u\in L^\infty((0,T_0); \mathcal{D}^{1, 2}(\R^3))$. Hence $u\in L^\infty((0,T_0);L^6(\R^3))$ by Sobolev's embedding. In the two-dimensional case in the presence of vacuum, the quantity that is naturally controlled by the energy estimate \eqref{enineq} is $\sqrt{\rho}u$, which generally does not ensure that the velocity $u\in L^p(\R^2)$  for some $p\in[1, \infty]$. 
Hence, one needs to find a way to control $u$ from $\|\sqrt{\rho}u(t,\cdot)\|_{L^2(\R^2)}$ and $\|\nabla u(t,\cdot)\|_{L^2(\R^2)}$. Such estimates can be proved, see Appendix \ref{A}, under specific assumptions that control the size of the vacuum region and the behavior of the density near the vacuum region. In the most favorable case, i.e. condition \eqref{cond1} or \eqref{cond2}, one can show a direct interpolation estimate for the $L^2$ norm of the velocity in terms of $\|\sqrt{\rho}u(t,\cdot)\|_{L^2(\R^2)}$ and $\|\nabla u(t,\cdot)\|_{L^2(\R^2)}$, see Proposition \ref{Prop-intp}. In the case of the far-field vacuum assumption, which is the most difficult case handled in our work, one can only show an estimate for weighted $L^m$ norms of the velocity, $m\geq 2$, see Proposition \ref{Le-Li-Xin}.

\medskip

\paragraph{\underline{The compatibility condition for the far-field vacuum case}}

A compatibility condition in the spirit of \eqref{cond-velocity} first appeared in the works \cite{CK,CK2}. Roughly speaking, such a condition boils down to assuming that $\sqrt{\rho}\partial_tu$ belongs to $L^2(\R^2)$ at initial time. The compatibility condition  was removed in the recent papers \cite{Li,LSZ}. In these works the existence and uniqueness of strong solutions is proved. However, without a compatibility condition, one is unable to infer certain information, such as continuity in time, on the velocity near initial time. One barely has time weighted estimates on the velocity.

To handle the two-dimensional far-field vacuum (Theorem \ref{thm2d} (Case 2)), which is the most difficult case, we show persistence of the initial condition \eqref{condition3'}. Doing so, we can only obtain bounds for $\rho u$ and time-weighted estimates of $u$. 
Condition \eqref{cond-velocity} is then crucial to get the space-weighted boundedness of the velocity near initial time, see \eqref{e.weighteduLinfty} and Proposition \ref{Prop-2dW}, which in turn is key for the uniqueness in Subsection \ref{subsec.2dfarfieldunique}.

The compatibility condition \eqref{cond-velocity} creates important difficulties when trying to approximate the initial data to construct an approximate sequence of smooth solutions. Our construction in Subsubsection \ref{subsubsec.approxffv} is inspired by \cite{CK2}. However, since we apply existence results for smooth data, we need to further regularize the right-hand-side $\sqrt{\rho_0}g$ of \eqref{cond-velocity}. We then require that $u_0\in L^2(\R^2)$ and $\nabla u_0\in L^1(\R^2)$ so as to be able to show the appropriate convergence and boundedness properties of the constructed sequence of approximate data.

\medskip

\paragraph{\underline{Outline of the proofs}}
The general strategy is that of proofs \`a la Hoff \cite{Hoff}. 

We first get lower-order estimates on the velocity. We give propagate the $\wt{\mathcal D}^{1,2}(\R^2)$ or the $\mathcal D^{1,2}(\R^3)$ regularity of the initial data for the velocity. In the three-dimensional case, this estimate directly yields a $L^\infty_tL^6_x$ bound on $u$. In the two-dimensional case, these estimates, in combination with the interpolation estimate of Proposition \ref{Prop-intp} (Case 1 of Theorem \ref{thm2d}) or the weighted interpolation estimate of Proposition \ref{Prop-2dW} (Case 2 of Theorem \ref{thm2d}), lead to estimates of the velocity itself. 

The next step is to get higher-order estimates on the velocity. We carry out time-weighted estimates of $\sqrt{\rho}\dot u$. Finally we transfer (shift) the integrability of $\sqrt{\rho}\dot u$ to second-order derivatives of the velocity via maximal regularity estimates for the stationary Stokes system. This enables us to get the crucial Lipschitz estimate for the velocity.

The Lipschitz estimate for the velocity is key to the uniqueness proofs. To show the uniqueness, we carry out a duality proof, which is a new approach in this context. In particular our approach differs from the Lagrangian approach used in \cite{DM19}.

For more details concerning the relationship between the results in the paper, see Figure \ref{fig.diagrammresults}.

\begin{figure}[h]
\begin{center}
\includegraphics[scale=.75,trim=0 4cm 0 0]{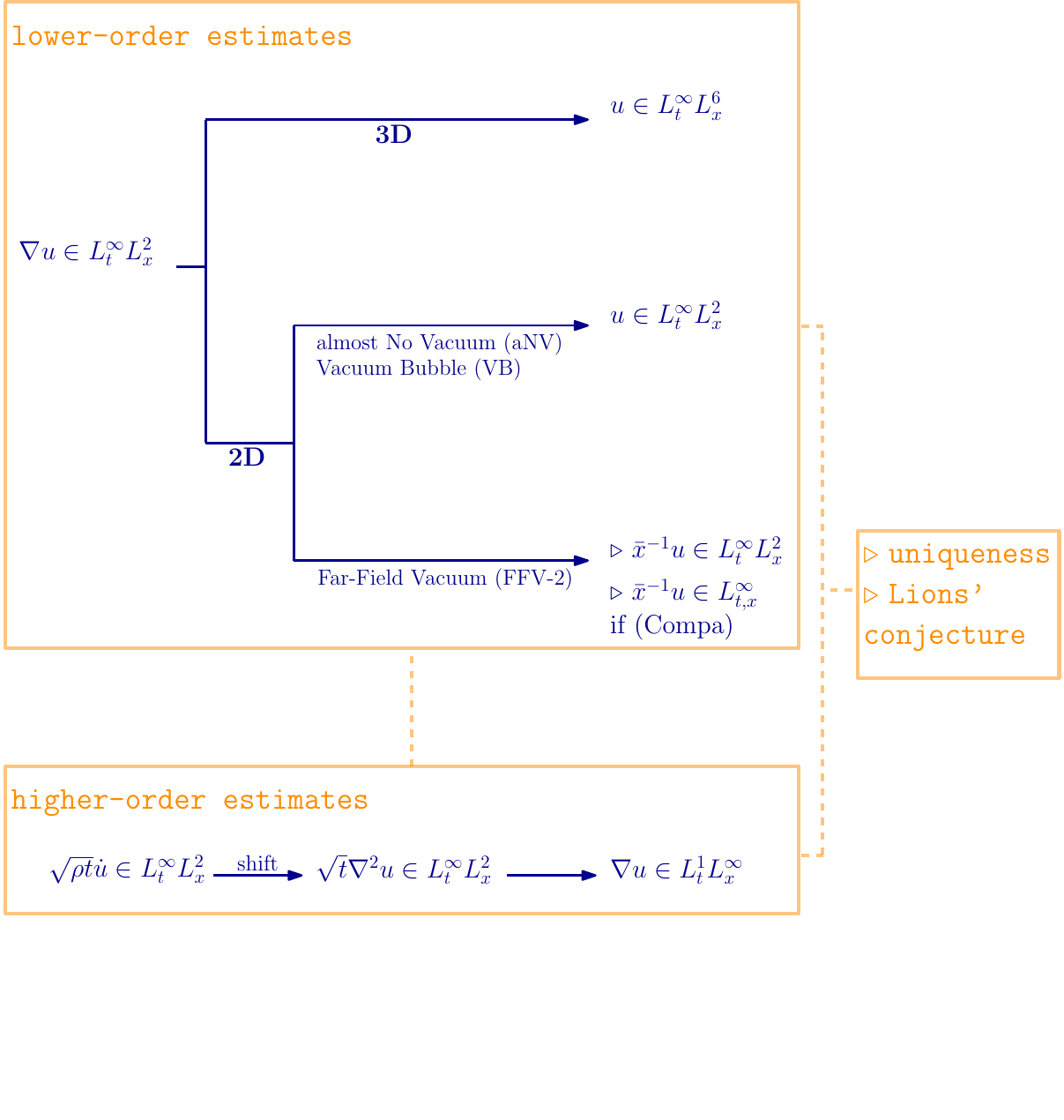}
\caption{Relationships between the results in the paper}
\label{fig.diagrammresults}
\end{center}
\end{figure}

\subsection{Some notations and assumptions}
Throughout, we use $B_n$ to denote the ball $\{x\in \R^d; ~ |x|\leq n\}.$ The notation $(f)_+=\max(f,0)$ stands for the positive part of a function $f$. The japanese bracket is defined as follows $\langle x\rangle:= (e+|x|^2)^{1/2}$ and we recall that the notation $\bar x$ stands for $\bar x:=\langle x\rangle\, (\ln\langle x\rangle)^2$.
We define material derivative $\dot{v}:=\partial_t v+u\cdot\nabla v$ for functions $v:\R_+\times\R^d\to \R^d.$
Sometimes,   we denote $\partial_t v$ by $v_t.$
For simplicity, we will fix the viscosity to be $1$ thanks to a standard rescaling argument.

\subsection{Outline of the paper}
Section \ref{s:existence3d} is concerned with the three-dimensional case. It is devoted to the proof of the existence part of Theorem \ref{thm3d}. Section \ref{s:existence2d} handles the two-dimensional case. It is devoted to the existence part of Theorem \ref{thm2d}. The proof is more involved, especially in the far-field vacuum case, than in 3D. The proof of  uniqueness partis of Theorem \ref{thm2d} and Theorem \ref{thm3d}  is carried out in Section \ref{s:uniqueness}. Theorem \ref{C1} is proved in Section \ref{B}. Some frequently used functional spaces and essential interpolation inequalities are recalled in  Appendix \ref{A}.

%%%%%%%%%%%%%%%%%%%%%%%%%%%%%%%%%%%%%%%%

\section{Proof of the existence results in the three-dimensional case, Theorem \ref{thm3d}}\label{s:existence3d}
This section mainly concerns the  proof of the existence part in Theorem \ref{thm3d} (3D case).
The idea is to take advantage of classical results to  construct smooth approximate solutions without vacuum. After that,   we show  persistence of $\mathcal{D}^{1, 2}(\R^3)$ regularity and
 some  time-weighted  estimates for  time derivatives of the velocity  independent of the lower bounds of the density.   Those estimates will  enable us to   shift the regularity to obtain $L^1_tLip_x$ estimates for the velocity, which is   important for our proof of uniqueness.  Finally,  we pass to the limit via classical compactness argument.
 
\medskip

  Recall that the initial data $(\rho_0, u_0)$  satisfies conditions \eqref{initialcond} and \eqref{cond:inidensity}, and that $  u_0\in \mathcal{D}^{1, 2}(\R^3), \nabla\cdot u_0=0.$
  Thus, we consider smoothed-out  initial data
  \begin{align*}
  \rho_0^\eps\in \mathcal{C}^\infty(\R^3)~\,\,{\rm with}~\,\,\eps\leq \rho_0^\eps\leq 2\rho_*\and u_0^\eps\in \mathcal{C}_0^\infty(\R^3) ~\,\,{\rm with}~\,\,\nabla\cdot u_0^\eps=0
  \end{align*}
such that as $\eps\to 0^+$
\begin{align*}
u_0^\eps \to u_0\quad{\rm in}~~\mathcal{D}^{1, 2}(\R^3),\quad \sqrt{\rho_0^\eps}u_0^\eps\,\rightharpoonup\, \sqrt{\rho_0}u_0 \quad{\rm in}~~L^2(\R^3)
\end{align*}
and
\begin{align*}
\rho_0^\eps \rightharpoonup \rho_0\quad{\rm weak*~ in}~~L^{\infty}(\R^3).
\end{align*}
Then, by the strong solution theory  (e.g. \cite[Theorem 0.2]{Dan2}), we know that there exists a local-in-time strong solution $(\rho^\eps, u^\eps)$ to \eqref{INS}.

\subsection{Uniform estimates}
 In what follows, we focus on uniform estimates for the approximate solutions  $(\rho^\eps, u^\eps)$.
We often use the following Sobolev embedding  
\begin{equation}\label{sobolevem}
 \| z\|_{L^6(\R^3)}   \leq C\|\nabla z \|_{L^2(\R^3)}\quad {\rm for } ~~z\in \dot{H}^1(\R^3),
\end{equation}
with a universal positive constant $C$. For notational simplicity,  we  omit the superscript  $\eps$ and write the solution $(\rho,u)$.

 \subsubsection{Persistence of $\mathcal{D}^{1, 2}{(\R^3)}$ regularity for the velocity}
\begin{prop}[gradient estimate]\label{prop-H13d}
Let $(\rho, u)$ be a smooth enough  solution to system \eqref{INS} on $[0, T_*)\times \R^3.$   There exist  positive  universal constants $c_0, C$ such that  if
\begin{equation}\label{smallcond3d}
\rho_*^\frac{3}{2}\|\sqrt{\rho_0}u_0\|_{L^2(\R^3)}\|\nabla u_0\|_{L^2(\R^3)}\leq c_0
\end{equation}
then   it holds for all $t\in[0, T_*)$
\begin{multline}\label{es-2-3d}
\|  u(t, \cdot)\|_{\mathcal{D}^{1, 2}(\R^3)}^2+2\int_{0}^t\|(\sqrt{\rho} {u}_t, \Delta u)(s, \cdot)\|_{L^2(\R^3)}^2\,ds+ 2\int_{0}^t\|P(s, \cdot)\|_{\mathcal{D}^{1, 2}(\R^3)}^2\,ds \\
\leq \|\nabla u_0
\|_{L^2(\R^3)}^2\exp(C\|\sqrt{\rho_0}u_0\|_{L^2(\R^3)}^2\|\nabla u_0\|_{L^2(\R^3)}^2).
\end{multline}
If \eqref{smallcond3d} is not satisfied, then \eqref{es-2-3d} holds true  on $[0, T]$ provided that
 \begin{equation}\label{time3d}
 T\leq \frac{c}{\rho_*^3\|\nabla u_0\|_{L^2(\R^3)}^4}\quad {\rm for~some ~universal~constant}~c.
 \end{equation}
\end{prop}

\begin{proof}
In order to estimate the second derivative of $u$ and the gradient of the pressure, we rewrite the momentum equation of \eqref{INS} in the form
\begin{equation}\label{sstokes}
\left\{\begin{aligned}
&-\Delta u+\nabla P=-\rho\dot{u}\hspace*{1.17cm} {\rm in}~ (0, T)\times \R^3,\\
&\nabla\cdot u=0 \hspace*{2.99cm}{\rm in}~ (0, T)\times \R^3,\\
&u\to 0,\quad{\rm as}~|x|\to \infty.
\end{aligned}\right.
\end{equation}
Applying the standard $L^2$-estimate to \eqref{sstokes} yields that
\begin{align}\label{P1-1}
\|\Delta u\|_{L^2(\R^3)^2}^2+\|\nabla P\|_{L^2(\R^3)}^2\leq  C \Big( {\rho_*}\|\sqrt{\rho} u_t\|_{L^2(\R^3)}^2+\rho_*^2 \|u\cdot\nabla u \|_{L^2(\R^3)}^2\Big).
\end{align}
Testing the momentum equation of \eqref{INS} against ${u_t}$ yields
\begin{align*}
\frac{1}{2}\frac{d}{dt}\|\nabla u\|_{L^2(\R^3)}^2+ \|\sqrt{\rho}{u}_t\|_{L^2(\R^3)}^2&=-\langle \rho u\cdot\nabla u,u_t\rangle.
\end{align*}
Hence, by H\"{o}lder's inequality
\begin{align*}
 \frac{1}{2}\frac{d}{dt}\|\nabla u\|_{L^2(\R^3)}^2+ \|\sqrt{\rho}{u}_t\|_{L^2(\R^3)}^2
\leq& \|\sqrt{\rho}u\cdot\nabla u\|_{L^2(\R^3)}  \|\sqrt{\rho} u_t\|_{L^2(\R^3)}\\
\leq& C \rho_*\|u\cdot\nabla u\|_{L^2(\R^3)}^2+   \frac{1}{4}\|\sqrt{\rho} u_t\|_{L^2(\R^3)}^2
\end{align*}
Putting the above inequality together with \eqref{P1-1} implies
\begin{align}\label{P1-2}
\frac{d}{dt}\|\nabla u\|_{L^2(\R^3)}^2+ \frac{1}{\rho_*}\|(\sqrt{\rho} {u}_t, \Delta u, \nabla P)\|_{L^2(\R^3)}^2
 \leq &  C \rho_*\|u\|_{L^6(\R^3)}^2 \|\nabla u\|_{L^3(\R^3)}^2\\
 \leq &  C \rho_*\|\nabla u\|_{L^2(\R^3)}^3 \|\Delta u\|_{L^2(\R^3)}\notag\\
 \leq & C   \rho_*^3\|\nabla u\|_{L^2(\R^3)}^6+\frac1{2\rho_*}\|\Delta u\|_{L^2(\R^3)}^2,\notag
\end{align}
in which  inequality \eqref{sobolevem} and the Gagliardo-Nirenberg inequality \eqref{GN-in}  were used.
 Therefore,  we see that whenever $T$ satisfies $C\rho_*^3\|\nabla u_0\|_{L^2(\R^3)}^2\int_0^T \|\nabla u\|_{L^2(\R^3)}^2\,ds<1,$ we have
 \begin{align*}
 \|\nabla u(t, \cdot)\|_{L^2(\R^3)}^2\leq \frac{\|\nabla u_0\|_{L^2(\R^3)}^2}{1-C\rho_*^3\|\nabla u_0\|_{L^2(\R^3)}^2\int_0^T \|\nabla u(s, \cdot)\|_{L^2(\R^3)}^2\,ds}\quad{\rm for~all}~t\in[0, T].
 \end{align*}
 Now, the energy inequality \eqref{enineq}  implies that
 \begin{align*}
  \|\nabla u_0\|_{L^2(\R^3)}^2\int_0^t \|\nabla u\|_{L^2(\R^3)}^2\,ds \leq \|\nabla u_0\|_{L^2(\R^3)}^2\|\sqrt{\rho_0}u_0\|_{L^2(\R^3)}^2\leq \frac{1}{2C\rho_*^3}\quad{\rm for~any}~t\geq0
 \end{align*}
 provided that $(\sqrt{\rho_*})^3\|\nabla u_0\|_{L^2(\R^3)}\|\sqrt{\rho_0}u_0\|_{L^2(\R^3)}\leq c_0$ for small enough $c_0.$ Using inequality \eqref{sobolevem} again,  we get \eqref{es-2-3d}.

  In the case when the smallness condition is not satisfied, we get from \eqref{P1-2} that
  \begin{align*}
  \|\nabla u(t, \cdot)\|_{L^2(\R^3)}^4\leq \frac{\|\nabla u_0\|_{L^2(\R^3)}^4}{1-2C\rho_*^3t\|\nabla u_0\|_{L^2(\R^3)}^4},
  \end{align*}
 which eventually enables us to finish the proof of the second statement.
\end{proof}

\begin{rmk}[boundedness in space of the velocity]\label{rem.bddness}
Notice that thanks to inequality \eqref{es-2-3d} we also have
\begin{align*}
 \int_0^T\|u(t, \cdot)\|_{ L^\infty(\R^3)}^4\,dt&\leq \int_0^T\|\nabla  u\|_{L^2(\R^3)}^2\|\nabla^2 u\|_{L^2(\R^3)}^2\,dt\\
&\leq \|\nabla  u_0\|_{L^2(\R^3)}^4\exp(C\|\sqrt{\rho_0}u_0\|_{L^2(\R^3)}^2\|\nabla u_0\|_{L^2(\R^3)}^2).
\end{align*}
\end{rmk}

\subsubsection{Estimates of the time derivative}
Here, we want to   bound time derivative $\sqrt{\rho t}  {u}_t$ in the space  $L^\infty(0, T; L^2(\R^3))$ and $\sqrt{t}\nabla {u}_t$ in the space $L^2((0, T)\times \R^3),$
since it is an important step towards the proof of existence and higher-order spatial estimates for the velocity.

\begin{prop}[time derivative estimates]\label{Prop-timees3d}
   Let $(\rho, u)$ be a smooth enough solution to system \eqref{INS} on $[0, T_*)\times \R^3.$ Then for all $T\in[0, T_*)$ it holds
\begin{equation}
\sup_{t\in[0, T]}\|\sqrt{\rho t} {u}_t \|_{  L^2(\R^3)}^2+ \int_0^T\|  \sqrt{t} {u}_t\|_{\mathcal{D}^{1, 2}(\R^3)}^2\,ds \leq C_{0,  T},
\end{equation}
where $C_{0, T}$ is a constant depending only on $T, \rho_*$ and norms $ \|\sqrt{\rho_0}u_0\|_{L^2(\R^3)}, \|\nabla u_0\|_{L^2(\R^3)}.$
\end{prop}
\begin{proof}
 At first,  applying the time derivative $\partial_t $  to the momentum equation in system \eqref{INS}  and multiplying  the resulting equation   by $\sqrt{t}$ yields
\begin{multline}\label{eq-td}
\rho\big(\partial_t  (\sqrt{t} {u}_t)+u\cdot\nabla (\sqrt{t} {u}_t)\big)-\Delta   (\sqrt{t}{u}_t)+\nabla  (\sqrt{t}{P}_t)\\
=  \frac{1}{2\sqrt{t}}\rho u_t -\sqrt{t}\rho_t u_t-\sqrt{t}\rho_t u\cdot\nabla u-\sqrt{t}\rho  u_t\cdot\nabla u.
\end{multline}
Taking the $L^2$ scalar product with $\sqrt{t} {u}_t$,  we get
\begin{align}\label{time-es1}
\frac{1}{2}\frac{d}{dt}\|\sqrt{\rho t} {u}_t\|_{L^2}^2+\|\nabla(\sqrt{t} {u}_t)\|_{L^2}^2=A_1+\dots+A_4
\end{align}
where
\begin{align*}
&A_1:=\frac{1}{2}\|\sqrt{\rho }{u}_t\|_{L^2(\R^3)}^2,\\
&A_2:=-\langle  t\rho_t , |u_t|^2\rangle,\\
&A_3:=- \langle\sqrt{t}\rho_t u\cdot\nabla u, \sqrt{t}{u}_t\rangle,\\
&A_4:=-\langle\sqrt t\rho   u_t\cdot\nabla u, \sqrt{t} {u}_t\rangle.
\end{align*}
In order to estimate  $ A_1, A_2, A_3, A_4,$  we proceed as follows.
Noticing that
\begin{align*}
A_1&= \frac{1}{2}\|\sqrt{\rho }(\dot{u}-u\cdot\nabla u) \|_{L^2(\R^3)}^2\\
&\leq \|\sqrt{\rho }\dot{u}  \|_{L^2(\R^d)}^2+ \rho_*\|u\|_{L^\infty(\R^3)}^2\|\nabla u \|_{L^2(\R^3)}^2.
\end{align*}
 Using the equation $\dot\rho=\partial_t \rho +u\cdot\nabla\rho=\partial_t \rho +\nabla\cdot(\rho u)=0,$ we write
\begin{align*}
A_2=-\langle  t \rho u , \nabla(|u_t|^2)\rangle
\end{align*}
and
\begin{align*}
|A_2|\leq& 2  \sqrt{\rho_*}\|u\|_{L^\infty(\R^3)}\|\sqrt{\rho t} u_t\|_{L^2(\R^3)}\|\nabla(\sqrt{t} u_t)\|_{L^2(\R^3)}\\
\leq& \frac{1}{8} \|\nabla(\sqrt{t} u_t)\|_{L^2(\R^3)}^2+ C\rho_* \|\sqrt{\rho t} u_t\|_{L^2(\R^3)}^2\|u\|_{L^\infty(\R^3)}^2.
\end{align*}
Similarly, we write
\begin{align*}
A_3=-\langle  t \rho u  , \nabla[(u\cdot\nabla u)\cdot  u_t]\rangle.
\end{align*}
and decompose
\begin{align*}
|A_3|&\leq \langle  |t \rho u|  , |\nabla u|^2\, |u_t|\rangle +\langle | t \rho u | ,  |u|\,|\nabla^2 u|\,|u_t|\rangle+ \langle  |t \rho u|  , |u|\, |\nabla u|\, |\nabla u_t|\rangle\\
&=: A_{31}+A_{32}+A_{33}.
\end{align*}
Thanks to   inequality \eqref{sobolevem} and Young's inequality
 \begin{align*}
A_{31}&\leq \rho_*\sqrt{ T}\| \sqrt{  t} u_t\|_{L^6(\R^3)}\| u\|_{L^6(\R^3)}\|\nabla u\|_{L^3(\R^3)}^2\\
&\leq C\rho_*\sqrt{ T}\|\nabla( \sqrt{  t} u_t)\|_{L^2(\R^3)}\|\nabla u\|_{L^2(\R^3)}^2\|\nabla^2 u\|_{L^2(\R^3)}\\
&\leq \frac{1}{8}\|\nabla( \sqrt{  t} u_t)\|_{L^2(\R^3)}^2+  C\rho_*^2 { T}\|\nabla u\|_{L^2(\R^3)}^4\|\nabla^2 u\|_{L^2(\R^3)}^2.\\
\end{align*}
As for $A_{32}, A_{33}$, we have
\begin{align*}
A_{32}&\leq \sqrt{ \rho_*T}\|\sqrt{\rho t}u_t\|_{L^2(\R^3)}\|u\|_{L^\infty(\R^3)}^2\|\nabla^2 u\|_{L^2(\R^3)}\\
&\leq\|\sqrt{\rho t}u_t\|_{L^2(\R^3)}^2\|u\|_{L^\infty(\R^3)}^4+ \rho_*T\|\nabla^2 u\|_{L^2(\R^3)}^2
\end{align*}
and
\begin{align*}
A_{33}&\leq  \rho_*\sqrt{ T} \|u\|_{L^\infty(\R^3)}^2\|\nabla u\|_{L^2(\R^3)}\|\nabla (\sqrt{t}u_t)\|_{L^2(\R^3)}\\
&\leq  \frac{1}{8}\|\nabla( \sqrt{  t} u_t)\|_{L^2(\R^3)}^2 + C\rho_*^2 { T} \|u\|_{L^\infty(\R^3)}^4\|\nabla u\|_{L^2(\R^3)}^2.
\end{align*}
To handle $A_4,$ we  use  inequality \eqref{sobolevem} and the Gagliardo-Nirenberg inequality \eqref{GN-in}  to get
\begin{align*}
|A_4|\leq& \sqrt{\rho_*}\|\sqrt{\rho t} u_t\|_{L^2(\R^3)}\|\nabla u\|_{L^3(\R^3)}\| \sqrt t u_t\|_{L^6(\R^3)}\\
\leq & \sqrt{\rho_*}\|\sqrt{\rho t} u_t\|_{L^2(\R^3)}\|\nabla u\|_{L^2(\R^3)}^\frac{1}{2}\|\nabla^2 u\|_{L^2(\R^3)}^\frac{1}{2}\|\nabla(\sqrt{t} u_t)\|_{L^2(\R^3)}\\
\leq & \frac{1}{8}\|\nabla(\sqrt{t} u_t)\|_{L^2(\R^3)}^2+ C {\rho_*}\|\sqrt{\rho t} u_t\|_{L^2(\R^3)}^2\|\nabla u\|_{L^2(\R^3)} \|\nabla^2 u\|_{L^2(\R^3)}.
\end{align*}
By putting all the above estimates into \eqref{time-es1}, we find that
 \begin{multline}\label{time-es3d}
  \frac{d}{dt}\|\sqrt{\rho t} {u}_t\|_{L^2(\R^3)}^2+\|\nabla(\sqrt{t} {u}_t)\|_{L^2(\R^3)}^2\leq \|\sqrt{\rho t}u_t\|_{L^2(\R^3)}^2\left( 1
 +\|u\|_{L^\infty(\R^3)}^4+\|\nabla u\|_{L^2(\R^3)} \|\nabla^2 u\|_{L^2(\R^3)}\right)\\
 +  C_{T}\left(   \|\sqrt{\rho }\dot{u}  \|_{L^2 (\R^3)}^2+  (\|u\|_{L^\infty(\R^d)}^4+1)\|\nabla u \|_{L^2(\R^3) }^2+( \|\nabla u\|_{L^2 (\R^3)}^4+1)\|\nabla^2 u\|_{L^2(\R^3) }^2\right).
\end{multline}
Using \eqref{es-2-3d} and Remark \ref{rem.bddness} again,  it is not difficult to conclude that  inequality  \eqref{time-es3d} can be rewritten as
\begin{align*}
 \frac{d}{dt}\|\sqrt{\rho t} {u}_t\|_{L^2(\R^3)}^2+\|\nabla(\sqrt{t} {u}_t)\|_{L^2(\R^3)}^2\leq B_{3d}(t)\|\sqrt{\rho t}u_t\|_{L^2(\R^3)}^2 +B_{3d}(t)
\end{align*}
for some function $B_{3d}\in L^1(0, T),$   the norm of which  bounded only in terms of  time $T,$ $\rho_*$ and norms $\|\sqrt{\rho_0} u_0\|_{L^2(\R^3)}, \|\nabla  u_0\|_{L^2(\R^3)}.$

Finally, Gronwall's lemma implies that for all $t\in[0, T]$
\begin{align}
\|\sqrt{\rho t} {u}_t\|_{ L^2(\R^3)}^2+\int_0^T \|\nabla(\sqrt{t} {u}_t)\|_{L^2(\R^3)}^2\,ds\leq \int_0^T B_{3d}(t)\,dt \exp\Big(\int_0^T B_{3d}(t)\,dt\Big ).
\end{align}
This completes the proof of proposition \ref{Prop-timees3d}.
\end{proof}

\subsubsection{Shift of regularity}

 As a consequence of Proposition \ref{Prop-timees3d},
we will get higher-order estimates for the velocity, via considering the following multi-dimensional stationary  Stokes problem
\begin{equation}\label{stokes-timeweighted}
\left\{\begin{aligned}
&-\Delta (\sqrt{t}u)+\nabla (\sqrt{t}P)=-\rho \sqrt{t}(u_t+u\cdot\nabla u)\hspace*{1.17cm} {\rm in}~ (0, T)\times \R^d,\\
&\nabla\cdot (\sqrt{t}u)=0 \hspace*{6.19cm}{\rm in}~ (0, T)\times \R^d,\\
&\sqrt{t}u\to 0, \quad |x|\to\infty.
\end{aligned}\right.
\end{equation}
We have
\begin{prop}[higher-order estimates]\label{Prop-shift3d}
   Let $(\rho, u)$ be a smooth enough  solution to system \eqref{INS} on $[0, T_*)\times \R^3$. Then for all $T\in[0, T_*)$ it holds
  that for all $r\in[2, 6]$ and $q\in [2, \frac{4r}{3r-6}]$
 \begin{align*}
\|\nabla^2(\sqrt{t} u)\|_{L^{q}(0, T; L^r(\R^3))}+\|\nabla(\sqrt{t}P)\|_{L^{q}(0, T; L^r(\R^3))}\leq C_{0, T}.
\end{align*}
Moreover, we have the key Lipschitz estimate
\begin{align*}
\int_0^T \|\nabla u(t, \cdot)\|_{L^\infty(\R^3)}\,dt\leq C_{0, T}\, T^{\frac{1}{4}}
 \end{align*}
 and $C_{0, T}\, T^{\frac{1}{4}} \to 0$ as $T\to 0.$
\end{prop}
\begin{proof}
 Applying the standard $L^p$ estimates for  system \eqref{stokes-timeweighted} yields that for all finite $p,$
\begin{align}\label{shift-es13d}
\|\nabla^2(\sqrt{t} u)\|_{ L^p(\R^3)}+\|\nabla(\sqrt{t}P)\|_{L^p(\R^3)}\leq C(p)\sqrt{\rho_*}\| \sqrt{\rho t} \dot{u}\|_{L^p(\R^3)}.
\end{align}

Now, we  estimate the material derivative in the right-hand side of above inequality.   On the one hand, we  have   by  inequality \eqref{sobolevem}  and Proposition \ref{Prop-timees3d}
 \begin{align*}
\int_0^T \| \sqrt{\rho t}u_t\|_{L^6(\R^3)}^2\,dt
& \leq   \rho_*\int_0^T \| \nabla(\sqrt{  t }u_t)\|_{L^2(\R^3)}^2\,dt
\leq C_{0, T}.
\end{align*}
Interpolating this with the estimate in  Proposition \ref{Prop-timees3d} gives that  for all  $r\in[2, 6]$ and  $q\in[2,  \frac{4r}{3r-6}]$
\begin{align}\label{es-shift-es113}
\| \sqrt{\rho t}u_t\|_{L^q(0, T; L^r(\R^3)}\leq C_{0, T}.
\end{align}
In the other hand, taking $p=2$ in inequality \eqref{shift-es13d}  and using Young's inequality  yield that
\begin{align}
& \|\nabla^2(\sqrt{t} u)\|_{L^\infty(0, T; L^2(\R^3))} \label{es-shift-3d00}\\
\leq& C(\sqrt{\rho_*}\|\sqrt{\rho t}u_t\|_{L^\infty(0, T; L^2(\R^3) )}+\rho_*\|\sqrt{t}u\cdot\nabla u\|_{L^\infty(0, T; L^2(\R^3) )})\notag\\
\leq& C(\sqrt{\rho_*} \|\sqrt{\rho t}u_t\|_{L^\infty(0, T; L^2 (\R^3))}+\rho_* \|\sqrt{t}u\|_{L^\infty(0, T; L^\infty(\R^3) )}\|\nabla u\|_{L^\infty(0, T; L^2(\R^3) )})\notag \\
\leq& C(\rho_*)\left( \|\sqrt{\rho t}u_t\|_{L^\infty(0, T; L^2(\R^3) )}+  \|(\sqrt{t}\nabla u, \sqrt{t}\nabla^2 u)\|_{L^\infty(0, T; L^2(\R^3) )}^\frac{1}{2}\|\nabla u\|_{L^\infty(0, T; L^2(\R^3) )}\right) \notag\\
\leq&  C_{0, T}. \notag
\end{align}
Then by the Gagliardo-Nirenberg inequality \eqref{GN-in}, Proposition \ref{prop-H13d} and   estimate \eqref{es-shift-3d00}
 \begin{align*}
 \|\sqrt{\rho t}u\cdot\nabla u\|_{L^2(0, T; L^6(\R^3))}
\leq&\sqrt{ \rho_* }  \left(\int_0^T \|  u\|_{L^{\infty}(\R^3)}^2  \|\sqrt{t}\nabla u \|_{L^{6}(\R^3)}^2\,dt)\right)^{{1}/{2}}\\
\leq& \sqrt{ \rho_*}  \left(\int_0^T \|\nabla u\|_{L^{2}(\R^3)}\|\nabla^2 u \|_{L^{2}(\R^3)}  \|\sqrt{t}\nabla u \|_{L^{6}(\R^3)}^2\,dt\right)^{1/2}\\
 \leq&  \sqrt{ \rho_*} \left( \int_0^T\|\nabla u\|_{L^{2}(\R^3)}  \big(\|\nabla^2 u \|_{L^{2}(\R^3)}^2+ \|\sqrt{t}\nabla u\|_{L^6(\R^3)}^4\big)\,dt\right)^{1/2}\\
 \leq& C_{0, T}.
\end{align*}
Notice that the Gagliardo-Nirenberg inequality \cite[page 54]{Galdibook} and  estimate \eqref{es-shift-3d00}  give that
\begin{align*}
\| \sqrt{\rho t}u\cdot\nabla u\|_{L^\infty(0, T; L^{2}(\R^3))}\leq &
\| \sqrt{\rho t}u\|_{L^\infty(0, T;  L^{\infty} (\R^3))} \|   \nabla u \|_{L^\infty(0, T;  L^{2} (\R^3))} \\
\leq &\sqrt{\rho_*  } \|\sqrt{t} u\|_{L^\infty(0, T; L^2 (\R^3))}^\frac{1}{2}\|\sqrt{t}\nabla^2 u\|_{L^\infty(0, T; L^2(\R^3) )}^\frac{1}{2}\|\nabla u\|_{L^\infty(0, T;  L^{2} (\R^3))}\\
\leq &  C_{0, T}.
\end{align*}
Hence, by interpolation,  for all $ r\in [2, 6]$ and   $q\in [2, \frac{4r}{3r-6}]$
\begin{align*}
\| \sqrt{\rho t}u\cdot\nabla u \|_{L^{q}(0, T;  L^{  r} (\R^2))}\leq    C_{0, T},
\end{align*}
 which together with \eqref{es-shift-es113}  and \eqref{shift-es13d} imply the desired estimate.

Finally,  we have by inequality \eqref{GN-in} again
 \begin{align*}
\int_0^T \|\nabla u\|_{L^\infty(\R^3)}\,dt\leq& \int_0^T\|\nabla u\|_{L^2(\R^3)}^\frac{1}{4} \|\nabla^2 (\sqrt{t} u)\|_{L^6(\R^3)}^\frac{3}{4}\, t^{-\frac{3}{8}}\,dt\\
 \leq &  \|\nabla u\|_{L^\infty(0, T; L^2(\R^3))}^\frac{1}{4}   \|\nabla^2 (\sqrt{t} u)\|_{L^{2}(0, T; L^6(\R^3))}^\frac{3}{4}  \|t^{-\frac{3}{8}}\|_{L^{\frac{8}{5}}(0, T)}\\
 \leq & C_{0, T}\, T^{\frac{1}{4}}.
 \end{align*}
It remains  to check   in the proof of Proposition \ref{Prop-timees3d}   that $C_{0, T}$ is uniformly bounded as $T\to 0.$
This completes the proof of Proposition \ref{Prop-shift3d}.
\end{proof}

 \subsection{Proof of existence}
\label{sec.exi3D}
One can now turn to the proof of the existence of a solution stated in Theorem \ref{thm3d}.   Denoting by $T^\eps$ the maximal time of existence for  approximate solutions $(\rho^\eps, u^\eps)$,  then \cite[Theorem 1.4]{Z} shows that $T^\eps\geq T_0$.   On the other hand, in the case the smallness condition \eqref{small3d} is satisfied, one can  continue $(\rho^\eps, u^\eps)$  globally-in-time  thanks to the   $L^1_tLip_x$ estimate in Proposition \ref{Prop-shift3d} and the blow-up criterion in \cite[Theorem 0.4]{Dan2} or \cite[Theorem 4]{CK2}.  It   remains to prove the convergence of approximate solutions.
 At this stage,  with all the estimates established in the previous subsection,  the standard compactness argument  yields that, up to a subsequence,
  $$\rho^\eps\rightharpoonup \rho\quad {\rm weak* ~in}~L^\infty(\R_+\times \R^3)\and 0\leq \rho\leq \rho_*,$$
  $$\nabla u^\eps\rightharpoonup \nabla u\quad {\rm weak ~in}~L^2( \R_+;  L^2(\R^3),$$
 $$u^\eps \rightharpoonup u\quad {\rm weak ~in}~L^2(\R_+;  L^6(\R^3),$$
 for some $(\rho, u)$ satisfies all the regularity results stated in Theorem \ref{thm3d}, except for the regularities of non-linear term $\sqrt{\rho} u_t.$
The compactness result  obtained  in \cite{BD97}  then shows that
  \begin{align*}
 {\rho^\eps}\to  {\rho}\quad{\rm in}~~\mathcal{C}_{\rm loc}(\R_+; L^q_{\rm loc}(\R^3)), \quad\forall  \,q\in[0, \infty)
 \end{align*}
 and
  \begin{align*}
 {\sqrt{\rho^\eps} u^\eps}\to  {\sqrt{\rho} u}\quad{\rm in}~~L^2_{\rm loc}(\R_+\times \R^3).
 \end{align*}
 From which, we know from Proposition \ref{Prop-timees3d} that $\sqrt{\rho^\eps} u^{\eps}_t\rightharpoonup \sqrt{\rho} u_t$ weak  in $\mathcal{D}'((0, T_0)\times \R^3))$ and thus $\sqrt{\rho} u_t\in L^2(0, T_0; L^2(\R^3))$ thanks to the uniform bounds \eqref{es-2-3d}. Similarly, one has  $\sqrt{\rho t} u_t \in L^\infty(0, T_0; L^2(\R^3)).$

Finally, all the  compactness information  above  is enough to justify that the couple $(\rho, u)$ is a weak solution to \eqref{INS} in the sense of distributions. This finishes the proof of the existence part of Theorem \ref{thm3d}.    \qed

\section{Proof of the existence results in the two-dimensional case, Theorem \ref{thm2d}}\label{s:existence2d}

This section is concerned with the  proof of the existence part in Theorem \ref{thm2d} (2D case).  The initial data $(\rho_0, u_0)$ satisfies condition \eqref{initialcond} and \eqref{cond:inidensity}, and $\nabla u_0\in L^2(\R^2), \nabla\cdot u_0=0$. We analyse separately the case when: (i) \eqref{cond1} or \eqref{cond2} is satisfied, see Subsection \ref{subsec.cond12} below, (ii) \eqref{condition3'} is satisfied, see Subsection \ref{s:ffv}. The idea is similar to the three-dimensional case. We first approximate the data or mollify the equation\footnote{We use two different approximation procedures depending on which case (i) or (ii) is considered} and rely on known existence results for smooth approximate solutions without vacuum.
Then, we show the persistence of the regularity $u\in \wt{\mathcal D}^{1,2}(\R^2)$ 
and other higher-order estimates uniformly for the approximate solutions. The actual existence proof is similar to the one in the three-dimensional case, see Subsection \ref{sec.exi3D}.

\medskip

Before going into the details of each case, let us recall an estimate obtained in \cite{LSZ} whose proof  is motivated by the work \cite{LX} of  Li and Xin on the two-dimensional compressible Navier-Stokes equations with vacuum. For the reader's convenience, we provide its proof here and  mention in particular that it is valid for our approximate solutions constructed in the follow-up subsections.

\begin{prop}[gradient estimate; {\cite[Lemma 3.2]{LSZ}}]\label{prop-H1}
There exists a universal constant $C>0$ such that for all $t>0$,
\begin{multline}\label{es-2-2d}
\|\nabla u(t, \cdot)\|_{L^2(\R^2)}^2+2\int_{0}^t\|(\sqrt{\rho}\dot{u}, \Delta u, \nabla P)(s, \cdot)\|_{L^2(\R^2)}^2\,ds\\
\leq\|\nabla u_0\|_{L^2(\R^2)}^2\exp( C\rho_*\|\sqrt{\rho_0}u_0\|_{L^2(\R^2)}^2).
\end{multline}
\end{prop}
 \begin{proof}
The main idea is to test the momentum equation against $\dot{u}.$ This gives that
 \begin{align}\label{es-2a1}
\frac{1}{2}\frac{d}{dt}\|\nabla u\|_{L^2(\R^2)}^2+\|\sqrt{\rho}\dot{u}\|_{L^2(\R^2)}^2=\langle\Delta u, u\cdot\nabla u\rangle+ \langle P, \div(u\cdot\nabla u)\rangle.
\end{align}
Using that $\nabla\cdot u=0$ we have  $\Delta u=\nabla^{\perp}(\nabla^{\perp}\cdot u)$ with $\nabla^{\perp}: =(-\partial_2, \partial_1),$ and thus
\begin{align*}
\langle\Delta u, u\cdot\nabla u\rangle=\langle \nabla^{\perp}\cdot u, u\cdot\nabla (\nabla^{\perp}\cdot u)\rangle=0.
\end{align*}
Meanwhile, using  $\nabla\cdot u=0$ again,   by  the  well-known duality between   ${\rm BMO}(\R^2)$ and the Hardy space $\mathcal{H}^1(\R^2)$ and the following inequality (see for instance \cite{PLL}), there exists a universal positive constant $C$ such that for all vectors $w, z\in L^2(\R^2; \R^2)$ satisfying $\nabla\cdot w=\nabla^\perp z=0,$
 \begin{equation}\label{ineq-000}
 \|w\cdot z\|_{\mathcal{H}^1(\R^2)}\leq C\|w\|_{L^2(\R^2)}\|z\|_{L^2(\R^2)},
 \end{equation}
 we have
\begin{align*}
|\langle P, \div(u\cdot\nabla u)\rangle|&\leq \sum_{j=1, 2}|\langle P, \partial_j u\cdot\nabla u^j\rangle|\\
&\leq \sum_{j=1, 2}  \|P\|_{{\rm BMO}(\R^2)} \|\partial_j u\cdot\nabla u^j\|_{\mathcal{H}^1(\R^2)}\\
&\leq  \|P\|_{{\rm BMO}(\R^2)} \|\nabla u\|_{L^2(\R^2)}^2\\
&\leq C  \|\nabla P\|_{L^2(\R^2)} \|\nabla u\|_{L^2(\R^2)}^2.
\end{align*}
Notice that in the last inequality we used the embedding $\dot{H}^1(\R^2)\hookrightarrow {\rm BMO}(\R^2)$ with a numerical constant $C$.  

Plugging the above inequalities into \eqref{es-2a1}  and noticing that
\begin{align*}
\|(\Delta u, \nabla P)\|_{L^2(\R^2)}^2\leq \|\rho \dot{u}\|_{L^2(\R^2)}^2\leq \rho_*\|\sqrt{\rho}\dot{u}\|_{L^2(\R^2)}^2,
\end{align*}
we obtain \eqref{es-2-2d} from the energy inequality \eqref{enineq}.
\end{proof}

Notice that this result does not require any of the assumptions \eqref{cond1}, \eqref{cond2} nor \eqref{condition3'}.

 \subsection{The case when $\eqref{cond1}$ or $\eqref{cond2}$ is satisfied}
 \label{subsec.cond12}

\subsubsection{Approximation procedure}

As in the three-dimensional case, we regularize the initial data. We require that the approximate initial density $\rho_0^\eps$ satisfies \eqref{cond1} or \eqref{cond2} whenever $\rho_0$ satisfies \eqref{cond1} or \eqref{cond2}. Therefore, we construct a sequence $\rho_0^\eps$ such that $\rho_0^\eps\geq \rho_0$. To do this we build the approximation on the sequence of balls $B_n$, $n\rightarrow\infty$.

Let $\eps>0$ and $n\in\N$, $n\geq 1$. Let $K_\eps$ denote a non-negative mollifier. We approximate the initial velocity $u_0$ by $u_0^{\eps,n}\in \mathcal{C}^\infty_{0}(B_n)$ in the following way
\begin{equation*}
u_0^{\eps,n}=(u_0\varphi_n)\star K_\eps,
\end{equation*}
where $\varphi_n=\varphi(\cdot/n)$ with $\varphi\in\mathcal C^\infty_{0}(B_1)$ a cut-off function such that $\varphi=1$ on $B_{1/2}$. 
As for the density, we approximate it in the following way on $B_n$
\begin{equation*}
\rho_0^{\eps,n}:=\rho_0\star K_\ep+\|\rho_0\star K_\ep-\rho_0\|_{L^\infty(B_n)}+\frac1ne^{-|x|^2}.
\end{equation*}
Notice that
 \begin{align*}
  \rho_0^{\eps,n}\in \mathcal{C}^\infty(\R^2)~\,\,{\rm with}~\,\,\frac1ne^{-n^2}\leq \rho_0^\eps\leq 2\rho_*\and u_0^{\eps,n}\in \mathcal{C}_0^\infty(\R^2) ~\,\,{\rm with}~\,\,\nabla\cdot u_0^{\eps,n}=0.
  \end{align*}
Moreover, $\|\rho_0\star K_\ep-\rho_0\|_{L^\infty(B_n)}\longrightarrow 0$ when $\eps\rightarrow 0^+$ and $\rho_0^{\eps,n}\geq\rho_0$. Therefore if $\eqref{cond1}$ (resp. \eqref{cond2}) is satisfied for $\rho_0$ it is also satisfied for $\rho_0^{\eps,n}$. 

We consider the solutions $(\rho^{\ep,n},u^{\eps,n})$ to the following mollified version of \eqref{INS}
\begin{equation}\label{approx-INS}
\left\{\begin{aligned}
  &\partial_t \rho^{\eps,n}+\nabla\cdot\big(\rho^{\eps,n} (K_\eps\star u^{\eps,n})\big) =0, \\
& \partial_t (\rho u^{\eps,n}) +\nabla\cdot\big(\rho^{\ep,n} (K_\eps\star u^{\eps,n})\otimes u^{\eps,n}\big)+\nabla P^{\eps,n} =\nu\Delta u^{\eps,n},\\
&\nabla\cdot u^{\eps,n}=0.
 \end{aligned}\right.\tag{INS-$\eps$}
\end{equation}
in the domain $B_n$ with no-slip boundary condition $u^{\eps,n}=0$ on $\partial B_n$ and initial data $(\rho^{\ep,n}_0,u^{\eps,n}_0)$. According to \cite[Theorem 2.6]{PLL}, there exists a smooth global-in-time solution $(\rho^{\eps,n},u^{\eps,n})\in \mathcal C^\infty([0,\infty)\times\overline B_n)$ to \eqref{approx-INS}.

Our objective is now to derive uniform estimates in $\eps$ and $n$ for the approximate solutions $(\rho^{\ep,n},u^{\eps,n})$. For notational simplicity, we drop the superscripts $\eps$ and $n$ below and simply write the solution $(\rho,u)$. Notice that all the estimates below are on $B_n$.\footnote{\label{foot.R2}Notice that as in the previous subsection, the estimates that do not involve second-order space derivatibves or first-order time derivatives of the velocity can be extended to $\R^2$ thanks to the fact that $u^{\eps,n}$ satisfies no-slip boundary conditions on $\partial B_{n}$.}

\subsubsection{Lower-order estimates}
 Here,  we establish some  lower-order  bounds following from condition \eqref{cond1} or \eqref{cond2}  for  $\rho_0.$
 \begin{prop}[$L^\infty_tL^2_x$ bound for the velocity]\label{Prop-L2}
We have
\begin{align}\label{es-lower}
\|u\|_{L^\infty(\R_+; L^2(B_n))}\leq C_*(\|\sqrt{\rho_0}u_0\|_{L^2(\R^2)}+\|\nabla u_0\|_{L^2(\R^2)})
\end{align}
and
\begin{align}\label{e.gronwalest}
\|u\|_{L^4(\R_+; L^\infty(B_n))}\leq C_*\big(\|\sqrt{\rho_0} u_0\|_{L^2(B_n)}+\|\nabla u_0\|_{L^2(B_n)}\big)^2\|\nabla u_0\|_{L^2(\R^2)}^2\exp(\|\sqrt{\rho_0} u_0\|_{L^2(\R^2)}^2),
\end{align}
with  positive constant $C_*$ depending only on the factors in condition \eqref{cond1} or \eqref{cond2}.
\end{prop}

\begin{proof}
Using the Lagrangian flow maps and the fact that they are measure preserving due to incompressibility, $\rho$ satisfies \eqref{cond1} or \eqref{cond2} along the evolution with the same constants $d$, $\delta_0$ and $\|( {1}/{{\eta}})\,\mathbf{1}_{\eta< \delta_0}\|_{L^1}$ in the case when \eqref{cond1} is satisfied, $p$, $d$, $\bar\eta$ and $\|(\bar\eta-\eta)_+\|_{L^p}$ in the case when \eqref{cond2}. 
Then by Proposition \ref{Prop-intp} and  \eqref{es-2-2d}, we   get
\begin{align*}
\|u\|_{L^\infty(\R_+; L^2(B_n))}&\leq C_*(\|\sqrt{\rho}u\|_{L^2(B_n)}+\|\nabla u\|_{L^2(B_n)}) \\
&\leq  C_*(\|\sqrt{\rho_0}u_0\|_{L^2(\R^2)}+\|\nabla u_0\|_{L^2(\R^2)}).
\end{align*}
Using  \eqref{es-lower} and \eqref{es-2-2d},  we   further get
\begin{align*}
& \int_0^\infty\|u(t, \cdot)\|_{ L^\infty(B_n)}^4\,dt\\
&\leq \int_0^\infty\|  u\|_{L^2(B_n)}^2\|\nabla^2 u\|_{L^2(B_n)}^2\,dt\\
&\leq C_*\big(\|\sqrt{\rho_0} u_0\|_{L^2(\R^2)}+\|\nabla u_0\|_{L^2(\R^2)}\big)^2  \|\nabla^2 u\|_{L^2(\R_+; L^2(R_n))}^2\\
&\leq C_*\big(\|\sqrt{\rho_0} u_0\|_{L^2(\R^2)}+\|\nabla u_0\|_{L^2(\R^2)}\big)^2\|\nabla u_0\|_{L^2(\R^2)}^2\exp(\|\sqrt{\rho_0} u_0\|_{L^2(\R^2)}^2).\qedhere
\end{align*}
\end{proof}

\subsubsection{Estimates of the time derivative}
In this step, we  want to  bound time derivatives $\sqrt{\rho t}  {u}_t$ in $L^\infty_{\rm loc}(\R_+; L^2(B_n))$ and $\sqrt{t}\nabla {u}_t$ in $L^2_{\rm loc}(\R_+; L^2(B_n)),$
since it is an important step towards   higher-order spatial estimates for the velocity.

\begin{prop}[time derivative estimates]\label{Prop-timees2d}
For any $T>0,$ there exits a positive constant  $C_{0, *}(T)$ depending only on $C_*, T, \rho_*$ and $ \|\sqrt{\rho_0}u_0\|_{L^2}, \|\nabla u_0\|_{L^2}$ such that
\begin{equation}\label{es-timederivetive2d}
\|\sqrt{\rho t} {u}_t \|_{ L^\infty(0, T;  L^2(B_n))}^2+  \|\nabla( \sqrt{t} {u}_t)\|_{L^2(0, T; L^2(B_n))}^2 \leq C_{0, *}(T).
\end{equation}
\end{prop}
\begin{proof}
The proof is similar to the three-dimensional case.  Recall that by testing the equation \eqref{eq-td} with $\sqrt{t} {u}_t$,  we  are led to
\begin{align}\label{time-es12d}
\frac{1}{2}\frac{d}{dt}\|\sqrt{\rho t} {u}_t\|_{L^2}^2+\|\nabla(\sqrt{t} {u}_t)\|_{L^2}^2=A_1+\dots+A_4
\end{align}
with
\begin{align*}
&A_1=\frac{1}{2}\|\sqrt{\rho }{u}_t\|_{L^2(\R^d)}^2,\\
&A_2=-\langle  t\rho_t , |u_t|^2\rangle,\\
&A_3=- \langle\sqrt{t}\rho_t u\cdot\nabla u, \sqrt{t}{u}_t\rangle,\\
&A_4=-\langle\sqrt t\rho   u_t\cdot\nabla u, \sqrt{t} {u}_t\rangle.
\end{align*}
In order to estimate  $ A_1, A_2, A_3, A_4,$  we proceed as follows.
Noticing that
\begin{align*}
A_1&= \frac{1}{2}\|\sqrt{\rho }(\dot{u}-u\cdot\nabla u) \|_{L^2(B_n)}^2\\
&\leq \|\sqrt{\rho }\dot{u}  \|_{L^2(B_n)}^2+ \rho_*\|u\|_{L^\infty(B_n)}^2\|\nabla u \|_{L^2(B_n)}^2.
\end{align*}
 Using the equation $\partial_t \rho +\nabla\cdot (\rho u)=0,$ we write
\begin{align*}
A_2=-\langle  t \rho u , \nabla(|u_t|^2)\rangle
\end{align*}
and
\begin{align*}
|A_2|\leq& 2  \sqrt{\rho_*}\|u\|_{L^\infty(B_n)}\|\sqrt{\rho t} u_t\|_{L^2(B_n)}\|\nabla(\sqrt{t} u_t)\|_{L^2(B_n)}\\
\leq& \frac{1}{8} \|\nabla(\sqrt{t} u_t)\|_{L^2(B_n)}^2+ C\rho_* \|\sqrt{\rho t} u_t\|_{L^2(B_n)}^2\|u\|_{L^\infty(B_n)}^2.
\end{align*}
Similarly, we write
\begin{align*}
A_3=-\langle  t \rho u  , \nabla[(u\cdot\nabla u)\cdot  u_t]\rangle,
\end{align*}
and  decompose
\begin{align*}
|A_3|&\leq \langle  |t \rho u|  , |\nabla u|^2\, |u_t|\rangle +\langle | t \rho u | ,  |u|\,|\nabla^2 u|\,|u_t|\rangle+ \langle  |t \rho u|  , |u|\, |\nabla u|\, |\nabla u_t|\rangle\\
&=: A_{31}+A_{32}+A_{33}.
\end{align*}
To bound $A_{31}$, we write  by the Gagliardo-Nirenberg inequality \eqref{GN-in}
\begin{align*}
A_{31}&\leq \sqrt{\rho_* T}\|\sqrt{\rho t} u_t\|_{L^2(B_n)}\|\nabla u\|_{L^4(B_n)}^2\|u\|_{L^\infty(B_n)}\\
&\leq  \|\sqrt{\rho t} u_t\|_{L^2(B_n)}^2\|u\|_{L^\infty(B_n)}^2+ \rho_* T \|\nabla u\|_{L^4(B_n)}^4\\
&\leq  \|\sqrt{\rho t} u_t\|_{L^2(B_n)}^2\|u\|_{L^\infty(B_n)}^2+ \rho_* T \|\nabla u\|_{L^2(B_n)}^2\|\nabla^2 u\|_{L^2(B_n)}^2.
\end{align*}
As for $A_{32}, A_{33}$, we have
\begin{align*}
A_{32}&\leq \sqrt{ \rho_*T}\|\sqrt{\rho t}u_t\|_{L^2(B_n)}\|u\|_{L^\infty(B_n)}^2\|\nabla^2 u\|_{L^2(B_n)}\\
&\leq\|\sqrt{\rho t}u_t\|_{L^2(B_n)}^2\|u\|_{L^\infty(B_n)}^4+ \rho_*T\|\nabla^2 u\|_{L^2(B_n)}^2
\end{align*}
and
\begin{align*}
A_{33}&\leq  \rho_*\sqrt{ T} \|u\|_{L^\infty(B_n)}^2\|\nabla u\|_{L^2(B_n)}\|\nabla (\sqrt{t}u_t)\|_{L^2(B_n)}\\
&\leq  \frac{1}{8}\|\nabla( \sqrt{  t} u_t)\|_{L^2(B_n)}^2 + C\rho_*^2 { T} \|u\|_{L^\infty(B_n)}^4\|\nabla u\|_{L^2(B_n)}^2.
\end{align*}
To handle $A_4$,   we need to use Proposition \ref{Prop-intp}. More precisely,  since    condition \eqref{cond1} or \eqref{cond2} is also satisfied by $\rho$,    so the functional inequality \eqref{intp-ineq} (taking $\eta= \sqrt{ t} u_t$) yields that
\begin{align} \label{es-2dw000}
\|\sqrt{t}u_t\|_{L^2(B_n)}\leq C_*\Big( \|\sqrt{\rho t} u_t\|_{L^2(B_n)}+\|\nabla(\sqrt{t}u_t)\|_{L^2(B_n)}\Big).
\end{align}
Then for all $r\in[2, \infty),$   we have by interpolation inequality
\begin{align*}
\|\sqrt{t}u_t\|_{L^r(B_n)}\leq&  C(\|\sqrt{t}u_t\|_{L^2(B_n)}+ \|\sqrt{t}\nabla u_t\|_{L^2(B_n)})\\
\leq&
  C_*( \|\sqrt{\rho t}u_t\|_{L^2(B_n)} +\nabla(\sqrt{t}u_t)\|_{L^2(B_n)}),
\end{align*}
which implies that
\begin{align*}
|A_4|\leq&  \|\sqrt{ \rho t}u_t\|_{L^4(B_n)}^2 \|\nabla u\|_{L^2(B_n)}\\
\leq&  \|\sqrt{\rho t}u_t\|_{L^2(B_n)}^\frac{1}{2}\| \sqrt{\rho t}u_t\|_{L^6(B_n)}^\frac{3}{2} \|\nabla u\|_{L^2(B_n)}\\
\leq&  C_* \rho_*^\frac{3}{4} \|\sqrt{\rho t}u_t\|_{L^2(B_n)}^\frac{1}{2} \big( \|\sqrt{\rho t} u_t\|_{L^2(B_n)}+\|\nabla(\sqrt{t}u_t)\|_{L^2(B_n)}\big)^\frac{3}{2}  \|\nabla u\|_{L^2(B_n)}\\
\leq&    C_* \rho_*^\frac{3}{4} \|\sqrt{\rho t}u_t\|_{L^2(B_n)}^2 \|\nabla u\|_{L^2(B_n)}+   C_*\rho_*^3 \|\sqrt{\rho t}u_t\|_{L^2(B_n)}^2 \|\nabla u\|_{L^2(B_n)}^4.
\end{align*}
Finally, by plugging all the above estimates of $A_i ~(i=1,\ldots\, 4)$ into \eqref{time-es12d},  we get
\begin{align}\label{time-es2d}
 & \frac{d}{dt}\|\sqrt{\rho t} {u}_t\|_{L^2(B_n)}^2+\|\nabla(\sqrt{t} {u}_t)\|_{L^2(B_n)}^2\\
  \leq&  \|\sqrt{\rho t}u_t\|_{L^2(B_n)}^2\left( 1
 +\|u\|_{L^\infty(B_n)}^4+\|\nabla u\|_{L^2(B_n)}+\|\nabla u\|_{L^2(B_n)}^4\right)\notag\\
 &+  C_{0, *}({T})\left(   \|\sqrt{\rho } {u}_t  \|_{L^2(B_n) }^2+  (\|u\|_{L^\infty(B_n)}^4+1)\|\nabla u \|_{L^2(B_n) }^2+( \|\nabla u\|_{L^2(B_n) }^2+1)\|\nabla^2 u\|_{L^2 (B_n)}^2\right).\notag
\end{align}
To conclude, thanks to inequalities \eqref{es-2-2d} and \eqref{e.gronwalest}, it is not difficult to find that  inequality \eqref{time-es2d} can be rewritten as
\begin{align*}
 \frac{d}{dt}\|\sqrt{\rho t} {u}_t\|_{L^2(B_n)}^2+\|\nabla(\sqrt{t} {u}_t)\|_{L^2(B_n)}^2\leq B_{2d}(t)\|\sqrt{\rho t}u_t\|_{L^2(B_n)}^2 +B_{2d}(t)
\end{align*}
for some function $B_{2d}\in L^1(0, T),$   the norm of which  bounded only in terms of    $C_*, T,$ $\rho_*$ and norms $\|\sqrt{\rho_0} u_0\|_{L^2(B_n)}, \|\nabla  u_0\|_{L^2(B_n)}.$

Therefore, Gronwall's lemma implies that for all $t\in[0, T]$
\begin{align*}
\|\sqrt{\rho t} {u}_t\|_{ L^2}^2+\int_0^T \|\nabla(\sqrt{t}   {u}_t)\|_{L^2}^2\,dt\leq \int_0^T B_{2d}(t)\,dt \exp\left(\int_0^T B_{2d}(t)\,dt \right).
\end{align*}
This completes the proof of proposition \ref{Prop-timees2d}.
\end{proof}

\subsubsection{Shift of regularity}

 As a consequence of Proposition \ref{Prop-timees2d},  we will get higher-order estimates for the velocity, via considering the Stokes problem \eqref{stokes-timeweighted}.
We have
\begin{prop}[higher-order estimates]\label{Prop-shift22d}
  Let $T>0.$    For all $r\in[2, \infty)$ and $q\in[2, \frac{2r}{r-2}),$ we have
   \begin{align}\label{e.estt12nabla2u}
\|\nabla^2(\sqrt{t} u)\|_{L^q(0, T; L^r(B_n))}+\|\nabla(\sqrt{t}P)\|_{L^q(0, T; L^r(B_n))}\leq C_{0, *}(T).
\end{align}
Moreover, we have the following key Lipschitz estimate
\begin{align}\label{e.keyestlip}
\int_0^T \|\nabla u(t, \cdot)\|_{L^\infty(\R^2)}\,dt\leq  C_{0, *}(T) T^\frac{1}{3}.
 \end{align}
  In particular, $C_{0, *}(T) T^\frac{1}{3}\to 0$ as $T\to 0.$
\end{prop}
\begin{proof}
 Applying the standard $L^p$ estimates for  system \eqref{stokes-timeweighted},  one gets for all finite $p,$
\begin{align}\label{shift-es12d}
\|\nabla^2(\sqrt{t} u)\|_{ L^p(B_n)}+\|\nabla(\sqrt{t}P)\|_{L^p(B_n)}\leq C(p)\sqrt{\rho_*}\| \sqrt{\rho t} \dot{u}\|_{L^p(B_n)}.
\end{align}
To estimate the material derivative in the right-hand side of above inequality,  we use the Gagliardo-Nirenberg inequality \eqref{GN-in}   and  \eqref{es-2dW00} to write that for all $ {\bar r}\in [2, \infty)$
\begin{align*}
\int_0^T\| \sqrt{\rho t}u_t\|_{L^{ \bar  r}(B_n)}^2\,dt&\leq  \int_0^T {\rho_*}(\| \sqrt{  t }u_t\|_{L^2(B_n)}+\|\nabla (\sqrt{ t }u_t)\|_{L^2(B_n)})^2\,dt\\
&\leq C_* {\rho_*} \int_0^T\left(\| \sqrt{  \rho t }u_t\|_{L^2(B_n)}^2+\|\nabla (\sqrt{ t}u_t)\|_{L^2(B_n)}^2\right)\,dt
\leq C_{0, *}(T),
\end{align*}
 where in the last inequality we used  bounds \eqref{es-timederivetive2d}.

Then,   interpolating above inequality  with  bounds \eqref{es-timederivetive2d}  yields that for all  $r\in[2, \infty)$ and   $q\in[2, r^*)$ with $r^*:= \frac{2r}{r-2},$  one has
 \begin{align}\label{es-shift-es1112d}
 \|\sqrt{\rho t}u_t\|_{L^{q}(0, T; L^r(B_n))}\leq   C_{0, *}(T).
 \end{align}
For the convective term, we notice that the Gagliardo-Nirenberg inequality \eqref{GN-in}, \eqref{es-2-2d}  and   Proposition \ref{Prop-L2} imply that  for all $ \bar r\in [2, \infty)$
\begin{align*}
\| \sqrt{\rho t}u\|_{L^\infty(0, T;  L^{\bar  r} (B_n))}\leq& \sqrt{\rho_* T}\big(\|u\|_{L^\infty(0, T; L^2(B_n))}+\|\nabla u\|_{L^\infty(0, T; L^2(B_n))}\big)
\leq     C_{0, *}(T),
\end{align*}
and for all $\bar r\in [2, \infty)$ and   $\bar m\in [2, \frac{2\bar r}{\bar{r}-2}]$
\begin{align*}
\| \nabla u\|_{L^{\bar m}(0, T;  L^{\bar r} (B_n))}\leq& \|\nabla u\|_{L^\infty(0, T; L^2(B_n))} +\|\nabla u\|_{L^2(0, T; \dot{H}^1(B_n))}
\leq     C_{0, *}(T).
\end{align*}
Thus, for all  $r\in[2, \infty)$ and  $q\in[2, r^*)$, we write according to the above two  inequalities
\begin{align}\label{es-shift-es1122d}
\| \sqrt{\rho t}u\cdot\nabla u\|_{L^q(0, T;  L^{  r} (B_n))} \leq  C_{0, *}(T).
\end{align}
Putting   estimates \eqref{es-shift-es1112d} and \eqref{es-shift-es1122d}   into \eqref{shift-es12d} yields
 \begin{align*}
 \|\nabla^2(\sqrt{t} u)\|_{ L^q(0, T; L^r(B_n))}+\|\nabla(\sqrt{t}P)\|_{L^q(0, T; L^r(B_n))}\leq C_{0, *}(T).
 \end{align*}
 To prove Lipschitz estimate,  we write by  inequality \eqref{GN-in}, \eqref{es-2-2d} and \eqref{e.estt12nabla2u} that\footnote{Notice that the time-weighted estimate  \eqref{e.estt12nabla2u} enables us to gain smallness in time in the final estimate \eqref{e.keyestlip}.}
 \begin{align*}
\int_0^T\|\nabla u\|_{L^\infty(B_n)}\,dt\leq& \int_0^T \|\nabla u\|_{L^2(B_n)}^\frac{1}{3} \|\nabla^2 (\sqrt{t} u)\|_{L^4(B_n)}^\frac{2}{3}\, t^{-\frac{1}{3}}\,dt\\
 \leq &  \|\nabla u\|_{L^\infty(0, T; L^2(B_n))}^\frac{1}{3}   \|\nabla^2 (\sqrt{t} u)\|_{L^2(0, T; L^4(B_n))}^\frac{2}{3}  \|t^{-\frac{1}{3}}\|_{L^{\frac{3}{2}}(0, T)}\\
 \leq & C_{0, *}(T)\, T^{\frac{1}{3}}.
 \end{align*}
 It is easy to notice that $C_{0, *}(T)$ stays bounded as $T\to 0,$ thus we complete the proof of Proposition \ref{Prop-shift22d}.
\end{proof}

\subsection{The case when \eqref{condition3'} is satisfied}\label{s:ffv}

This case corresponds to the far-field vacuum.

\subsubsection{Approximation procedure}
\label{subsubsec.approxffv}

The construction is similar to the one in Section \ref{s:existence2d}. Without  loss of generality, we assume that the initial density $\rho_0$ satisfies $\|\rho_0\|_{L^1(\R^2)}=1,$
which implies that there exists a positive constant $N_0$ such that
\begin{align*}
\int_{B_{N_0}} \rho_0\,dx>\frac{1}{2}.
\end{align*}
We require that the approximate initial density satisfies \eqref{condition3'} whenever the initial density satisfies \eqref{condition3'}. 

Let $\eps>0$ and $n\in\N$, $n\geq 1$. Let $K_\eps$ denote a non-negative mollifier. We approximate the density in the following way on $B_n$
\begin{equation*}
\rho_0^{\eps,n}:=\rho_0\star K_\ep
+\frac1ne^{-|x|^2},
\end{equation*}
Notice that
 \begin{align*}
  \rho_0^{\eps,n}\in \mathcal{C}^\infty(\R^2)~\,\,{\rm with}~\,\,\frac1ne^{-n^2}\leq \rho_0^\eps\leq 2\rho_*\and u_0^{\eps,n}\in \mathcal{C}_0^\infty(\R^2) ~\,\,{\rm with}~\,\,\nabla\cdot u_0^{\eps,n}=0.
  \end{align*}
For $n\in\N$, $n\geq 1$ fixed and for $\eps\rightarrow 0^+$,
\begin{equation}
\left\{\begin{aligned}
  &\int_{B_{N_0}} \rho^{\eps,n}_0\,dx >\frac{1}{4},\\
  &\bar{x}^\alpha \rho_0^{\eps,n}    \rightharpoonup \bar{x}^\alpha\rho_0\quad{\rm weak}*\quad{\rm  in}~~L^{\infty}(\R^2), \\
&\bar{x}^\alpha \rho^{\eps,n}_0 \to  \bar{x}^\alpha \rho_0 \quad{\rm  in}~~L^{p}(\R^2),\quad \forall~p\in[1, \infty).
 \end{aligned}\right.
\end{equation}
The construction of an approximate initial velocity $u_0^{\eps,n}\in \mathcal{C}^\infty_{0}(B_n)$ is inspired by \cite{CK2}. We let $u_0^{\eps,n}\in\mathcal C^\infty(\overline B_n)$ be the solution to 
\begin{align}\label{cond-velocity-approx}
\tag{Approx-compa}
-\Delta u^{\eps, n}_0 +\nabla P_0^{\eps, n}=\sqrt{\rho^{\eps,n}_0}g^{\eps,n},\quad \nabla\cdot u^{\eps, n}_0=0\qquad\mbox{on}\quad B_n,
\end{align}
with no-slip boundary conditions $u^{\eps,n}_0=0$ on $\partial B_n$, with  $g^{\eps,n}:= (g\star K_\eps)\varphi_n$. Notice that we have the following estimate 
\begin{equation}\label{e.cvu0epsn}
\|\nabla(u^{\eps, n}_0-u_0)\|_{L^2(B_n)}\leq\ Cn\|\sqrt{\rho^{\eps,n}_0}g^{\eps,n}-\rho_0g\|_{L^2(B_n)}+n^{-1}\|u_0\|_{L^2(B_n^c)}+\|u_0\|_{H^1(B_n^c)},
\end{equation}
where we used the rescaled trace inequality
$$
n^{-\frac12}\|u_0\|_{L^2(\partial B_n)}+\|u_0\|_{H^\frac12(\partial B_n)}\leq n^{-1}\|u_0\|_{L^2(B_n^c)}+\|u_0\|_{H^1(B_n^c)}.
$$
Furthermore, by the Poincar\'e-Sobolev inequality \cite[estimate II.3.7]{Galdibook}, we have
\begin{align}\label{e.cvu0epsnbis}
\begin{split}
\|u^{\eps, n}_0\|_{L^2(\R^2)}\leq\ &C\|\nabla u^{\eps, n}_0\|_{L^1(\R^2)}=C\|\nabla u^{\eps, n}_0\|_{L^1(B_n)}\\
\leq\ &Cn\|\nabla(u^{\eps, n}_0-u_0)\|_{L^2(B_n)}+C\|\nabla u_0\|_{L^1(\R^2)}\\
\leq\ &Cn^2\|\sqrt{\rho^{\eps,n}_0}g^{\eps,n}-\rho_0g\|_{L^2(B_n)}+C\|\nabla u_0\|_{L^1(\R^2)}.
\end{split}
\end{align}
We rely on \eqref{e.cvu0epsn}, \eqref{e.cvu0epsnbis} and on the facts that\footnote{Notice that here we use the assumption that $u_0$ in addition belongs to $L^2(\R^2)$ so that $u_0\in H^1(\R^2)$.} 
$$
n^{-1}\|u_0\|_{L^2(B_n^c)}+\|u_0\|_{H^1(B_n^c)}\rightarrow 0,\qquad n\rightarrow\infty
$$
and 
$$
Cn^2\|\sqrt{\rho^{\eps,n}_0}g^{\eps,n}-\rho_0g\|_{L^2(B_n)}\rightarrow 0,\qquad \eps\rightarrow 0
$$
when $n$ is fixed, to choose a sequence $(\eps_k,n_k)$ on which $(\rho^{\eps_k,n_k}_0,u_0^{\eps_k,n_k})$ have the ad hoc boundedness and convergence properties. 

Let $k\in\N$. Assume $n_1<\ldots\ <n_{k-1}$ and $\eps_1>\ldots\ >\eps_{k-1}>0$ are constructed. There exists $n_k\in\N$, $n_k>n_{k-1}$ 
\begin{align*}
\|u_0\|_{L^2(B_{n_k}^c)}+\|u_0\|_{H^1(B_{n_k}^c)}
\leq \frac1{2k}.
\end{align*}
Then there exists $0<\ep_{k}<\ep_{k-1}$ such that 
$$
Cn_k^2\|\sqrt{\rho^{\eps_k,n_k}_0}g^{\eps_k,n_k}-\rho_0g\|_{L^2(B_{n_k})}\leq \frac1{2k}.
$$
Hence, for all $k\in\N$, 
\begin{equation*}
\|\nabla(u^{\eps_k, n_k}_0-u_0)\|_{L^2(B_{n_k})}\leq \frac1k,
\end{equation*}
which implies 
\begin{align*}
\|\nabla(u^{\eps_k, n_k}_0-u_0)\|_{L^2(\R^2)}=\ &\|\nabla(u^{\eps_k, n_k}_0-u_0)\|_{L^2(B_{n_k})}+\|\nabla u_0\|_{L^2(B_{n_k}^c)}\\
\leq\ & \frac1k+\|\nabla u_0\|_{L^2(B_{n_k}^c)}\rightarrow 0,\qquad n\rightarrow\infty,
\end{align*}
and\footnote{Remark that here the assumption that $\nabla u_0\in L^1(\R^2)$ plays a key role.} 
$$\|u^{\eps_k,n_k}_0\|_{L^2(\R^2)}\leq C\|\nabla u_0\|_{L^1(\R^2)}+\frac12.$$
Therefore
\begin{equation*}
\|u^{\eps_k, n_k}_0\|_{L^2(\R^2)},\quad\|\sqrt{\rho^{\eps_k, n_k}_0}u^{\eps_k, n_k}_0\|_{L^2(\R^2)}\quad\mbox{and}\quad \|\nabla u^{\eps_k, n_k}_0\|_{L^2(\R^2)}
\end{equation*}
are bounded uniformly in $k$ by 
\begin{equation}\label{e.defA_0}
\mathcal A_0:=\rho_*+\|\rho_0\|_{L^1(\R^2)}+\|\sqrt{\rho_0}u_0\|_{L^2(\R^2)}+\|\nabla u_0\|_{L^1(\R^2)}+\|\nabla u_0\|_{L^2(\R^2)}
\end{equation}
and we have the following convergence properties
\begin{align*}
u_0^{\eps_k, n_k} \to u_0\quad{\rm in}~~\wt{\mathcal{D}}^{1, 2}(\R^2),\quad \sqrt{\rho_0^{\eps_k, n_k}}u_0^{\eps_k, n_k}\, \rightharpoonup\, \sqrt{\rho_0}u_0 \quad{\rm in}~~L^2(\R^2).
\end{align*}
Finally, elliptic estimates for the Neumann boundary value problem imply 
\begin{equation}\label{e.estP_0epsn}
\|\nabla P_0^{\eps_k,n_k}\|_{L^2(B_{n_k})}\leq C\|\sqrt{\rho^{\eps_k,n_k}_0}g^{\eps_k,n_k}\|_{L^2(B_{n_k})}\leq C\|\sqrt{\rho_0}g\|_{L^2(\R^2)}+\frac12,
\end{equation}
with $C$ uniform in $k$, and by \eqref{cond-velocity-approx} and maximal regularity for the stationary Stokes system, we have 
\begin{equation}\label{e.estnabla^2uepsn}
\|\nabla^2 u^{\eps_k,n_k}\|_{L^2(B_{n_k})}\leq C\|\sqrt{\rho^{\eps_k,n_k}_0}g^{\eps_k,n_k}\|_{L^2(B_{n_k})}\leq C\|\sqrt{\rho_0}g\|_{L^2(\R^2)}+\frac12.
\end{equation}

We consider the solutions $(\rho^{k},u^{k})$ to the following mollified version of \eqref{INS}
\begin{equation}\label{INS-D}
\left\{\begin{aligned}
  &\partial_t \rho^{k}+\nabla\cdot\big(\rho^{k} (K_{\eps_k}\star u^{k})\big) =0, \\
& \partial_t (\rho^{k} u^{k}) +\nabla\cdot\big(\rho^{k} (K_{\eps_k}\star u^{k})\otimes u^{k}\big)+\nabla P^{k} =\nu\Delta u^{k},\\
&\nabla\cdot u^{k}=0.
 \end{aligned}\right.\tag{INS-$\eps_k$}
\end{equation}
in the domain $B_{n_k}$ with no-slip boundary condition $u^{k}=0$ on $\partial B_{n_k}$ and initial data 
$$(\rho^{k}_0,u^{k}_0):=(u_0^{\eps_k, n_k},\rho_0^{\eps_k, n_k}).$$ 
According to \cite[Theorem 2.6]{PLL}, there exists a smooth global-in-time solution $(\rho^{k},u^{k})\in \mathcal C^\infty([0,\infty)\times\overline B_{n_k})$ to \eqref{approx-INS}.

Our objective is now to derive uniform estimates in $k$ for the approximate solutions $(\rho^{k},u^{k})$. 
 Notice that all the estimates below are on $B_{n_k}$.\footnote{See comment in Footnote \ref{foot.R2}.}

\subsubsection{First uniform estimates}
Since $(\rho^{k},u^{k})$ belongs to $\mathcal{C}^\infty([0, \infty)\times \overline{B}_{n_k})$ and satisfies the no-slip boundary condition $u^{k}=0$ on $\partial B_{n_k}$, we have all the estimates from the proof of \cite[Theorem 1.1]{LSZ}. 
More precisely, we have the following strengthened version of Proposition \ref{prop-H1} that in addition to gradient estimates provides time-weighted estimates for the convective derivative.
 \begin{prop}[gradient and convective-derivative estimates; {\cite[Lemma 3.2, 3.2 and 3.4]{LSZ}}]\label{Le-2DUes}
 There exists a  positive $C_0$ depending only on $\mathcal A_0$ defined by \eqref{e.defA_0} 
 such that $0\leq \rho^{k}\leq  \rho_*$ and
 \begin{equation}\label{Le-2DUes1}
 \|(\sqrt{\rho^{k} }u^{k}, \nabla u^{k})\|_{L^\infty(\R_+; L^2(B_{n_k}))}^2+ \|(\nabla u^{k}, \sqrt{\rho^{k}}\dot{u}^{k}, \Delta u^{k}, \nabla P^{k})\|_{L^2(\R_+\times B_{n_k})}^2
\leq C_0
 \end{equation}
 and  for all $T>0$ it holds that
 \begin{equation}\label{Le-2DUes222}
\sup_{t\in[0, T]} \|\sqrt{\rho^{k} t} \dot{u}^{k} \|_{  L^2(B_{n_k})}^2+ \int_0^T\| \sqrt{t}\nabla  \dot {u}^{k}\|_{L^2(B_{n_k})}^2\,dt \leq C_0.
  \end{equation}
  Moreover,
  \begin{equation}\label{es-mass}
\inf_{t\in[0, T]}  \int_{B_{N_1}} \rho^{k}(t, x)\,dx\geq \frac{1}{4},
  \end{equation}
  for some positive constant $N_1$ depending only on  $\|\rho_0\|_{L^1(\R^2)}, \|\sqrt{\rho_0}u_0\|_{L^2(\R^2)}, N_0$ and $T.$
 \end{prop}
 
 Remark that \eqref{Le-2DUes1} is a combination of estimates \cite[(3.4), Lemma 3.2 and Lemma 3.3]{LSZ}, \eqref{Le-2DUes222} follows from \cite[Lemma 3.3]{LSZ} and \eqref{es-mass} follows from \cite[(3.24)]{LSZ}.

 At this stage, we want to shift the regularity in order to get  $L^1_tLip_x$ estimate for the velocity. To achieve it, we first need to show the following
 spatial-weighted estimates for the velocity and the density.
 \begin{prop}[weighted estimates for the velocity and the density]\label{Prop-loc}
For any    $T>0,$  there exists a positive constant $C_{0, T}$ depending only  on   $N_0$, $T$ and on $\mathcal A_0$ defined by \eqref{e.defA_0} 
such that
  for all $a\in(2, \infty)$ and $b\in (0, 1]$
   \begin{equation}\label{es-Wvelocity}
\sup_{t\in [0, T]} \|\bar x^{-b} u^{k}(t, \cdot)\|_{ L^{\frac{a}{b}}(B_{n_k})}\leq   C_{0, T},
 \end{equation}
  \begin{align}
\sup_{t\in [0, T]} \|\bar x^{-1} u^{k}(t, \cdot)\|_{ L^2(B_{n_k})}&\leq   C_{0, T},\label{es-locL2}\\
 \|\bar x^{-b} u^{k}\|_{L^{2+b}(0, T;  L^{\infty}(B_{n_k}))}&\leq  C_{0, T}.\label{es-locLinfty}
 \end{align}
Moreover for $\alpha>1$, for all $p\in[1, \infty]$,
  \begin{equation}\label{es-locD}
\sup_{t\in [0, T]} \|\bar x^{\alpha} \rho^{k}(t, \cdot)\|_{L^p(B_{n_k})}\leq C_{0, T}.
 \end{equation}
 \end{prop}
 \begin{proof}
At first, we prove inequality  \eqref{es-locL2}.   It can be derived directly  from classical weighted inequalities and the uniform estimates in Proposition \ref{Le-2DUes}. Indeed, estimate \eqref{e.estprop14} with inequality \eqref{es-mass} and  \eqref{Le-2DUes1} implies that    there exists a positive constant $C$ depending only on   $\rho_*, m, \ell, N_0, T$ and norms $\|\rho_0\|_{L^1(\R^2)}, \|\sqrt{\rho_0}u_0\|_{L^2(\R^2)}$   such that
 \begin{align} \label{es-locL21}
\left(\int_{B_{n_k}} \frac{|u^{k}(t, x)|^m}{ {\langle x \rangle}^2}\, (\ln \langle x \rangle )^{-\ell}\,dx\right)^{1/m} &\leq C (\|\sqrt{\rho^{k}} u^{k}\|_{L^2(B_{n_k})}+\|\nabla u^{k}\|_{L^2(B_{n_k})})\\
 &\leq C (\|\sqrt{\rho_0} u_0\|_{L^2(\R^2)}+\|\nabla u_0\|_{L^2(\R^2)}),\notag
\end{align}
 for all $m\in [2, \infty)$ and $\ell\in (1+\frac{m}{2}, \infty).$
Recalling that  $\bar x=\langle x\rangle (\ln\langle x\rangle)^2$ and taking  $m=2, \ell=4$ in inequality \eqref{es-locL21}, we  easily get  \eqref{es-locL2}.

Taking  $m=4, \ell=4$ in inequality \eqref{es-locL21}, we   find that
\begin{align}\label{es-L4}
 \| \langle x\rangle^{-1} u^{k}(t, \cdot)\|_{L^4(B_{n_k})}&\leq \left(\int_{B_{n_k}} | u^{k} (t, x) |^4 \langle x\rangle^{-2}  (\ln  \langle x\rangle)^{-4} \,dx\right)^\frac{1}{4}\\
 &\leq    C  (\|\sqrt{\rho_0} u_0\|_{L^2(\R^2)}+\|\nabla u_0\|_{L^2(\R^2)}),\notag
\end{align}
 where we used  the fact that
 \begin{align}\label{es-wf1}
 (\ln\langle x\rangle)^{\gamma}\leq C(\gamma) \langle x\rangle,\quad {\rm for~any}~~\gamma\in[1, \infty),~~x\in \R^2.
 \end{align}
The proof of \eqref{es-Wvelocity} be follows similarly by using \eqref{es-locL21} and \eqref{es-wf1} again.

To prove \eqref{es-locLinfty}, we use the Gagliardo-Nirenberg inequality  \eqref{GN-in} and Young's inequality to write that
 \begin{align*}
\| \bar{x}^{-b} {u}^{k} (t, \cdot) \|_{L^\infty(B_{n_k})}&\leq C(\| \bar{x}^{-b} {u}^{k} (t, \cdot) \|_{L^\frac{a}{b}(B_{n_k})} +  \| \nabla(\bar{x}^{-b} {u}^{k} (t, \cdot) )\|_{L^\frac{a}{b}(B_{n_k})})\\
&\leq C  (\| \bar{x}^{-b} {u}^{k} (t, \cdot) \|_{L^\frac{a}{b}(B_{n_k})} +  \| \nabla  {u}^{k} (t, \cdot) \|_{L^\frac{a}{b}(B_{n_k})}).
\end{align*}
 Thus, taking $a=4$, one then has \eqref{es-locLinfty} thanks to inequality \eqref{es-Wvelocity} and the uniform estimate \eqref{Le-2DUes1}.

Now, we can  prove \eqref{es-locD}.  Noticing that  $\bar{x}^\alpha  \rho^{k} $ satisfies
 \begin{equation*}
 \partial_t (\bar{x}^\alpha \rho^{k}) +u^{k}\cdot\nabla(\bar x^\alpha \rho^{k}) =\alpha  (\bar x^\alpha \rho^{k}) u^{k}\cdot\nabla \ln\bar x,
\end{equation*}
 then   the standard $L^p$ estimate for the transport equation gives that
 \begin{align*}
 \|\bar{x}^\alpha \rho^{k}\|_{L^p(B_{n_k})}^p\leq C \|\bar{x}^\alpha \rho_0\|_{L^p(\R^2)}^p \exp( \|  u^{k}\cdot\nabla \ln\bar x\|_{L^1(0, T; L^\infty(B_{n_k}))} ).
 \end{align*}
 Using the fact that
 \begin{align}\label{es-wf2}
 |\nabla\bar x|\leq C  (\ln\langle x\rangle)^{2},\quad  |\nabla\langle x\rangle| \leq C, \quad {\rm for~any}~~x\in \R^2,
 \end{align}
together with   inequalities \eqref{GN-in}, \eqref{Le-2DUes1}, \eqref{es-L4}, one obtains
 \begin{align}
 \int_0^T\|  u^{k}\cdot\nabla \ln\bar x\|_{ L^\infty(B_{n_k})}\,dt&\leq  C \int_0^T\|   \langle x\rangle^{-1} u^{k} \|_{ L^\infty(B_{n_k})}\,dt\label{es-locLinfty1}\\
 &\leq C\int_0^T\|   \langle x\rangle^{-1} u^{k} \|_{ L^4(B_{n_k})}^\frac{1}{2}  \|\nabla(\langle x\rangle^{-1} u^{k} )\|_{ L^4(B_{n_k})}^\frac{1}{2}\,dt\notag\\
  &\leq C\int_0^T\|   \langle x\rangle^{-1} u^{k} \|_{ L^4(B_{n_k})} \|\nabla u^{k} \|_{ L^4(B_{n_k})}^\frac{1}{2}\,dt\notag\\
 &\leq C  (\|\sqrt{\rho_0} u_0\|_{L^2(\R^2)}+\|\nabla u_0\|_{L^2(\R^2)})^\frac{3}{2}\,T^\frac{3}{4}.\notag
 \end{align}
 This finishes the proof of Proposition \ref{Prop-loc}.
 \end{proof}

\subsubsection{Shift of regularity}

\begin{prop}[improved higher-order estimates]\label{Prop-shift22d2}
For any    $T>0$ and  for all $r\in[2, \infty)$ and $q\in[2, \frac{2r}{r-2}),$  we have
   \begin{align}\label{es-shift-2d00}
\|\nabla^2(\sqrt{t} u^{k})\|_{L^q(0, T; L^r(B_{n_k}))}+\|\nabla(\sqrt{t}P^{k})\|_{L^q(0, T; L^r(B_{n_k}))}\leq C_{0, T}
\end{align}
and
 \begin{align}\label{es-shift-0}
\sup_{t\in[0, T]}\left(\|\nabla^2(\sqrt{t} u^{k})\|_{ L^2(B_{n_k})}+\|\nabla(\sqrt{t}P^{k})\|_{ L^2(B_{n_k})}\right)\leq C_{0, T},
\end{align}
where $C_{0,T}$ depends only    on   $N_0$, $T$ and on $\mathcal A_0$ defined by \eqref{e.defA_0}. 
Moreover, we have the following key Lipschitz estimate
\begin{align*}
\int_0^T \|\nabla u^{k}(t, \cdot)\|_{L^\infty(B_{n_k})}\,dt\leq  C_{0, T} T^\frac{1}{3}
 \end{align*}
and $C_{0,T} T^\frac{1}{3}\to 0$ as $T\to 0.$
\end{prop}
\begin{proof}
 Applying the standard $L^p$ estimates for  Stokes problem \eqref{stokes-timeweighted},  we have for all finite $p,$
\begin{align}\label{shift-es2d111}
\|\nabla^2(\sqrt{t} u^{k})\|_{ L^p(B_{n_k})}+\|\nabla(\sqrt{t}P^{k})\|_{L^p(B_{n_k})}\leq C(p) \| \sqrt{  t} \rho \dot{u}^{k}\|_{L^p(B_{n_k})}.
\end{align}
Obviously, inequality \eqref{Le-2DUes222} yields \eqref{es-shift-0} directly.

Similarly to the proof of inequality \eqref{es-locL21}, applying Proposition \ref{Le-Li-Xin}  to $\dot{u}$ implies that  for all $r\in[2, \infty)$   (taking $\ell=2r$)
 \begin{align*} 
\left(\int_{B_{n_k}} \frac{|\sqrt{t}\dot{u}^{k}(t, x)|^r}{ {\langle x \rangle}^2}\, (\ln \langle x \rangle )^{-2r}\,dx\right)^{1/r} &\leq C (\|\sqrt{\rho^{k} t} \dot{u}^{k}\|_{L^2(B_{n_k})}+\|\sqrt{t}\nabla \dot{u}^{k}\|_{L^2(B_{n_k})}).
\end{align*}
Thus by using $\alpha>1,$  one has
 \begin{align*}
 \|\sqrt{ t} \rho \dot{u}^{k}\|_{L^r(B_{n_k})} =\ & \left(\int_{B_{n_k}}   (\bar x^\alpha\rho^{k})^r (\langle x\rangle^{2-\alpha r} (\ln\langle x\rangle)^{2r-2\alpha r}   |\sqrt{ t}\dot{u}^{k}|^r\langle x\rangle^{-2}(\ln\langle x\rangle)^{-2r}\,dx\right)^{1/r}\\
 \leq\ &  C \|\bar x^\alpha\rho^{k}\|_{L^\infty}   (\|\sqrt{\rho^{k} t} \dot{u}^{k}\|_{L^2(B_{n_k})}+\|\sqrt{t}\nabla \dot{u}^{k}\|_{L^2(B_{n_k})}),
 \end{align*}
which together with inequalities \eqref{es-locD}, \eqref{Le-2DUes222} and \eqref{shift-es2d111} yields \eqref{es-shift-2d00} in the case $q=2$.
Further  interpolating with  \eqref{Le-2DUes222} this finally implies \eqref{es-shift-2d00}.

 To prove the Lipschitz estimate,  we write by  inequalities \eqref{GN-in},  \eqref{Le-2DUes1} and \eqref{es-shift-2d00} that
 \begin{align*}
\int_0^T\|\nabla u^{k}\|_{L^\infty(B_{n_k})}\,dt\leq& \int_0^T \|\nabla u^{k}\|_{L^2(B_{n_k})}^\frac{1}{3} \|\nabla^2 (\sqrt{t} u^{k})\|_{L^4(B_{n_k})}^\frac{2}{3}\, t^{-\frac{1}{3}}\,dt\\
 \leq &  \|\nabla u^{k}\|_{L^\infty(0, T; L^2(B_{n_k}))}^\frac{1}{3}   \|\nabla^2 (\sqrt{t} u^{k})\|_{L^2(0, T; L^4(B_{n_k}))}^\frac{2}{3}  \|t^{-\frac{1}{3}}\|_{L^{\frac{3}{2}}(0, T)}\\
 \leq & C_{0, T}\, T^{\frac{1}{3}},
 \end{align*}
 where $C_{0,T}$ remains bounded when $T\rightarrow 0$.
\end{proof}

\subsubsection{Weighted $L^\infty$ estimates}

In order to prove uniqueness, we need the following improved  weighted-estimate for the velocity, with the help of initial condition \eqref{cond-velocity}.

\begin{prop}[almost boundedness of the velocity]\label{Prop-2dW}
For any  $T>0,$ there exits a constant $C_{\#}$ depending only on the norms $\|g\|_{L^2(\R^2)},  \|\nabla u_0\|_{L^2(\R^2)}, \|\Delta u_0\|_{L^2(\R^2)}$ and on $C_{0, T}$ (defined in Proposition \ref{Prop-shift22d2}) such that for any $b\in (0, 1]$,
\begin{equation}\label{es-2dW1}
\sup_{t\in[0, T]}\| \bar{x}^{-b} {u}^{k} (t, \cdot) \|_{L^\infty(B_{n_k})}\leq C_{\#}.
\end{equation}
\end{prop}
\begin{proof}
 At first, using the Gagliardo-Nirenberg inequality \eqref{GN-in} and Young's inequality, one has for any $a\in(2, \infty)$, 
\begin{align}\label{e.barx-bu}
\| \bar{x}^{-b} {u}^{k} (t, \cdot) \|_{L^\infty(B_{n_k})}
&\leq C(\| \bar{x}^{-b} {u}^{k} (t, \cdot) \|_{L^\frac{a}{b}(B_{n_k})} +\|\nabla u^{k}(t,\cdot)\|_{L^\frac{a}{b}(B_{n_k})}).
\end{align}
Let us bound the right-hand-side of \eqref{e.barx-bu} for $\frac ab=4$. The bound for the first term follows simply from \eqref{es-Wvelocity}.

It remains to obtain  estimate for $\nabla u^{k}$ in $L^\infty(0, T; L^4(B_{n_k}))$. Estimate \eqref{es-shift-0} can be used to obtain an estimate away from initial time. The difficulty is that this estimate degenerates near initial time. Our goal is now to get an estimate near $t=0$. In order to achieve this, we need the compatibility condition \eqref{cond-velocity}. 
  From \cite[Lemma 3.3]{LSZ}, we  know that there exists a  positive constant $C$ depending only on on $\mathcal A_0$ defined by \eqref{e.defA_0} such that
\begin{align}\label{es-2dW00}
\frac{d}{dt}\Big(  \|\sqrt{\rho} \dot{u}^k\|_{L^2(B_{n_k})}^2 -2\sum_{j=1, 2}\langle   P,  \partial_j u^k\cdot\nabla (u^k)^j \rangle \Big)+\|\nabla\dot{u}^k\|_{L^2(B_{n_k})}^2\leq  C  \|\sqrt{\rho} \dot{u}^k\|_{L^2(B_{n_k})}^4.
\end{align}
First, one has for all $0\leq \tau\leq t\leq T$
\begin{align}\label{es-2dW2}
\|\sqrt{\rho} \dot{u}^k(\tau, \cdot)\|_{L^2(B_{n_k})}^2= \langle  \rho \dot{u}^k, \dot{u}^k \rangle =\langle   \Delta u^k, \dot{u}^k\rangle    -\langle   \nabla P^k, \dot{u}^k\rangle
 \end{align}
and 
\begin{align*}
\langle   \Delta u^k, \dot{u}^k\rangle&=\langle  ( \Delta u^k-\nabla P_0^{k})/{\sqrt{\rho}} , \sqrt{\rho^k}\dot{u}^k\rangle +\langle  \nabla P_0^{k},u^k\cdot\nabla u^k\rangle \\
&=\langle  ( \Delta u^k-\nabla P_0^{k})/{\sqrt{\rho^k}} , \sqrt{\rho^k}\dot{u}^k\rangle -\sum_{j=1, 2}\langle   P_0^{k},  \partial_j u^k\cdot\nabla (u^k)^j \rangle,
\end{align*}
 and
 \begin{align*}
 \langle \nabla P^k, \dot{u}^k\rangle=   \langle \nabla P^k, u^k\cdot\nabla u^k\rangle =-\sum_{j=1, 2}\langle   P^k,  \partial_j u^k\cdot\nabla (u^k)^j \rangle.
 \end{align*}
Second, using inequality \eqref{ineq-000} we have
 \begin{align}\label{es-2dW3}
 |\langle   P^k-P_0^k,  \partial_j u^k\cdot\nabla (u^k)^j \rangle| &\leq \|P^k-P_0^k\|_{{\rm BMO}(\R^2)}\|\partial_j u^k\|_{L^2(\R^2)}\|\nabla (u^k)^j\|_{L^2(\R^2)}\\
 & \leq  C(\rho_*)\big(\|\sqrt{\rho_0}g\|_{L^2(\R^2)}+\|\sqrt{\rho^k}\dot{u}^k\|_{L^2(\R^2)}\big)  \|\nabla u^k\|_{L^2(\R^2)}^2,\notag
 \end{align}
 where in the last inequality we used the embedding $\dot{H}^1(\R^2) \hookrightarrow{\rm BMO}(\R^2)$, the standard maximal regularity estimates for the stationary Stokes system applied to \eqref{INS-D} and estimate \eqref{e.estP_0epsn} for $P_0^k$.

 Thus by H\"older's and Young's inequalities and \eqref{es-2dW3},  the identity \eqref{es-2dW2} leads to the following estimate
 \begin{multline*}
 \|\sqrt{\rho^k} \dot{u}^k(\tau, \cdot)\|_{L^2(B_{n_k})}^2\leq C\Big( \|(\Delta u^k-\nabla P_0^{k})/{\sqrt{\rho^k}}\|_{L^2(B_{n_k})}^2\\
 +\| {\sqrt{\rho^k}}u^k\cdot\nabla  u^k\|_{L^2(B_{n_k})}^2+\|\nabla u^k\|_{L^2(B_{n_k})}^4\Big).
 \end{multline*}
 Moreover, using H\"older's inequality again, \eqref{GN-in} and the estimate
  \begin{align*}
 \sup_{t\in[0, T]}\| \sqrt{\rho} u^{k}(t,\cdot)\|_{L^4(\R^2)} &= \left(\int_{\R^2}   (\bar x^\alpha\rho)^2 \langle x\rangle^{2-2\alpha } (\ln\langle x\rangle)^{4- 4\alpha  }   | {u}^{k}|^4\langle x\rangle^{-2}(\ln\langle x\rangle)^{-4}\,dx\right)^{1/4}\\ &\leq  C \|\bar x^\alpha\rho\|_{L^\infty(\R^2)}   \big(\|\sqrt{\rho^{k} }  {u}^{k}\|_{L^2(\R^2)}+\| \nabla  {u}^{k}\|_{L^2(\R^2)}\big)\leq C_{0, T},\notag
 \end{align*} 
the above inequality can be further rewritten as
 \begin{multline*}
 \|\sqrt{\rho^k} \dot{u}^k(\tau, \cdot)\|_{L^2(B_{n_k})}^2\leq C \|(\Delta u^k-\nabla P_0^{k})/{\sqrt{\rho^k}}\|_{L^2(B_{n_k})}^2 \\
 + C_{0, T} \|\nabla  u^k\|_{L^2(B_{n_k})} \|\nabla^2 u^k\|_{L^2(B_{n_k})}+\|\nabla u^k\|_{L^2(B_{n_k})}^4.
 \end{multline*}
Therefore, by the fact that $u^k\in C_tH^2$ and by the approximate compatibility condition \eqref{cond-velocity-approx}, we have
\begin{multline}
\limsup_{\tau\to0^+}\|\sqrt{\rho} \dot{u}^k(\tau, \cdot)\|_{L^2(B_{n_k})}^2\\
\leq C(\|g\|_{L^2(\R^2)}^2 + C_{0, T} \|\nabla  u_0\|_{L^2(\R^2)} \|\nabla^2 u_0\|_{L^2(\R^2)}+\|\nabla u_0\|_{L^2(\R^2)}^4):=C_{\#}.
\end{multline}
Notice that we used \eqref{e.estnabla^2uepsn} to bound $\|\nabla^2 u^k\|_{L^2(\R^2)}$ uniformly in $k$. 
By a variant of inequality \eqref{es-2dW3}, we see that
\begin{multline*}
\frac{1}{2}\|\sqrt{\rho} \dot{u}^k(t, \cdot)\|_{L^2(B_{n_k})}^2 -C\|\nabla u^k(t, \cdot)\|_{L^2(B_{n_k})}^4\leq\  \|\sqrt{\rho^k} \dot{u}^k(t, \cdot)\|_{L^2(B_{n_k})}^2\\
  -2 \sum_{j=1, 2}   \langle   P^k,  \partial_j u^k\cdot\nabla (u^k)^j \rangle|(t).
\end{multline*}
 Then, integrating \eqref{es-2dW00} on the time interval $[\tau, t]$ and letting $\tau\to 0^+$ in the resulting  inequality, we obtain  that there exits a positive time $T_{\#}$ depending only on $C_{\#}$ such that
 \begin{align}
 \|\sqrt{\rho^k} \dot{u}^k(t, \cdot)\|_{L^2(B_{n_k})}^2\leq C C_{\#}, \quad{\rm for~all~time}~t\in(0, T_{\#}].
 \end{align}
Using the standard estimates for the Stokes system, we then get for all $t\in(0, T_{\#}]$,
\begin{align*}
\|\Delta u^k(t, \cdot)\|_{L^2(B_{n_k})}^2 +\|\nabla P^k(t, \cdot)\|_{L^2(B_{n_k})}^2\leq C\rho_*  \|\sqrt{\rho^k} \dot{u}^k(t, \cdot)\|_{L^2(B_{n_k})}^2\leq C C_{\#}.
\end{align*}
This together with estimate \eqref{es-shift-0} implies that
 \begin{align}
\sup_{t\in[0, T]}\|\nabla^2 u^k(t, \cdot)\|_{L^2(B_{n_k})}^2  \leq C (C_{\#}+ C_{0, T}).
\end{align}
Finally, from the  Gagliardo-Nirenberg  inequality \eqref{GN-in} and  \eqref{Le-2DUes1},
 \begin{align*}
\sup_{t\in[0, T]} \|\nabla u^k(t, \cdot)\|_{L^4(B_{n_k})} \leq\ & \sup_{t\in[0, T]}(\|\nabla u^k(t, \cdot)\|_{L^2(B_{n_k})} + \|\nabla^2 u^k(t, \cdot)\|_{L^2(B_{n_k})})\\
\leq\ &   C (C_{\#}+ C_{0, T}).
 \end{align*}
This completes the proof of proposition \ref{Prop-2dW}.
\end{proof}

%%%%%%%%%%%%%%%%%%%%%%%%%%%%%%%%%%%%%%%%%%%%%%%%%

%%%%%%%%%%%%%%%%%%%%%%%%%%%%%%%%%%%%%%%%%%%%%%%%%
\section{Proof of the uniqueness results}\label{s:uniqueness}
In this section, we show the uniqueness of the solutions that we constructed in the paper, both in the two- and three-dimensional case. The main difficulty we have to face is that having only bounded  solutions can not  ensure $L^2$ stability for the transport equation. Of course, one may reformulate system \eqref{INS} in Lagrangian coordinates to prove uniqueness. The advantage of doing so is obvious: the density is constant  along the flow. However, motivated by Hoff's paper \cite{Hoff} on the compressible Navier-Stokes equations, it is possible to directly estimate the difference of the densities in $\dot{H}^{-1}$  (see  for example \cite{DM23, DW}). 

  \subsection{The case when $u$ is in $L^2(\R^2)$ or  $L^6(\R^3)$}
  \label{sec.wsucase1}
 \begin{prop}[uniqueness in 2D and 3D under control of a Lebesgue norm of the velocity]\label{Prop-uniqueness}
Consider two finite-energy weak solutions $(\rho, u)$ and $(\bar\rho, \bar u)$ (in the sense of Definition \ref{defweaksolu}) to  system \eqref{INS} corresponding to the same initial data $(\rho_0, u_0)$ satisfying \eqref{initialcond}.
Assume that
\begin{equation}\label{assump-uniqueness}
\left\{\begin{aligned}
&\nabla\bar u\in  L^\infty(0, T; L^2(\R^d))\cap L^1(0, T; L^\infty(\R^d)),\\&\sqrt{t}\dot{\bar u}\in L^2(0, T; L^6(\R^d)),~\nabla (\sqrt{t}\dot{\bar u})\in L^2((0, T)\times \R^d),
\end{aligned}\right.
\end{equation}
and
\begin{itemize}
\item In the two-dimensional case, assume in addition that  $\rho_0$ satisfies condition \eqref{cond1} or \eqref{cond2},

\item In the three-dimensional case,  assume in addition that,
\begin{align}\label{assump-uniqueness3d}
\nabla u\in L^4(0, T; L^2(\R^3)),
\end{align}
\end{itemize}
then $(\rho, u)\equiv (\bar\rho, \bar u)$ on $[0, T]\times \R^d$.\footnote{Let us stress that the uniqueness is in the class of finite-energy weak solutions. Moreover, notice that the result in 2D is of weak-strong uniqueness type.}
\end{prop}

\begin{proof}
We focus on the case when $T$ is small, say $T\leq \frac{1}{2}$. We remark that a standard connectivity argument enables us to prove the  the case of arbitrary large $T$.

\medskip

\noindent{\underline{Step 1: control of the difference of the densities.}}\\
Define the difference $\dr:=\rho-\bar\rho$ and $\du:= u-\bar u,$  then the system for $(\dr, \du)$ reads
\begin{equation}\label{es-D}
\left\{
\begin{aligned}
&\partial_t \dr +\bar u\cdot\nabla\dr +\du\cdot\nabla \rho=0,\\
&\rho(\partial_t\du +u\cdot\nabla \du)+\nabla\delta P- \Delta \du=-\dr\dot{\bar u}-\rho\du\cdot\nabla \bar u,\\
&\nabla\cdot \du=0,\\
&(\dr, \du)|_{t=0}=0.
\end{aligned}\right.
\end{equation}
We perform estimates for $\dr$ in $L^\infty(0, T; \dot{H}^{-1}(\R^d)$ and for $\du$ in $L^2((0, T)\times \R^d ).$  We remark that the general strategy is the same for dimensions $d=2, 3.$ 
To achieve these estimates, we set
\begin{align}\label{def-dr}
\phi:= (-\Delta)^{-1}\dr \quad{\rm so ~that} \quad\|\nabla\phi\|_{L^2(\R^d)}=\|\dr\|_{\dot{H}^{-1}(\R^d)}.
\end{align}
Now, testing the first equation of \eqref{es-D} against $\phi$ yields
\begin{align}\label{es-phi1}
\frac{1}{2}\frac{d}{dt}\|\nabla \phi\|_{L^2(\R^d)}^2\leq |\langle \bar u\cdot\nabla \Delta \phi, \phi\rangle|+ |\langle \du\cdot\nabla \rho, \phi\rangle|.
\end{align}
For the first term in the right-hand-side of \eqref{es-phi1}, using that $\nabla\cdot\bar u=0$ one has
\begin{align*}
\langle \bar u\cdot\nabla \Delta \phi, \phi\rangle= \sum_{1\leq i, j\leq d}\langle \bar u^{i} \partial_{ijj} \phi, \phi\rangle = \sum_{1\leq i, j\leq d}\langle \partial_j \bar{u}^{i} \partial_{j} \phi, \partial_i \phi\rangle
\end{align*}
so that
\begin{align*}
|\langle \bar{u}\cdot\nabla \Delta \phi, \phi\rangle|\leq \|\nabla \bar{u}\|_{L^\infty(\R^d)}\|\nabla\phi\|_{L^2(\R^d)}^2.
\end{align*}
As for the second term in the right-hand-side of \eqref{es-phi1}, noticing that $\nabla\cdot \du=0,$ then
\begin{align*}
\langle \du\cdot\nabla \rho, \phi\rangle=-\langle \du\cdot\nabla \phi, \rho\rangle,
\end{align*}
  and thus
  \begin{align*}
 |\langle \du\cdot\nabla \rho, \phi\rangle|\leq \sqrt{\rho_*}\|\sqrt{\rho}\du\|_{L^2(\R^d)}\|\nabla \phi\|_{L^2(\R^d)}.
  \end{align*}
Collecting the above two estimates into \eqref{es-phi1}, we discover that
\begin{align*}
\|\nabla\phi(t, \cdot)\|_{L^2(\R^d)}\leq \int_0^t \|\nabla \bar u(s, \cdot)\|_{L^\infty(\R^d)} \|\nabla \phi(s, \cdot)\|_{L^2(\R^d)}\,ds +\sqrt{\rho_*}\int_0^t \|\sqrt{\rho}\du(s, \cdot)\|_{L^2(\R^d)}\,ds.
\end{align*}
Hence, using \eqref{def-dr} and denoting
$$D(t):=\sup_{0<s\leq t} \,s^{-\frac{1}{2}}\|\dr (s, \cdot)\|_{\dot{H}^{-1}(\R^d)},$$
we get after using Young's inequality and Gronwall's lemma, for all $t\in [0, T],$
\begin{align}\label{es-deltarho}
D(t)\leq \sqrt{\rho_*}\|\sqrt{\rho}\du\|_{L^2((0, t)\times\R^d)} \exp(\|\nabla\bar u\|_{L^1(0, T; L^\infty(\R^d))}).
\end{align}

\medskip

\noindent{\underline{Step 2: duality argument.}}\\
At this stage,  in order to control the difference   $\sqrt{\rho}\du$ in $L^2((0, t)\times \R^d)$, we introduce the solution $v$ to the following linear backward parabolic system:
 \begin{equation}\label{eq-backward}
\left\{
\begin{aligned}
&\rho(\partial_t v + u\cdot\nabla v) +\nabla Q+ \Delta v=\rho\du,\\
&\nabla\cdot v=0,\\
&v|_{t=T}=0.
\end{aligned}\right.
\end{equation}
Let us first claim an important a priori estimate for the above system:
\begin{multline}\label{apriori-bw}
\sup_{t\in[0, T]} \|\left(\sqrt{\rho}v,  \nabla v\right)(t, \cdot)\|_{L^2(\R^d)}^2 +\int_0^T \left( \|\nabla v(t,\cdot  )\|_{L^2}^2+ \|(\sqrt{\rho} \dot{v},  \nabla^2 v, \nabla Q)(t,\cdot  )\|_{L^2}^2\right)\,dt\\
\leq  C \|\sqrt{\rho}\du\|_{L^2((0, T)\times \R^d)}^2,
\end{multline}
that we prove below. Solving the above problem is not included in the classical theory for the linear evolutionary Stokes system, since the coefficients are rough and may vanish.  It is not hard to find that if $(\rho, u)$ are regular with $\rho$ bounded from below away from zero then the existence issue can be solved thanks to the a priori estimate \eqref{apriori-bw}. In our setting,  the existence  may be established by a regularizing procedure of $(\rho, u),$  after using estimate \eqref{apriori-bw}, as in the proof of the existence part of Theorem \ref{thm2d} and \ref{thm3d}.

For the moment, we assume that estimate \eqref{apriori-bw} is satisfied. Testing the   equation of $\du$ in system \eqref{es-D} by  $v$ yields that
\begin{align}\label{es-bw-du}
\|\sqrt{\rho}\du\|_{L^2((0, T)\times \R^d)}^2\leq\int_0^T |\langle \dr \dot{\bar u}, v\rangle|\,dt + \int_0^T|\langle \rho\du\cdot\nabla\bar u, v\rangle|\,dt.
\end{align}
We have,  by H\"{o}lder's inequality
\begin{multline*}
\int_0^T |\langle \dr \dot{\bar u}, v\rangle|\,dt \leq   \|t^{-1/2}\dr \|_{L^\infty(0, T; \dot{H}^{-1}(\R^d))} \|\sqrt{t}\nabla (\dot{\bar u}\cdot v)\|_{L^1(0, T; L^2(\R^d))}\\
\qquad\leq D(T) (\|\nabla (\sqrt{t}\dot{\bar u})\|_{L^2((0, T)\times \R^d)} \|v\|_{L^2(0, T; L^\infty(\R^d))}+ \|\sqrt{t}\dot{\bar u}\|_{L^2 (0, T; L^6(\R^d))}\|\nabla v\|_{L^2(0, T; L^3(\R^d))}),
\end{multline*}
and
\begin{align*}
\int_0^T|\langle \rho\du\cdot\nabla\bar u, v\rangle|\,dt\leq \sqrt{\rho_*}\|\sqrt{\rho}\du\|_{L^2((0, T)\times \R^d)}\|\nabla \bar u\|_{L^\infty(0, T; L^2(\R^d))}\|v\|_{L^2(0, T; L^\infty(\R^d))}.
\end{align*}
By estimate \eqref{apriori-bw} and Gagliardo-Nirenberg's inequality \cite[Lemma II.3.3]{Galdibook}, we see that in 2D
\begin{align*}
 \|v\|_{ L^4(0, T; L^\infty(\R^2))} &\leq \left(\int_0^T\|  v\|_{L^2(\R^2)}^2\|\nabla^2 v\|_{L^2(\R^2)}^2\,dt\right)^{1/4}\\
&\leq C \|(\sqrt{\rho } v, \nabla v)\|_{L^\infty(0, T; L^2(\R^2))}^{1/2} \|\nabla^2 v\|_{L^2((0, T)\times \R^2)}^{1/2}\\
&\leq C   \|\sqrt{\rho}\du\|_{L^2((0, T)\times \R^2)},
\end{align*}
and in 3D
\begin{align*}
\|v\|_{ L^4(0, T; L^\infty(\R^3))} &\leq \left(\int_0^T\| \nabla  v\|_{L^2(\R^3)}^2\|\nabla^2 v\|_{L^2(\R^3)}^2\,dt\right)^{1/4}\\
&\leq C \|  \nabla v\|_{L^\infty(0, T; L^2(\R^3))}^{1/2} \|\nabla^2 v\|_{L^2((0, T)\times \R^3)}^{1/2}\\
&\leq C   \|\sqrt{\rho}\du\|_{L^2((0, T)\times \R^3)}.
\end{align*}
Moreover, one has (recall that $T\leq 1/2$)
\begin{align*}
\|\nabla v\|_{L^2(0, T; L^3(\R^d))}&\leq C \|\nabla v\|_{L^2((0, T)\times \R^d)}^{1-\frac{d}{6}}\|\nabla^2 v\|_{L^2((0, T)\times \R^d)}^\frac{d}{6}\\
&\leq C T^{\frac{1}{2}-\frac{d}{12}}  \|\nabla v\|_{L^\infty((0, T; L^2(\R^d))}^{1-\frac{d}{6}}\|\nabla^2 v\|_{L^2((0, T)\times \R^d)}^\frac{d}{6}\\
&\leq C T^\frac{1}{3}\|\sqrt{\rho}\du\|_{L^2((0, T)\times \R^d)}.
\end{align*}
So that one can conclude that
\begin{align*}
 \|v\|_{ L^2(0, T; L^\infty(\R^d))}   +\|\nabla v\|_{L^2(0, T; L^3(\R^d))}
 \leq  CT^\frac{1}{4}  \|\sqrt{\rho}\du\|_{L^2((0, T)\times \R^d)}.
 \end{align*}
Taking these estimates into \eqref{es-bw-du} and using assumption \eqref{assump-uniqueness}, we get
\begin{align*}
\|\sqrt{\rho}\du\|_{L^2((0, T)\times \R^d)}^2 \leq CT^\frac{1}{3}(D(T) \|\sqrt{\rho}\du\|_{L^2((0, T)\times \R^d)} + \|\sqrt{\rho}\du\|_{L^2((0, T)\times \R^d)}^2).
\end{align*}
This together with inequality \eqref{es-deltarho} implies that if $T$ is small enough then
\begin{align*}
 \sqrt{\rho}\du  \equiv 0\and D(t)\equiv0\quad{\rm on}~[0, T].
\end{align*}
Now, since we already know that $ \sqrt{\rho}\du$ and $\dr$ are zero on $[0, T],$  a direct $L^2$ estimate for the equation of $\du$ in system \eqref{es-D} gives directly that
\begin{align}
\int_0^t \|\nabla \du(s, \cdot)\|_{L^2(\R^d)}^2\,ds=0.
\end{align}
In the 2D case, using again the condition \eqref{cond1} or \eqref{cond2} together with Proposition \ref{Prop-intp}, one gets $\du\equiv0$ on $[0, T].$ In the 3D case, one gets same conclusion by  recalling that $\du \in L^2(\R_+; L^6(\R^3)).$

\medskip

\noindent{\underline{Step 3: proof of estimate \eqref{apriori-bw}.}}\\
It remains to prove estimate \eqref{apriori-bw}. Actually, the proof  is  similar to Proposition \ref{prop-H1}.  To achieve it, we proceed as follows.  At first, testing the equation by $v,$ one finds that
\begin{multline}\label{bw-es-L2}
\sup_{t\in(0, T)} \|\sqrt{\rho}v(t, \cdot)\|_{L^2(\R^d)}^2 +  \int_0^T\|\nabla v(t,\cdot  )\|_{L^2(\R^d)}^2\,dt\\
\leq  \|\sqrt{\rho}\du\|_{L^2((0, T)\times \R^d)}^2+ T\sup_{t\in(0, T)} \|\sqrt{\rho}v(t, \cdot)\|_{L^2(\R^d)}^2.
\end{multline}
Secondly, in order to obtain the higher-order estimates of $v$, we apply the standard maximal regularity estimates for the stationary Stokes system to \eqref{eq-backward} to get
that for any $p\in(1, \infty),$
\begin{equation}\label{bw-elliptic-1}
 \|\Delta v\|_{L^p(\R^d)}+\|\nabla Q\|_{L^p(\R^d)}\leq C(p)\|\rho\dot{v}\|_{L^p(\R^d)}\leq C(p) \sqrt{\rho_*}\|\sqrt{\rho}\dot{v}\|_{L^p(\R^d)}.
\end{equation}
Testing the first equation of \eqref{eq-backward} by $\dot{v}$ gives that
\begin{align*}
 \|\sqrt{\rho} \dot{v}(t,\cdot  )\|_{L^2(\R^d)}^2+ \langle \Delta v, \dot{v}\rangle+\langle\nabla Q, \dot{v}\rangle=\langle\rho\du, \dot{v}\rangle.
\end{align*}
Since
\begin{align*}
\langle \Delta v, \dot{v}\rangle=-\frac{d}{dt}\|\nabla v\|_{L^2}^2+ \sum_{1\leq i, j, k\leq d}\langle  \partial_{kk} v^i,   u^j \partial_j v^i\rangle
\end{align*}
and
\begin{align*}
\langle  \partial_{kk} v^i,   u^j \partial_j v^i\rangle= -\langle  \partial_{k} v^i,   \partial_k u^j \partial_j v^i\rangle.
\end{align*}
thus
\begin{align}\label{es-bw-1}
-\frac{d}{dt}\|\nabla v\|_{L^2}^2+   \|\sqrt{\rho} \dot{v}(t,\cdot  )\|_{L^2(\R^d)}^2=\langle \rho\du, \dot{v}\rangle-\langle\nabla Q, \dot{v}\rangle+
\sum_{1\leq i, j, k\leq d}\langle  \partial_{k} v^i,   \partial_k u^j \partial_j v^i\rangle.
\end{align}
Note that in the 2D case, we have by the Gagliardo-Nirenberg inequality \eqref{GN-in}
\begin{align}
\sum_{1\leq i, j, k\leq d}|\langle  \partial_{k} v^i,   \partial_k u^j \partial_j v^i\rangle|\leq& C \|\nabla v\|_{L^4(\R^2)}^2\|\nabla u\|_{L^2(\R^2)}\label{es-bw-11}\\
\leq& C\|\nabla v\|_{L^2(\R^2)}\|\nabla^2 v\|_{L^2(\R^2)} \|\nabla u\|_{L^2(\R^2)}\notag\\
\leq&  \frac{1}{ C{\rho_*}} \|\nabla^2 v\|_{L^2(\R^2)}^2 + C\rho_*   \|\nabla u\|_{L^2(\R^2)}^2\|\nabla v\|_{L^2(\R^2)}^2,\notag
\end{align}
while in the 3D case,
\begin{align}
\sum_{1\leq i, j, k\leq d}|\langle  \partial_{k} v^i,   \partial_k u^j \partial_j v^i\rangle|\leq& C \|\nabla v\|_{L^4(\R^3)}^2\|\nabla u\|_{L^2(\R^3)} \label{es-bw-12}\\
\leq& C\|\nabla v\|_{L^2(\R^3)}^\frac{1}{2}\|\nabla^2 v\|_{L^2(\R^3)}^\frac{3}{2} \|\nabla u\|_{L^2(\R^3)}\notag\\
\leq&   \frac{1}{C{\rho_*}} \| \nabla^2 v\|_{L^2(\R^3)}^2 + C\rho_*^3    \|\nabla u\|_{L^2(\R^3)}^4\|\nabla v\|_{L^2(\R^3)}^2.\notag
\end{align}
Meanwhile, using that $\nabla\cdot v=0$, we write
\begin{align*}
\langle\nabla Q, \dot{v}\rangle=-\langle Q,  \nabla\cdot (u\cdot\nabla v)\rangle=-\sum_{1\leq i, j\leq d} \langle  Q,  \partial_i u^j \partial_j v^i\rangle=-\sum_{1\leq i\leq d}\langle  Q,  \partial_i u\cdot  \nabla  v^i\rangle.
\end{align*}
Therefore in the  2D case,   inequality \eqref{ineq-000} implies
\begin{align}
|\langle\nabla Q, \dot{v}\rangle|&\leq  C\|Q\|_{{\rm BMO}(\R^2)} \|\nabla u\|_{L^2(\R^2)}\|\nabla v\|_{L^2(\R^2)}\label{es-bw-2}\\
&\leq C \|\nabla Q\|_{L^2(\R^2) }\|\nabla u\|_{L^2(\R^2)}\|\nabla v\|_{L^2(\R^2)}\notag\\
&\leq   \frac{1}{C{\rho_*}}  \|\nabla Q\|_{L^2(\R^2) }^2+ C\rho_*
\|\nabla u\|_{L^2(\R^2)}^2\|\nabla v\|_{L^2(\R^2)}^2.\notag
\end{align}
As for the 3D case, one has
\begin{align}
|\langle\nabla Q, \dot{v}\rangle| &\leq C \| Q\|_{L^6(\R^3) }\|\nabla u\|_{L^2(\R^3)}\|\nabla v\|_{L^3(\R^3)}\label{es-bw-3}\\
&\leq C \|   Q\|_{\mathcal{D}^{1, 2}(\R^3) }\|\nabla u\|_{L^2(\R^3)}\|\nabla v\|_{L^2(\R^3)}^\frac{1}{2} \|\nabla^2 v\|_{L^2(\R^3)}^\frac{1}{2} \notag\\
&\leq  \frac{1}{C{\rho_*}} (\|   \nabla^2 v\|_{L^2}^2 + \|   Q\|_{\mathcal{D}^{1, 2}(\R^3) }^2) + C\rho_*^3  \|\nabla u\|_{L^2(\R^3)}^4      \|\nabla v\|_{L^2(\R^3)}^2.\notag
\end{align}
It is easy to find that
\begin{align*}
|\langle\rho\du, \dot{v}\rangle|\leq \|\sqrt{\rho}\du\|_{L^2(\R^d)}\|\sqrt{\rho}\dot{v}\|_{L^2(\R^d)}.
\end{align*}
Putting the above estimate and \eqref{bw-elliptic-1},  \eqref{es-bw-11}-\eqref{es-bw-3} into \eqref{es-bw-1}, we conclude after applying Gronwall's lemma  that, in the 2D case
\begin{multline*}
\sup_{t\in(0, T)} \|\nabla v(t, \cdot)\|_{L^2(\R^2)}^2  +   \int_0^T\|(\sqrt{\rho}\dot{v},  \nabla^2 v, \nabla Q)(t,\cdot  )\|_{L^2(\R^2)}^2\,dt\\
\leq     \|\sqrt{\rho}\du\|_{L^2((0, T)\times \R^2)}^2 \exp(C\rho_* \|\sqrt{\rho_0}u_0\|_{L^2(\R^2)}^2)
\end{multline*}
and in the 3D case
\begin{multline*}
\sup_{t\in(0, T)} \|\nabla v(t, \cdot)\|_{L^2(\R^3)}^2  +   \int_0^T\|(\sqrt{\rho}\dot{v},  \nabla^2 v, \nabla Q)(t,\cdot  )\|_{L^2(\R^3)}^2\,dt\\
\leq    \|\sqrt{\rho}\du\|_{L^2((0, T)\times \R^3)}^2 \exp(C\rho_*^3\|\nabla u\|_{L^4(0, T; L^2(\R^3))}^4).
\end{multline*}
Finally, these estimates together with estimate \eqref{bw-es-L2} prove \eqref{apriori-bw}. This completes the proof of Proposition \ref{Prop-uniqueness}.
\end{proof}

 %%%%%%%%%%%%%%%%%%%%%%%%%%%%%%%%%%%%%%%%%%%%%%%%%

\subsection{The case of two-dimensional far-field vacuum}
\label{subsec.2dfarfieldunique}
\begin{prop}[uniqueness in 2D in the case of far-field vacuum]\label{Prop-uniqueness1}
Let $d=2$. Consider two finite-energy weak solutions $(\rho, u)$ and $(\bar\rho, \bar u)$ (in the sense of Definition \ref{defweaksolu}) to  system \eqref{INS} corresponding to the same initial data $(\rho_0, u_0)$ satisfying \eqref{initialcond}.   Assume  in addition  that\footnote{Here  as in the whole paper $\alpha>1$, see \eqref{condition3'}.} $\rho\bar x^\alpha\in L^\infty(0, T;L^1(\R^2))\cap L^\infty((0,T)\times\R^2)$   and
\begin{equation}\label{assump-uniqueness1}
\left\{\begin{aligned}
& \bar x^{-1} \bar u\in L^\infty(0, T; L^2(\R^2)),  ~~\bar x^{-1}\bar u,~ \bar x^{-1}   u \in L^\infty((0, T)\times \R^2), ~~ \langle x\rangle^{-1}\bar u\in L^1(0, T; L^\infty(\R^2)),\\
&\nabla\bar u\in  L^\infty(0, T; L^2(\R^2))\cap L^1(0, T; L^\infty(\R^2)),\\
&\nabla (\sqrt{t}\dot{\bar u})\in L^2((0, T)\times \R^2), \sqrt{t}\dot{\bar u} \bar x^{-\beta}\in L^2(0, T; L^6(\R^2))~~{\rm for ~some }~\beta \in \Big(\frac{1}{3}, \frac{1}{2}\Big).
\end{aligned}\right.
\end{equation}
Then $(\rho, u)\equiv (\bar\rho, \bar u)$ on $[0, T]\times \R^2$.\footnote{Notice that the uniqueness is for finite-energy weak solutions. However, this result is not of weak-strong uniqueness type.}
\end{prop}

\begin{proof}
\noindent{\underline{Step 1: control of the difference of the densities.}}\\
Recall  the system \eqref{es-D} satisfied by the difference. We define\footnote{Notice that in contrast with Subsection \ref{sec.wsucase1}, we work here, in the far-field case, with a weighted version of the difference of the densities.} $\delta\!\varrho:=\bar{x}^\beta\dr$  (note that $\beta< \frac{\alpha}{2}$) and thus get the following  equation
\begin{align}\label{eq-varrho}
 \partial_t  \delta\!\varrho +\bar u\cdot\nabla \delta\!\varrho +\bar{x}^\beta\du\cdot\nabla \rho=\beta\delta\!\varrho\bar u\cdot\nabla \ln\bar{x}.
\end{align}
In contrast to the previous subsection, we define an inhomogeneous version of the test function for the duality proof, namely
\begin{align}\label{def-drr}
\phi:= ({\rm Id}-\Delta)^{-1}\delta\!\varrho\quad{\rm so ~that} \quad\|\phi\|_{H^1(\R^2)}=\|\delta\!\varrho\|_{{H}^{-1}(\R^2)}.
\end{align}
Now, testing the first equation of \eqref{eq-varrho} against $\phi$ yields that
\begin{align}\label{es-vr1}
\frac{1}{2}\frac{d}{dt}\|  \phi\|_{H^1(\R^2)}^2\leq&~ |\langle \bar u\cdot\nabla \Delta \phi, \phi\rangle|+|\langle \bar{x}^\beta \du\cdot\nabla \rho, \phi\rangle|+ \beta|\langle \Delta \phi\bar{u}\cdot\nabla \ln\bar{x}, \phi\rangle|\notag\\
&+ \beta|\langle  \phi\bar{u}\cdot\nabla \ln\bar{x}, \phi\rangle|.
\end{align}

\subsubsection*{First term in the right-hand-side of \eqref{es-vr1}} 
One has
\begin{align}\label{es-U1}
|\langle \bar{u}\cdot\nabla \Delta \phi, \phi\rangle|\leq \|\nabla \bar{u}\|_{L^\infty(\R^2)}\|\nabla\phi\|_{L^2(\R^2)}^2.
\end{align}

\subsubsection*{Second term in the right-hand-side of \eqref{es-vr1}} 
Noticing that $\nabla\cdot \du=0$, we get
\begin{align}\label{e.eqsecondtermes-vr1}
\langle \bar{x}^\beta\du\cdot\nabla \rho, \phi\rangle=-\langle \du\cdot\nabla \phi, \bar{x}^\beta\rho\rangle-\langle \du\cdot\nabla \bar{x}^\beta,  \rho\phi\rangle.
\end{align}
By H\"older's inequality
  \begin{align}
 |\langle \du\cdot\nabla  \phi, \bar{x}^\beta \rho\rangle|&\leq  \|\sqrt{\rho}\du\|_{L^2(\R^2)}\|\nabla \phi\|_{L^2(\R^2)}\|\sqrt{\rho}\bar{x}^\beta\|_{L^\infty(\R^2)}\label{es-U2}\\
 &\leq  \|\sqrt{\rho}\du\|_{L^2(\R^2)}\|\nabla \phi\|_{L^2(\R^2)}(1+ \| {\rho}\bar{x}^{2\beta}\|_{L^\infty(\R^2)}).\notag
  \end{align}
For the second term in the right-hand-side of \eqref{e.eqsecondtermes-vr1} we use the following inequality from \cite[Theorem 1.1]{MR} which is valid for any ${\rm BMO}(\R^2)$ function $f$ with compact support and $g\in L^1(\R^2)\cap L^\infty(\R^2)$,
  \begin{align}\label{ineq-Hardy}
 |\langle f, g\rangle| \leq   C \|f\|_{{\rm BMO}(\R^2)} \, \|g\|_{L^1(\R^2)}\left(|\ln\|g\|_{L^1(\R^2)}|+\ln (e+\|g\|_{L^\infty(\R^2)})\right),
  \end{align}
  with the functions $f=\phi,~g= \rho\delta u\cdot\nabla \bar{x}^\beta.$ 
In particular, we use that $t^{-\frac{1}{2}}\phi\in L^{\infty}(0, T; H^1(\R^2))$ and the fact that the space $\mathcal{D}(\R^2)$ of smooth compactly supported functions on $\R^2$ is dense in $H^1(\R^2).$   
Noticing
  \begin{align*}
  \|\rho\du\cdot\nabla \bar{x}^\beta\|_{L^1(\R^2)}&\leq C \|\sqrt{\rho}\du\|_{L^2(\R^2)} \|\sqrt{\rho}\bar{x}^{\beta-1}(\ln\langle x\rangle)^2\|_{L^2(\R^2)}\\
  &\leq C \|\sqrt{\rho}\du\|_{L^2(\R^2)} \| {\rho}\bar{x}^{2\beta}\|_{L^1}^\frac{1}{2},
  \end{align*}
  and that
\begin{align*}
\|{\rho}\du\cdot\nabla\bar{x}^\beta\|_{L^\infty(\R^2)}&\leq   \|{\rho}\bar{x}^{\beta} (\ln\langle x\rangle)^2 \|_{L^\infty(\R^2)} \|\bar{x}^{-1} \du\|_{L^\infty(\R^2)}\\
&\leq C\|{\rho}\bar{x}^{2\beta}  \|_{L^\infty(\R^2)}
( \|\bar{x}^{ -1} u \|_{L^\infty(\R^2)}+  \|\bar{x}^{ -1}  \bar u \|_{L^\infty(\R^2)}),
\end{align*}
we see that estimate \eqref{ineq-Hardy} yields 
\begin{align*}
  |\langle \du\cdot\nabla \bar{x}^\beta,  \rho\phi\rangle|&\leq C \|  \phi\|_{{\rm BMO}(\R^2)}  \|\sqrt{\rho}\du\|_{L^2(\R^2)}  ( |\ln  \|\rho\du\cdot\nabla \bar{x}^\beta\|_{L^1(\R^2)} |+ C   )\\
 & \leq C \|\nabla \phi\|_{L^2(\R^2)}  \|\sqrt{\rho}\du\|_{L^2(\R^2)},
\end{align*}
where in order to get the last inequality, we used  the embedding $\dot{H}^1(\R^2)\hookrightarrow {\rm BMO}(\R^2)$ and that
\begin{align*}
 \|\rho\du\cdot\nabla \bar{x}^\beta\|_{L^1(\R^2)}&\leq
    \|{\rho}\bar{x}^{\beta} (\ln\langle x\rangle)^2 \|_{L^1(\R^2)} \|\bar{x}^{-1} \du\|_{L^\infty(\R^2)}\\
&\leq C\|{\rho}\bar{x}^{2\beta}  \|_{L^1(\R^2)}
( \|\bar{x}^{ -1} u \|_{L^\infty(\R^2)}+  \|\bar{x}^{ -1}  \bar u \|_{L^\infty(\R^2)})\leq C.
\end{align*}
These results  combined with inequality \eqref{es-U2} imply that
\begin{align}\label{es-U3}
   | \langle \bar{x}^\beta\du\cdot\nabla \rho, \phi\rangle|\leq C \|\nabla \phi\|_{L^2(\R^2)}  \|\sqrt{\rho}\du\|_{L^2(\R^2)}.
\end{align}
\subsubsection*{Third term in the right-hand-side of \eqref{es-vr1}} 
We rewrite it into
 \begin{align}\label{e.estthirdtermes-vr1}
| \langle \Delta\phi\bar{u}\cdot\nabla \ln\bar{x}, \phi\rangle |\leq | \langle \nabla\phi\cdot\nabla(\bar{u}\cdot\nabla\ln \bar{x}), \phi\rangle |+ | \langle |\nabla\phi|^2, \bar{u}\cdot\nabla \ln\bar{x}\rangle |.
 \end{align}
Similarly, 
as above one has
\begin{align*}
| \langle \nabla\phi\cdot\nabla(\bar{u}\cdot\nabla \ln\bar{x}), \phi\rangle |\leq& \|\nabla \phi\|_{L^2(\R^2)} \|\nabla \phi \cdot\nabla (\bar{u}\cdot\nabla \ln\bar{x})\|_{\mathcal{H}^1(\R^2)}.
\end{align*}
We now rely on \eqref{ineq-Hardy} with $f:=\nabla \phi \cdot\nabla (\bar{u}\cdot\nabla \ln\bar{x})$ to estimate the right-hand-side in the previous estimate. 
Noticing  from inequalities \eqref{es-wf1} and \eqref{es-wf2} that
\begin{align*}
 \|\nabla \phi \cdot\nabla (\bar{u}\cdot\nabla \ln\bar{x})\|_{L^1(\R^2)}&\leq   \|\nabla \phi\|_{L^2(\R^2)}(\|\nabla \bar{u} \nabla \ln\bar x\|_{L^2(\R^2)}+ \|\bar{u}\langle x\rangle^{-2}\|_{L^2(\R^2)}) \\
 &  \leq C\|\nabla \phi\|_{L^2(\R^2)}(\|\nabla \bar{u}  \|_{L^2(\R^2)}+ \|\bar{u} \bar x^{-1}\|_{L^2(\R^2)})
\end{align*}
 and  additionally using the Gagliardo-Nirenberg inequality \eqref{GN-in} we have \begin{align*}
 \|\nabla \phi \cdot\nabla (\bar{u}\cdot\nabla \ln\bar{x})\|_{L^\infty(\R^2)}&\leq  C\|\nabla \phi\|_{L^\infty(\R^2)} (\|\nabla \bar u\|_{L^\infty(\R^2)}+ \|\bar u \bar{x}^{-1}\|_{L^\infty(\R^2)})\\
 &\leq    C (\|\nabla \phi\|_{L^2 (\R^2)} + \|\bar{x}^\beta\dr\|_{L^4(\R^2)})(\|\nabla \bar u\|_{L^\infty(\R^2)}+ \|\bar u \bar{x}^{-1}\|_{L^\infty(\R^2)}).
 \end{align*}
Thus one has
\begin{align}\label{es-U4}
  | \langle \nabla\phi\cdot\nabla(\bar{u}\cdot\nabla \ln\bar{x}), \phi\rangle |\leq C \|\nabla \phi\|_{L^2(\R^2)}^2 (|\ln \|\nabla \phi\|_{L^2(\R^2)}|+   \|\nabla \bar u\|_{L^\infty(\R^2)}+ \|\bar u \bar{x}^{-1}\|_{L^\infty(\R^2)} +1).
\end{align}
For the remaining term in \eqref{e.estthirdtermes-vr1}, notice that
 \begin{align*}
 | \langle |\nabla\phi|^2, \bar{u}\cdot\nabla \ln\bar{x}\rangle |\leq \|\nabla \phi\|_{L^2(\R^2)}^2\|\bar u \langle x\rangle^{-1}\|_{L^\infty(\R^2)}.
 \end{align*}
\subsubsection*{Fourth term in the right-hand-side of \eqref{es-vr1}} 
By the same reasoning as above, we have
\begin{align}\label{es-U5}
|\langle  \phi\bar{u}\cdot\nabla \ln\bar{x}, \phi\rangle|&\leq \|\phi\|_{L^2(\R^2)}^2 \|\bar{u}\cdot\nabla \ln{\bar{x}}\|_{L^\infty(\R^2)}\notag\\
&\leq C\|\phi\|_{L^2(\R^2)}^2 \|\bar{u}    \langle x\rangle^{-1}\|_{L^\infty(\R^2)}.
\end{align}

\subsubsection*{Final estimate of \eqref{es-vr1}} 
Putting the above estimates \eqref{es-U1}, \eqref{es-U3}, \eqref{es-U4}, \eqref{es-U5} together into \eqref{es-vr1}, we get
\begin{align*}
    \frac{d}{dt}\|  \phi(t, \cdot)\|_{H^1(\R^2)}^2\leq& C \|  \phi\|_{H^1(\R^2)}^2(|\ln \| \phi\|_{H^1(\R^2)}|+\|\nabla\bar u\|_{L^\infty(\R^2)} + \|\bar u \langle x\rangle^{-1}\|_{L^\infty(\R^2)} +1 ) \\
   &\quad + \| \phi\|_{H^1(\R^2)}\|\sqrt{\rho}\du\|_{L^2(\R^2)}.
\end{align*}
Hence,  denoting
$$\bar{D}(t):=\sup_{0<s\leq t} \,s^{-\frac{1}{2}}\|\delta\!\varrho (s, \cdot)\|_{ {H}^{-1}(\R^2)},$$
 we further get that for all $t\in [0, T],$
\begin{align} \label{es-U000}
\bar{D}(t)\leq&  C \int_0^t \bar{D}(s)\left( |\ln \bar{D}(s)|+ |\ln s|+\|\nabla\bar u\|_{L^\infty(\R^2)} + \|\bar u \langle x\rangle^{-1}\|_{L^\infty(\R^2)} +1\right)\,ds \\
&\quad+C\|\sqrt{\rho}\du\|_{L^2((0, t)\times\R^2)}.\notag
\end{align}

\medskip

\noindent{\underline{Step 2: duality argument.}}\\
At this stage,  in order to control the difference   $\sqrt{\rho}\du$ in $L^2((0, t)\times \R^2)$, we will estimate the solution $v$ to the  linear backward parabolic system \eqref{eq-backward} as in Subsection \ref{sec.wsucase1}.  Indeed, in our current setting, the solvability of problem \eqref{eq-backward} can be achieved by following the steps in Subsection \ref{s:ffv}. Moreover,  similarly to  estimate \eqref{apriori-bw} and also to \eqref{es-Wvelocity}, \eqref{es-locLinfty}, we have
\begin{multline}\label{apriori-bw1}
\sup_{t\in(0, T)} \|\left(\sqrt{\rho}v,\nabla v\right)(t, \cdot)\|_{L^2(\R^2)}+ \|(\nabla v,  \nabla^2 v) \|_{L^2(0, T; L^2(\R^2))} +\|\bar x^{-\beta} v\|_{L^{2+\beta}(0, T; L^\infty(\R^2))} \\
+\|\bar x^{-\frac{1}{2}-\beta} v\|_{L^{\infty}(0, T; L^3(\R^2))}\leq   C \|\sqrt{\rho}\du\|_{L^2((0, T)\times \R^2)}.
\end{multline}
Now, recall that
\begin{align}\label{es-UU}
\|\sqrt{\rho}\du\|_{L^2((0, T)\times \R^2)}^2\leq\int_0^T |\langle \dr \dot{\bar u}, v\rangle|\,dt + \int_0^T|\langle \rho\du\cdot\nabla\bar u, v\rangle|\,dt.
\end{align}
By H\"{o}lder's inequality
\begin{align*}
&\int_0^T|\langle \rho\du\cdot\nabla\bar u, v\rangle|\,dt\\
 \leq& \|\sqrt{\rho}\du\|_{L^2((0, T)\times \R^2)}\|\nabla \bar u\|_{L^\infty(0, T; L^2(\R^2))}\|\sqrt{\rho}v\|_{L^2(0, T; L^\infty(\R^2))}\\
 \leq& C \|\sqrt{\rho}\du\|_{L^2((0, T)\times \R^2)}\|\nabla \bar u\|_{L^\infty(0, T; L^2(\R^2))} \|\bar x^{-\beta}v\|_{L^2(0, T; L^\infty(\R^2))} (1+ \|\rho \bar x^{2\beta}\|_{L^\infty((0, T\times \R^2)}),
\end{align*}
and
\begin{align*}
\int_0^T |\langle \dr \dot{\bar u}, v\rangle|\,dt &\leq   \|t^{-1/2}\dr  \bar{x}^\beta\|_{L^\infty(0, T; \dot{H}^{-1}(\R^2))} \|\sqrt{t}\nabla (\bar{x}^{-\beta}\dot{\bar u}\cdot v)\|_{L^1(0, T; L^2(\R^2))}\\
&\leq \bar D(T) \|\sqrt{t}\nabla (\bar{x}^{-\beta}\dot{\bar u}\cdot v)\|_{L^1(0, T; L^2(\R^2))}.
\end{align*}
Using inequalities \eqref{es-wf1}, \eqref{es-wf2} and  $\beta< \frac{1}{2}$, we have
\begin{align*}
   | \nabla (\bar{x}^{-\beta}\dot{\bar u}\cdot v)| \leq  C(|\nabla \dot{\bar u}||\bar{x}^{-\beta}v| +  |\dot{\bar u}  \bar{x}^{-\beta}| |\nabla v| + |\dot{\bar u}\bar{x}^{-\beta} | |v \bar{x}^{-\frac{1}{2}-\beta}|),
\end{align*}
which implies that
\begin{align*}
    &\|\sqrt{t}\nabla (\bar{x}^{-\beta}\dot{\bar u}\cdot v)\|_{L^1(0, T; L^2(\R^2))}\\
    \leq & \|\sqrt{t}\nabla\dot{\bar u}\|_{L^2((0, T)\times \R^2)} \|\bar x^{-\beta}v\|_{L^2(0, T; L^\infty(\R^2))}+  \|\sqrt{t} \dot{\bar u}\bar x^{-\beta}\|_{L^2 (0, T; L^6(\R^2))} \|\nabla v\|_{L^2(0, T; L^3(\R^2))}\\
&\quad+ T^\frac{1}{2}\|\sqrt{t} \dot{\bar u}\bar x^{-\beta}\|_{L^2 (0, T; L^6(\R^2))} \|  v \bar x^{-\frac{1}{2}-\beta}\|_{L^\infty(0, T; L^3(\R^2))}
\end{align*}
So by estimate \eqref{apriori-bw1}  we can bound the right-hand-side of \eqref{es-UU} and get
\begin{align*}
\|\sqrt{\rho}\du\|_{L^2((0, T)\times \R^2)}^2\leq  CT^{\frac{\beta}{6}}\big(\bar{D}(T)\|\sqrt{\rho}\du\|_{L^2((0, T)\times \R^2)}+ \|\sqrt{\rho}\du\|_{L^2((0, T)\times \R^2)}^2\big).
\end{align*}
Clearly, the above inequality implies that, if $T$ is small enough then
\begin{align}
\|\sqrt{\rho}\du\|_{L^2((0, T)\times \R^2)}\leq  C \bar{D}(T) T^\frac{\beta}{6}.\label{es-U0}
\end{align}
Putting the above inequality into \eqref{es-U000}, we get for $T$ small enough
\begin{align*}
\bar{D}(T)\leq  C\left(\int_0^T \bar{D}(t)\left(\|\nabla\bar u\|_{L^\infty(\R^2)} + \|\bar u \langle x\rangle^{-1}\|_{L^\infty(\R^2)} + |\ln \bar{D}(t)|- \ln t\right)\,dt \right).
\end{align*}
From Osgood's lemma \cite[Lemma 3.4]{BCD}, we then infer that $\bar D(t)\equiv0$ on $[0, T],$  and thus $\sqrt{\rho}\du\equiv 0,~\dr\equiv0$ on $[0, T]$ thanks to inequality \eqref{es-U0}
and equation \eqref{eq-varrho}, respectively.
Finally,  $\du\equiv 0$ and $\dr\equiv0$ on $[0, \infty)$ can be concluded similarly as previous subsection. In particular, one needs to use the functional inequality \eqref{es-locL21}.
\end{proof}

\subsection{Proof of uniqueness}
To complete the proof of uniqueness part of Theorem \ref{thm2d} and Theorem \ref{thm3d}, it suffices to observe   that all the  assumptions  in Proposition \ref{Prop-uniqueness} and Proposition \ref{Prop-uniqueness1} are satisfied by those solutions constructed in Theorem \ref{thm2d} and Theorem \ref{thm3d}.\qed

\section{Proof of Theorem \ref{C1}} \label{B}

We remark that the general strategy is the same in dimensions $d=2$ and $3$.
Assume that $\partial\Omega_0$ corresponds to the level set $\{f_0=0\}$ of some $\mathcal{C}^{1, \gamma}$ function $f_0: \Omega_0\to \R.$ Then we know that $\partial\Omega_t=X(t, f^{-1}_0(\{0\}))$ corresponds to the level set $\{f_t=0\}$ with $f_t:= f_0\circ (X(t,\cdot))^{-1},$ where   $(X(t,\cdot))^{-1}$ is the inverse function of $X(t,\cdot)$. Indeed, as the flow is incompressible, the Jacobian of $X$ is identically equal to $1$. Hence the classical inverse function theorem ensures the existence and regularity of $ X^{-1}.$

Now, fix some $T>0.$
In the 2D case, according to Theorem \ref{thm2d},   for all $r\in[2, \infty)$ 
we can find  $q>2$ such that
\begin{align*}
    \| \nabla^2 u\|_{L^{ 1} (0, T; L^r(\R^2)}\leq \|\nabla^2( \sqrt{t} u)\|_{L^{q} (0, T; L^r(\R^2))} \|{t}^{-\frac{1}{2}}\|_{L^{q'}(0,T)} \leq C_{0, T}
\end{align*}
and by the Gagliardo-Nirenberg inequality \eqref{GN-in} and Young's inequality
\begin{align*}
    \| \nabla  u\|_{L^{ 1} (0, T; L^r(\R^2)}\leq C(\|\nabla u\|_{L^{1} (0, T; L^2(\R^2))} + \|\Delta u\|_{L^{1}(0, T; L^2(\R^2))} )\leq C_{0, T}.
\end{align*}
By Sobolev's embedding, one has $\nabla u\in L^1(0, T; \mathcal{C}^{0, \gamma})$ for all $\gamma\in (0, 1).$ Consequently, the flow $X(t, \cdot)$ is $\mathcal{C}^{0, \gamma}$ and so is $f_t.$

For the 3D case, Theorem \ref{thm3d} ensures that  for all $r\in[2, 6)$, $\nabla u\in L^1(0, T; L^r(\R^3))$ and we can find  $q>2$ such that $\nabla^2( \sqrt{t}u)\in {L^{ q} (0, T; L^r(\R^3))}.$ Thus $\nabla u\in L^1(0, T; W^{1, r}(\R^3)).$ This finally implies that $f_t$ is a $\mathcal{C}^{0, \gamma}$ function, if $\gamma<\frac{1}{2}.$\qed

%%%%%%%%%%%%%%%%%%%%%%%%%%%%%%%%%%%%%%%%%%%%%%%%%
\begin{appendices}
\section{Functional spaces and  inequalities}\label{A}
For the reader's convenience, we here recall a few results
that are used repeatedly in the paper. Let us first recall
the definitions of homogeneous Sobolev spaces.
\begin{defi}[homogeneous Sobolev space]
Let $s$ be in $\mathbb{R}.$ The homogeneous Sobolev space $\dot{H}^s(\mathbb{R}^d)$ is the set of tempered distributions $u$ on $\mathbb{R}^d,$ with  Fourier transform in  $L^1_{loc}(\mathbb{R}^d),$  satisfying
$$\|u\|_{\dot{H}^s(\R^d)}:=\||\xi|^s\mathcal{F}(u)(\xi)\|_{L^2(\R^d)}<\infty.$$
\end{defi}

We often use the following  Gagliardo-Nirenberg inequalities.
\begin{prop}[Gagliardo-Nirenberg inequalities; {\cite[Lemma II.3.3]{Galdibook}}]\label{GN}
   If $(q,r)\in (1,\infty)^2$, there exists a constant $C$ depending on $q$ and $r$ such that
\begin{equation}\label{GN-in}
  \|z\|_{L^p(\R^d)}\leq C\|\nabla z\|_{L^r(\R^d)}^{ \theta}\|z\|_{L^q(\R^d)}^{1-\theta}
\end{equation}
  with $\frac{1}{p}= \theta\left(\frac{1}{r}-\frac{1}{d}\right)+\frac{1-\theta}{q}, ~0\leq \theta\leq1.$
\end{prop}

Let us now recall some   properties of the  spaces $\wt{D}^{1, 2}(\R^2)$ and $\mathcal{D}^{1, 2}(\R^3)$ that are used in the paper.  For more details, see Appendix A and B of \cite{PLL}.
The first one is  the  following weighted estimate for elements of the  space $\wt{D}^{1, 2}(\R^2).$
\begin{lem}[{\cite[Theorem B.1]{PLL}}]\label{Le-H11}
 For $m\in [2, \infty)$ and $\ell\in (1+\frac m2, \infty),$ there exists a positive constant $C$ depending on $m$ and $l$ such that for all $z\in \wt{D}^{1, 2}(\R^2),$
 \begin{align}
 \left(\int_{\R^2} \frac{|z|^m}{ {\langle x \rangle}^2}\, (\log \langle x \rangle )^{-\ell}\,dx\right)^{1/m}\leq C \left( \|z\|_{L^2(B_1)}+\|\nabla z\|_{L^2(\R^2)}\right).
 \end{align}
\end{lem}

Then one has the following Proposition due to  Li and Xin \cite{LX},  which is a combination of Lemma \ref{Le-H11} with the Poincar\'{e} inequality; see \cite{LX}, (2.6) and (2.8) in the proof of Lemma 2.4 therein.
\begin{prop}[{\cite[Lemma 2.4]{LX}}]\label{Le-Li-Xin}
Let $m\in [2, \infty)$ and $\ell\in (1+\frac m2, \infty)$. Let $z\in \wt{D}^{1, 2}(\R^2)$ and $\eta\in L^\infty(\R^2)$. Assume that $\sqrt{\eta }z\in L^2(\R^2)$ and satisfies
 \begin{align*}
 0\leq \eta\leq \eta^*,\quad M\leq \int_{B_{R}} \eta\,dx,
 \end{align*}
for   positive constants $\eta^*, M, R.$ Then there exists a positive constant $C$ depending only on $m, \ell, \eta^*, M, R$ such that
\begin{equation}\label{e.estprop14}
\left(\int_{\R^2} \frac{|z|^m}{ {\langle x \rangle}^2}\, (\log \langle x \rangle )^{-\ell}\,dx\right)^{1/m}\leq C(\|\sqrt{\eta} z\|_{L^2(\R^2)}+\|\nabla z\|_{L^2(\R^2)}).
\end{equation}
\end{prop}

We also needed the following conditional $L^2$ bound  for elements of the  space $\wt{D}^{1, 2}(\R^2).$
\begin{lem}[{\cite[Lemma B.1 and Remark B.1]{PLL}}]\label{Le-H1}
$\wt{D}^{1, 2}(\R^2)$ is a Hilbert space for the scalar product $\langle\nabla z, \nabla w\rangle+\langle\mathbf{1}_{B_1}\,z, \,w\rangle$ and an equivalent norm is given by $\|\nabla z\|_{L^2(\R^2)}+|\int_{B_1} z\,dx|.$ Moreover, we have $\wt{D}^{1, 2}(\R^2)\cap (L^1(\R^2)+L^2(\R^2))\hookrightarrow H^1(\R^2),$ more precisely, there exists a constant $C>0$ such that for all $z\in\wt{D}^{1, 2}(\R^2)$ satisfying $z=z_1+z_2$ with $z_1\in L^1(\R^2), \,z_2\in L^2(\R^2)$
\begin{equation}
\|z\|_{L^2(\R^2)}\leq C\left(\|z_1\|_{L^1(\R^2)}^\frac{1}{2}\|\nabla z\|_{L^2(\R^2)}^\frac{1}{2}+\|z_2\|_{L^2(\R^2)}\right).
\end{equation}
A similar result holds for all $z\in \mathcal{D}^{1, 2}({\R^3}).$ There exists a constant $C>0$ such that for all $z\in\mathcal{D}^{1, 2}(\R^3)$ satisfying $z=z_1+z_2$ with $z_1\in L^1(\R^3), \,z_2\in L^2(\R^3)$
\begin{equation}
\|z\|_{L^2(\R^3)}\leq C\left(\|z_1\|_{L^1(\R^3)}^\frac{2}{5}\|\nabla z\|_{L^2(\R^3)}^\frac{3}{5}+\|z_2\|_{L^2(\R^3)}\right).
\end{equation}
\end{lem}

Then based on Lemma \ref{Le-H1},  we have
\begin{prop}[interpolation estimate]\label{Prop-intp}
Let $d=2,\ 3$. For all $z\in \wt{D}^{1, 2}(\R^2)$ or $\mathcal{D}^{1, 2}(\R^3)$ and all non-negative function $\eta\in L^\infty(\R^d)$ that satisfies  either
\begin{equation}\label{condf1}
(1/\eta)\,\mathbf{1}_{\eta<\delta_0}\in L^1(\R^d),\quad{\rm for ~some}~\delta_0>0,
\end{equation}
or
\begin{equation}\label{condf2}
(\underline{\eta}-\eta)_+\in L^p(\R^d), \quad{\rm for ~some}~\underline{\eta}\in (0, \infty), \,\, p\in (d/2, \infty),
\end{equation}
there exists a constant $C_*$ depending on $d$, $\delta_0$ and $\|( {1}/{{\eta}})\,\mathbf{1}_{\eta< \delta_0}\|_{L^1}$ in the case when \eqref{condf1} is satisfied, $p$, $d$, $\bar\eta$ and $\|(\bar\eta-\eta)_+\|_{L^p}$ in the case when \eqref{condf2} is satisfied, 
such that
\begin{equation}\label{intp-ineq}
\|z\|_{L^2(\R^d)}\leq C_*(\|\sqrt{\eta}z\|_{L^2(\R^d)}+\|\nabla z\|_{L^2(\R^d)}).
\end{equation}
\end{prop}
\begin{proof}
We first consider $\eta$ satisfying \eqref{condf2}.  It is easy to check that
\begin{align*}
\sqrt{\underline{\eta}}\leq \sqrt{(\underline{\eta}-\eta)_+}+\sqrt{\eta}.
\end{align*}
We then write by H\"{o}lder's inequality,  the  Gagliardo-Nirenberg inequality in Proposition \ref{GN} and Young's inequality that
 \begin{align*}
\sqrt{\underline{\eta}}\|z\|_{L^2(\R^d)}\leq& \|\sqrt{(\underline{\eta}-\eta)_+}\, z\|_{L^2(\R^d)}+\|\sqrt{\eta} z\|_{L^2(\R^d)}\\
\leq& \| (\underline{\eta}-\eta)_+ \|_{L^p(\R^d)}^\frac{1}{2}\|z\|_{L^{\frac{2p}{p-1}}}+\|\sqrt{\eta} z\|_{L^2(\R^d)} \\
\leq&C \|(\underline{\eta}-\eta)_+\|_{L^p(\R^d)}^\frac{1}{2}\|z\|_{L^2(\R^d)}^{1-\frac{d}{2p}}\|\nabla z\|_{L^2(\R^d)}^{\frac{d}{2p}}+\|\sqrt{\eta} z\|_{L^2(\R^d)}\\
\leq&C \|(\underline{\eta}-\eta)_+\|_{L^p(\R^d)}^\frac{d}{p} \|\nabla z\|_{L^2(\R^d)}+\frac{\sqrt{\underline{\eta}}}{2}\| z\|_{L^2(\R^d)}  +  \|\sqrt{\eta} z\|_{L^2(\R^d)},
\end{align*}
which enables us to obtain \eqref{intp-ineq}.

If \eqref{condf1} is satisfied, we decompose
\begin{align*}
z=z\,\mathbf{1}_{\eta<\delta_0}+z\,\mathbf{1}_{\eta\geq \delta_0}
\end{align*}
and write that
\begin{align*}
 \|z \,\mathbf{1}_{\eta\geq \delta_0}\|_{L^2(\R^d)}&\leq\frac{1}{\sqrt{\delta_0}}\|\sqrt{\eta }z\,\mathbf{1}_{\eta\geq \delta_0}\|_{L^2(\R^d)}\leq \frac{1}{\sqrt{\delta_0}}\|\sqrt{\eta }z\|_{L^2(\R^d)},\\
 \|z \,\mathbf{1}_{\eta< \delta_0} \|_{L^1(\R^d)}&\leq \|\sqrt{\eta }z\|_{L^2(\R^d)}\| ({1}/{\sqrt{\eta}})\,\mathbf{1}_{\eta< \delta_0}\|_{L^2(\R^d)}\\
&\leq \|\sqrt{\eta }z\|_{L^2(\R^d)}\|( {1}/{{\eta}})\,\mathbf{1}_{\eta< \delta_0}\|_{L^1(\R^d)}^\frac{1}{2}.
\end{align*} 
We complete the proof by using Lemma \ref{Le-H1} and Young's inequality.\qedhere
\end{proof}

\end{appendices}
 %%%%%%%%%%%%%%%%%%%%%%%%%%%%%%%%%%%%%%%%%%%%%%%%

\section*{Acknowledgment}
CP and JT are partially supported by the Agence Nationale de la Recherche,
project BORDS, grant ANR-16-CE40-0027-01 and by the CY
Initiative of Excellence, project CYNA (CY Nonlinear Analysis). CP is also partially supported by the Agence Nationale de la Recherche,
 project SINGFLOWS, grant ANR-18-CE40-0027-01, project CRISIS, grant ANR-20-CE40-0020-01, by the CY Initiative of Excellence, project CYFI (CYngular Fluids and Interfaces). JT is also supported by the Labex MME-DII.  JT  thanks Rapha\"{e}l Danchin for his remarkable suggestions. Both authors also thank Marcel Zodji for his remarks on an earlier version of the paper.

\section*{Data availability statement}

Data sharing is not applicable to this article as no datasets were generated or analyzed during the current study.

\section*{Conflict of interest}

The authors declare that they have no conflict of interest.

%%%%%%%%%%%%%%%%%%%%%%%%%%%%%%%%%%%%%%%%%%%%


\begin{thebibliography}{99}
\bibitem{A} H. Abidi: Equation de Navier-Stokes avec densit\'{e} et viscosit\'{e} variables dans l’espace critique, {\em Rev. Mat. Iberoam}, {\bf 23} (2007), no. 2, 537--586.

\bibitem{AG} H. Abidi and G. Gui: Global Well-posedness for the 2-D inhomogeneous incompressible Navier-Stokes System with large initial data in critical Spaces, {\em Archiv. Rat. Mech. Anal.}, {\bf 242}, (2021), 1533--1570.

\bibitem{AP} H. Abidi and M. Paicu: Existence globale pour un fluide inhomog\`{e}ne,  {\em Ann. Inst. Fourier}, {\bf 57}
(2007), no. 3, 883--917.


\bibitem{AGZ} H. Abidi, G. Gui, and P. Zhang: On the wellposedness of three-dimensional inhomogeneous Navier-Stokes equations in the critical spaces, {\em Arch. Rational Mech. Anal.}, {\bf 204} (2012),  189--230.



\bibitem{AKM90} S. Antontsev,  A. Kazhikhov and V.  Monakhov:  {\em Boundary value problems in mechanics of nonhomogeneous fluids.}  Studies in Mathematics and Its Applications, {\bf 22},  North-Holland, Amsterdam, 1990.



\bibitem{BCD}  {H. Bahouri, J. Y. Chemin and R. Danchin}:  { Fourier Analysis and Nonlinear Partial Differential Equations}, {\bf 343}, Grundlehren der Mathematischen Wissenschaften, Springer 2011.

 \bibitem{Cosmin} C. Burtea:  Optimal well-posedness for the inhomogeneous incompressible Navier-Stokes system with
general viscosity, {\em Anal. PDE}, {\bf 10} (2017), 439--479.

 \bibitem{CCFGG} A. Castro, D. C\'{o}rdoba, C. Fefferman,  F. Gancedo  and J. G\'omez-Serrano:  Splash Singularities for the Free Boundary Navier-Stokes Equations, {\em Ann. PDE}, {\bf 5}: 12, (2019).




\bibitem{CK} Y. Cho and H. Kim: Unique solvability for the density-dependent Navier-Stokes equations, {\em Nonlinear Anal.}, {\bf 59}(4), (2004), 465--489.

\bibitem{CK2} H. J. Choe and H. Kim:  Strong solutions of the Navier-Stokes equations for nonhomogeneous incompressible fluids, {\em  Comm. Partial Differential Equations}, {\bf 28} (2003), 1183--1201.

 \bibitem{CHW13} W. Craig,  X. Huang and Y. Wang: Global well-posedness for the 3D inhomogeneous incompressible Navier-Stokes equations, {\em J. Math. Fluid Mech.}, {\bf 15}(4), (2013),  747--758.


 \bibitem{Dan1} R. Danchin:  Density-dependent incompressible viscous fluids in critical spaces, {\em Proc. Roy. Soc. Edinburgh Sect. A}, {\bf 133}, (2003),  1311--1334.

 \bibitem{Dan2} R. Danchin: Local and global well-posedness results for flows of inhomogeneous viscous fluids,
 {\em Adv. Diff. Eq.}, {\bf 9} (2004), 353--386.


 \bibitem{DM12} R. Danchin and P.B. Mucha: A Lagrangian Approach for the Incompressible Navier-Stokes Equations with Variable Density, {\bf 65}, (2012), 1458--1480.

\bibitem{DM13} R. Danchin and P.B. Mucha: Incompressible flows with piecewise constant density,
{\em Arch. Ration. Mech. Anal.}, {\bf 207},  (2013), 991--1023.


\bibitem{DM19} R. Danchin and P.B. Mucha:   The incompressible Navier-Stokes equations in vacuum, {\em Comm. Pure Appl. Math.},
72   (2019), 1351--1385.

\bibitem{DM23} R. Danchin and P.B. Mucha: Compressible Navier-Stokes equations with ripped density, {\em Comm. Pure Appl. Math.}, to appear.

\bibitem{DMP} R. Danchin, P.B. Mucha and T. Piasecki: Stability of the density patches problem with vacuum for incompressible inhomogeneous viscous flows, {\em Ann. Inst. H. Poincar\'{e} Anal. Non Lin\'{e}aire},  (2023).

\bibitem{DPZ} R.~Danchin and P.~ Zhang: Inhomogeneous Navier-Stokes equations in the half-space,
with only bounded density,  {\em J. Funct. Anal.}, {\bf 267}, (2014), 2371--2436.


\bibitem{DZ} R. Danchin and X. Zhang: On the persistence of H\"older regular patches of density for the inhomogeneous Navier-Stokes equations,  {\em Journal de l'Ecole Polytechnique}, {\bf 4}, (2017), 781--811.

\bibitem{DW} R. Danchin and S. Wang: Global Unique Solutions for the Inhomogeneous Navier-Stokes equations with only Bounded Density, in Critical Regularity Spaces, {\em  Commun. Math. Phys.},  (2022).


\bibitem{BD97} B. Desjardins:  Global existence results for the incompressible density-dependent Navier-Stokes equations in the whole space, {\em Differential Integral Equations}, {\bf 10}(3),  (1997),  587--598.

\bibitem{BD97-2} B. Desjardins:  Regularity results for two-dimensional flows of multiphase viscous fluids, {\em  Arch. Rational Mech. Anal.},
{\bf 137}(2), (1997), 135--158.

\bibitem{BD97-3} B. Desjardins: Regularity of weak solutions of the compressible isentropic Navier-Stokes equations, {\em Comm. Partial Differential Equations}, {\bf 22}(5-6),  (1997),  977--1008.

\bibitem{DL} R.J. DiPerna and P.L. Lions, Ordinary differential equations, transport theory and Sobolev
spaces, {\em  Invent. Math.}, {\bf 98}, (1989), 511--547.

\bibitem{Galdibook}  G. P. Galdi: {\em An introduction to the mathematical theory of the
              {N}avier-{S}tokes equations} (second edition), Springer Monographs in Mathematics, 2011.

\bibitem{GGJ18}    F. Gancedo and  E. Garcia-Juarez:
    Global regularity of 2D density patches for inhomogeneous Navier-Stokes, {\em Arch. Ration. Mech. Anal.}, {\bf 229}, (2018), 339–-360.



\bibitem{GGJ23}    F. Gancedo and  E. Garcia-Juarez:   Global regularity of 2D Navier–Stokes free boundary with small viscosity contrast, {\em  Ann. Inst. H. Poincaré Anal. Non Linéaire}, (2023).



\bibitem{Germain} P. Germain: Strong solutions and weak-strong uniqueness for the nonhomogeneous
Navier-Stokes equation, {\em J. Anal. Math.}, {\bf 105} (2008),  169--196.


\bibitem{Hoff} D. Hoff: Uniqueness of weak solutions of the Navier-Stokes equations of multidimensional, compressible flow, {\em SIAM J. Math. Anal.}, {\bf 37}(6), (2006), 1742--1760.


\bibitem{HPZ} J.~Huang, M.~Paicu and P.~Zhang:
 Global well-posedness of incompressible inhomogeneous fluid systems   with bounded density or non-Lipschitz velocity,
   {\em Arch. Ration. Mech. Anal.}, {\bf 209}(2), (2013), 631--682.



\bibitem{Ka74} A. Kazhikhov:  Solvability of the initial-boundary value problem for the equations of the motion of an inhomogeneous viscous incompressible fluid,
{\em Dokl. Akad. Nauk SSSR}, {\bf 216},  (1974), 1008--1010.

\bibitem{Lady} O.  Ladyzhenskaya:  Solution `in the large' of the non-stationary boundary value problem for the Navier-Stokes system with two space variables, {\em Comm. Pure Appl. Math.}, {\bf 12}, (1959) 427--433.

\bibitem{LS} O. Ladyzhenskaya and V. Solonnikov: Unique solvability of an initial and boundary value
problem for viscous incompressible inhomogeneous fluids. {\em J. Sov. Math.}, {\bf 9}(5), (1978),
 697--749.


 \bibitem{Leray} J. Leray: Sur le mouvement d'un liquide visqueux remplissant l'espace, {\em Acta
Mathematica}, {\bf 63} (1934), 193--248.



  \bibitem{Li} J. Li: Local existence and uniqueness of strong solutions to the Navier-Stokes equations with nonnegative density,   {\em J. Differential Equations},  {\bf 263}(10), (2017), 6512--6536.

   \bibitem{LX} J. Li and Z. P. Xin : Global Well-Posedness and Large Time Asymptotic Behavior of Classical Solutions to the Compressible Navier–Stokes Equations with Vacuum, {\em Ann. PDE}, (2019) 5:7.

 \bibitem{LL1} J. Li and Z.  Liang:  On local classical solutions to the Cauchy problem of the two-dimensional barotropic compressible Navier–Stokes equations with vacuum, {\em J. Math. Pures Appl.},  {\bf 102} (2014), 64--671.

 \bibitem{LZ1}  X. Liao and P. Zhang: On the global regularity of 2D density patch for inhomogeneous incompressible viscous flow, {\em  Arch. Ration. Mech. Anal.}, {\bf 220}(2),  (2016), 937--981.

  \bibitem{LZ2}  X. Liao and P. Zhang:
Global regularity of 2D density patches for viscous inhomogeneous incompressible flow with general density: Low regularity case, {\em Comm. Pure Appl. Math.},  {\bf 72},   (2019), 835--884.

 \bibitem{LZ3}  X. Liao and P. Zhang: Global regularity of 2-D density patches for viscous inhomogeneous incompressible flow with general density: High regularity case, {\em Anal. Theory Appl.}, {\bf 35}, (2019),  163--191.

\bibitem{PLL}  P.-L. Lions: {\em Mathematical topics in fluid
mechanics. Incompressible models}, Oxford Lecture Series in Mathematics and its Applications, {\bf 3}, 1996.

 \bibitem{LSZ}  B. L\"{u}, X. Shi and X. Zhong: {\em  Global existence and large time asymptotic behavior of strong solutions to the Cauchy problem of 2D density-dependent Navier-Stokes equations with vacuum},  {\em Nonlinearity}, {\bf 31},  (2018),  2617--2632.


\bibitem{MR} P. B. Mucha and W. M. Rusin:    Zygmund Spaces, Inviscid Limit and Uniqueness of Euler Flows, {\em Commun. Math. Phys.}, {\bf 280}, (2008), 831--841.


\bibitem{PZ} M. Paicu and P. Zhang:  Striated Regularity of 2-D Inhomogeneous Incompressible Navier–Stokes System with Variable Viscosity, {\em  Commun. Math. Phys.},  {\bf 376}, (2020), 385-439.

\bibitem{PZZ} M. Paicu, P. Zhang and Z. Zhang: Global Unique Solvability of Inhomogeneous Navier-Stokes Equations with Bounded Density, {\em Comm. Partial Differential Equations}, {\bf 38}, (2013), 1208--1234.

\bibitem{Si} J.~Simon:  Nonhomogeneous viscous incompressible fluids: existence of velocity,
  density, and pressure,  {\em SIAM J. Math. Anal.}, {\bf 21}(5), (1990), 1093--1117.


\bibitem{XLZ} H. Xu, Y. Li, and X. Zhai: On the well-posedness of 2-D incompressible Navier–Stokes equations with variable
viscosity in critical spaces,  {\em J. Differential Equations}, {\bf 260}(8) (2016), 6604--6637.


\bibitem{Z} P. Zhang, Global Fujita-Kato solution of 3-D inhomogeneous incompressible Navier-Stokes system, {\em  Adv.  Math.}, {\bf 363}, (107007), (2020).

\end{thebibliography}
\end{document}